\documentclass[review,hidelinks,onefignum,onetabnum]{siamart220329}

\nolinenumbers
\usepackage{lipsum}
\usepackage{amsfonts}
\usepackage{graphicx}
\usepackage{epstopdf}
\usepackage{algorithmic}
\usepackage{subcaption}
\usepackage{wrapfig}

\ifpdf
  \DeclareGraphicsExtensions{.eps,.pdf,.png,.jpg}
\else
  \DeclareGraphicsExtensions{.eps}
\fi

\newsiamremark{remark}{Remark}
\newsiamremark{assumption}{Assumption}
\newsiamremark{hypothesis}{Hypothesis}
\crefname{hypothesis}{Hypothesis}{Hypotheses}
\newsiamthm{claim}{Claim}

\headers{Constrained Feedback Stackelberg Equilibria}{
Li, Sojoudi, Tomlin, and Fridovich-Keil
}

\title{The Computation of Approximate Feedback Stackelberg Equilibria in Multi-player Nonlinear Constrained Dynamic Games\thanks{
This is a preprint manuscript. 
\funding{This work was supported by the DARPA Assured Autonomy and ANSR programs, the
NASA ULI program in Safe Aviation Autonomy, and the ONR Basic Research Challenge in Multibody Control Systems. D. Fridovich-Keil was supported by the National Science Foundation under Grant No. 2211548.}}}

\author{
Jingqi Li\thanks{Electrical Engineering and Computer Sciences, University of California, Berkeley, Berkeley, CA
94720 USA (\email{jingqili@berkeley.edu}, (765) 337-0678, \email{sojoudi@berkeley.edu}, (510) 643-6610, \email{tomlin@berkeley.edu}, (626) 395-3813)} \and
Somayeh Sojoudi\footnotemark[2] \and
Claire Tomlin\footnotemark[2]\and 
David Fridovich-Keil\thanks{Aerospace Engineering and Engineering Mechanics, University of Texas at Austin, Austin, TX
78712-1221 USA (\email{dfk@utexas.edu}, (512) 471-4257)}
}

\usepackage{amsopn}

\newcommand{\zTi}{{\mathbf{z}_{T}^{i*}}}
\newcommand{\zTrhoistar}{{\mathbf{z}_{T,\rho}^{i*}}}
\newcommand{\zTNstar}{{\mathbf{z}_{T}^{N*}}}
\newcommand{\zTrhoNstar}{{\mathbf{z}_{T,\rho}^{N*}}}
\newcommand{\compactdomainforgradientanalysis}{{\mathcal{S}}}
\newcommand{\linearoperator}{{\hat{\mathcal{M}}}}
\newcommand{\matrixinverseoperator}{{\mathcal{M}}}

\newcommand{\operator}{{\mathcal{A}}}

\newcommand{\nablastar}{{\nabla^*}}
\newcommand{\alphaupper}{{\hat{\alpha}}}
\newcommand{\FSE}{FSE}
\newcommand{\quasi}{quasi-policy}
\newcommand{\quasipolicies}{quasi-policies}
\newcommand{\quasigradients}{quasi-policy gradients}
\newcommand{\quasigradient}{quasi-policy gradient}

\newcommand{\xvec}{{\mathbf x}}
\newcommand{\uvec}{{\mathbf u}}
\newcommand{\gammavec}{{\boldsymbol \gamma}}
\newcommand{\muvec}{{\boldsymbol \mu}}
\newcommand{\lambdavec}{{\boldsymbol \lambda}}
\newcommand{\svec}{{\mathbf s}}
\newcommand{\psivec}{{\boldsymbol \psi}}
\newcommand{\etavec}{{\boldsymbol \eta}}

\newcommand{\zvec}{{\mathbf z}} % total shared states
\newcommand{\zvectilde}{\tilde{\mathbf{z}}}
\newcommand{\zvecrho}{\mathbf{z}_{\rho}}
\newcommand{\zvecrhostar}{{\mathbf{z}_\rho^*}}

\newcommand{\yvec}{{\mathbf y}} % total shared states of player i
\newcommand{\F}{{\mathcal F}}
\newcommand{\pitrhoi}{{\pi_{t,\rho}^i}}
\newcommand{\piTrhoi}{{\pi_{T,\rho}^i}}
\newcommand{\piTrhoN}{{\pi_{T,\rho}^N}}

\newcommand{\piTrhoinext}{{\pi_{T,\rho}^{i+1}}}

\newcommand{\Ktirhostar}{K_{t,\rho}^{i*}}

\newcommand{\localpitrhoi}{{\tilde{\pi}_{t,\rho}^i}}
\newcommand{\localpiTrhoi}{{\tilde{\pi}_{T,\rho}^i}}
\newcommand{\localpiTrhoN}{{\tilde{\pi}_{T,\rho}^N}}
\newcommand{\localpitrhoN}{{\tilde{\pi}_{t,\rho}^N}}

\newcommand{\localpiTrho}{\tilde{\pi}_{T,\rho}}

\newcommand{\KTNrho}{K_{T,\rho}^N}
\newcommand{\KTirho}{K_{T,\rho}^i}
\newcommand{\KtNrho}{K_{t,\rho}^N}
\newcommand{\Ktirho}{K_{t,\rho}^i}
\newcommand{\KTnextirho}{K_{T,\rho}^{i+1}}
\newcommand{\hatKTirho}{\hat{K}_{T,\rho}^i}

\newcommand{\Knexttonerho}{K_{t+1,\rho}^1}
\newcommand{\Krho}{K_\rho}
\newcommand{\Ktnextirho}{K_{t,\rho}^{i+1}}
\newcommand{\Kzeroonerho}{K_{0,\rho}^1}

\newcommand{\knownxt}{\bar{x}_t}
\newcommand{\knownutoneiminusone}{\bar{u}_t^{1:i-1}}

\newcommand{\piti}{\pi_t^{i}}
\newcommand{\pitj}{\pi_t^{j}}
\newcommand{\pitauj}{\pi_\tau^{j}}

\ifpdf
\hypersetup{
  pdftitle={On the Computation of Feedback Stackelberg Equilibria},
  pdfauthor={Jingqi Li, Somayeh Sojoudi, Claire Tomlin, David Fridovich-Keil}
}
\fi

\begin{document}

\maketitle

% REQUIRED
\begin{abstract}
% Stackelberg dynamic games are an important game-theoretic decision model where players make decision hierarchically. When players share coupled constraints, it is challenging to solve feedback Stackelberg games due to the feedback interactions among players. 
Solving feedback Stackelberg games with nonlinear dynamics and coupled constraints, a common scenario in practice, presents significant challenges. This work introduces an efficient method for computing approximate local feedback Stackelberg equilibria in multi-player general-sum dynamic games, with continuous state and action spaces. Different from existing (approximate) dynamic programming solutions that are primarily designed for unconstrained problems, our approach involves reformulating a feedback Stackelberg dynamic game into a sequence of nested optimization problems, enabling the derivation of Karush–Kuhn–Tucker (KKT) conditions and the establishment of a second-order sufficient condition for local feedback Stackelberg equilibria. We propose a Newton-style primal-dual interior point method for solving constrained linear quadratic (LQ) feedback Stackelberg games, offering provable convergence guarantees. Our method is further extended to compute local feedback Stackelberg equilibria for more general nonlinear games by iteratively approximating them using LQ games, ensuring that their KKT conditions are locally aligned with those of the original nonlinear games. %This is the first iterative LQ game approximation approach for constrained feedback Stackelberg games that has provable exponential convergence guarantees. %
We prove the exponential convergence of our algorithm in constrained nonlinear games. %, under certain conditions of these iterative LQ game approximations. 
In a feedback Stackelberg game with nonlinear dynamics and (nonconvex) coupled costs and constraints, our experimental results reveal the algorithm's ability to handle infeasible initial conditions and achieve exponential convergence towards an approximate local feedback Stackelberg equilibrium.
\end{abstract}

% REQUIRED

% REQUIRED

\section{Introduction}
% {\color{red} Simultaneous vs. sequential. example, more appropriate. Blinker, car leader. Not just a solver. But this paper hopes to bring some new perspective about the feedback Nash and Stackelberg equilibria in dynamic constrained games. 

% ST $\rightarrow$ Feedback ST

% Building upon Forrest's paper, 

% Active set method has exponential complexity, but PDIP has polynomial complexity. 
% active set method, scaling issue. Require feasible solution. PDIP polynomial complexity. 

% }

% [Game theory, general, what is Stackelberg games? why do people study it?]
Dynamic game theory \cite{basar1999dynamic} provides tools for analyzing strategic interactions in multi-agent systems%\cite{li2017review}
. It has broad applications in control \cite{chen1998game}, biology \cite{lavalle1993game}, and economics \cite{he2007survey}. 
A well-known equilibrium concept in dynamic game theory is the \emph{Nash equilibrium} \cite{nash1951non}, where players %, having different objectives, 
pursue strategies that are unilaterally optimal, and %an equilibrium where no player wants to deviate from due to the high induced cost. 
% [Nash vs. Stackelberg]
%In a Nash game, %Nash equilibrium assumes that 
players make decisions simultaneously. However, this may not apply to a broad class of games where a decision hierarchy exists, such as lane-merging in highway driving \cite{yoo2013stackelberg}, predator-prey competition in biology \cite{bakule1987structural}, and retail markets in economics \cite{li2017review}. These games could be more naturally formulated as \emph{Stackelberg games} \cite{von1952theory}, %papavassilopoulos1979nonclassical
where players act sequentially in a predefined order. For such games, the \emph{Stackelberg equilibrium}  is the appropriate equilibrium concept.%, as the Nash equilibrium does not consider decision hierarchy. %in the former players making decision first need to account how other players react, but the latter does not consider this decision hierarchy. %For those problems, \emph{Stackelberg equilibrium} is a more appropriate concept than Nash equilibrium because in the former players who make decision first need to consider their followers' reaction. Stackelberg equilibrium is an equilibrium concept in Stackelberg games, where no player wants to deviate from due to the high induced cost. %The reason is because, in Stackelberg equilibrium, players are assumed to take actions following a predefined order and they need to account for how their followers react when they make decision. Players pursue an equilibrium such that any unilateral deviation from that equilibrium leads to a high cost for that player. %There are several reasons why Stackelberg equilibrium may be more favorable than Nash equilibria. For example, at the intersection of roads, pedestrian could first decide whether go across the intersection, and then the vehicle would act accordingly to avoid collision against the pedestrian. This sequential decision nature of previous examples is different from Nash equilibria, where players make decisions simultaneously. 

% [open-loop vs. feedback] 
%The precise 
The formulation of Stackelberg equilibria depends on the information structure \cite{basar1999dynamic}. For instance, in scenarios where players lack access to the current game state, one can compute an \emph{open-loop Stackelberg equilibrium} (OLSE). At such an equilibrium, players' decisions depend on the initial state of a game and followers’ decisions are influenced by the leaders'. %An open-loop Stackellberg equilibrium is not sub-game perfect, meaning that the decisions are not optimal for future stages if the state is perturbed at an intermediate stage. 
When players also have access to state information and their prior players' actions, it becomes appropriate to compute a \emph{feedback Stackelberg equilibrium} (FSE), where each player’s decision is contingent upon the current state and the actions of preceding players. %Conversely, when players remember current and past states, it becomes appropriate to compute a \emph{closed-loop Stackelberg equilibrium}. A notable class of closed-loop Stackelberg equilibria is \emph{feedback Stackelberg equilibria} (FSE), where each player's decision is contingent only on the current state, instead of all past states. 
One advantage of FSE over OLSE is its sub-game perfection, meaning that decision policies remain optimal for future stages, even if the state is perturbed at an intermediate stage. This feature is particularly beneficial in scenarios with feedback interactions among players, such as in lane merging during highway driving \cite{talebpour2015modeling} and human-robot interactions \cite{franceschi2023human}. In these situations, the sub-game perfection of FSE makes it a more suitable equilibrium concept than OLSE, as it allows players to adjust their decisions based on the current state information.
%An advantage of FSE over OLSE is the sub-game perfection, i.e., the decision policies remain optimal for future stages even if the state undergoes perturbation at some intermediate stage. {\color{red}%This robustness to accidental perturbations makes FSE more favorable than OLSE in practical applications.
%When there are feedback interactions among players, such as in lane merging during highway driving \cite{talebpour2015modeling} and human-robot interaction \cite{franceschi2023human}, %, such as those problems in economics and traffic control \cite{rubio2006coincidence,wie2007dynamic}, 
%the sub-game perfection makes FSE a more appropriate equilibrium concept than OLSE because players could adjust their decisions based on the current state information.} %Hence, FSE is preferable when ensuring sub-game perfectness is a priority.% if we want to ensure sub-game perfectness. %we prefer feedback Stackelberg equilibria when we want to ensure the sub-game perfectness. %However, open-loop Stackelberg equilibrium is not sub-game perfect. % and therefore it is robust to state perturbation. %This formulation encodes more feedback interactions among players than its open-loop counterpart, and the decision under this formulation is generally robust to perturbations to the states that may occur when playing games. 

Though FSE is conceptually appealing, computing it poses significant challenges \cite{ho1981information,weintraub1983existence,xie1997time,martin2021coincidence}. %with broad applications in robotics, autonomous driving and network security, 
%In this paper, we focus on the \emph{pure feedback Stackelberg equilibrium} (PFSE), where a player's optimal decision at each stage is a deterministic function of the current state and prior players' actions. 
Previous research has extensively explored the \FSE{} problem in finite dynamic games, characterized by a finite number of states and actions \cite{simaan1973additional,basar1999dynamic,tolwinski1983stackelberg,korzhyk2011stackelberg,bai2021sample,wang2022coordinating}. %bai2021sample,gerstgrasser2023oracles
%, where we have a finite number of states and actions. 
In contrast, infinite dynamic games--those with an infinite number of states and actions--%i.e., games with an infinite number of states and actions%, a naive application of results in finite dynamic games requires gridding the continuous state and action space, which suffers from curse of dimensionality. 
have mostly been considered within the framework of linear quadratic (LQ) games, featuring linear dynamics and stage-wise quadratic costs %most of the prior works considered linear quadratic (LQ) games 
\cite{gardner1978feedback,basar1999dynamic,dockner2000differential,vamvoudakis2012online,jungers2014feedback}. %, where we have linear dynamics and stage-wise quadratic costs. 
The computation of \FSE{} for more general nonlinear games is more challenging than for LQ games. A naive application of existing dynamic programming solutions in finite dynamic games necessitates gridding the continuous state and action spaces, often leading to computational intractability \cite{bellman1956dynamic}. % and has not been sufficiently explored. 
Recent works \cite{mylvaganam2014approximate, xu2020adaptive} have proposed using approximate dynamic programming to compute an approximate \FSE{} for input-affine systems. Additionally, several iterative linear-quadratic (LQ) approximation approaches have been proposed in \cite{hu2023emergent,khan2023leadership}, but they lack convergence guarantees.% to compute the approximate \FSE. 

%However, %to the best of the authors’ knowledge, 
Moreover, existing approaches are ill-suited for handling coupled %generalized
equality and inequality constraints on players’ states and decisions, which frequently arise in safety-critical applications such as autonomous driving \cite{sun2018feasibility} and human-robot interaction \cite{kimmel2017invariance}. For instance, existing iterative LQ game solvers \cite{khan2023leadership,hu2023emergent} cannot be directly integrated with the primal log barrier penalty method \cite{peters2020inference} %\cite{fridovich2020efficient} 
to incorporate these constraints.
% However, to the best of the authors' knowledge, the problem of computing FSE under nonlinear coupled equality and inequality constraints has not been fully addressed. {\color{red}The presence of these coupled equality and inequality constraints makes the straightforward approach of integrating existing iterative LQ game solvers \cite{khan2023leadership,nakamura2023opinion} with the primal log barrier penalty method \cite{fridovich2020efficient} infeasible.} 
The most relevant studies, such as \cite{mondal2019linear,fabiani2021local,maljkovic2023finding}, focus on computing OLSE in games under linear constraints. This paper aims to bridge this gap in the literature. %The closest works are \cite{mondal2019linear,fabiani2021local,maljkovic2023finding}, where the authors compute open-loop Stackelberg equilibrium for games under linear constraints. %a Riccati equation is derived for LQ games under a special class of constraints and then an open-loop Stackelberg equilibrium is computed by solving the Riccati equation. 

Our contributions are threefold: (1) We first reformulate the $N$-player feedback Stackelberg equilibrium problem, characterized by $N$ players making sequential decisions over time, into a sequence of nested optimization problems. This reformulation enables us to derive the Karush–Kuhn–Tucker (KKT) conditions and a second-order sufficient condition for the feedback Stackelberg equilibrium. % and to propose a second-order sufficient condition for the local feedback Stackelberg policy. 
(2) Using these results, we propose a Newton-style primal-dual interior point (PDIP) algorithm for computing a local FSE for LQ games. Under certain regularity conditions, we show the convergence of our algorithm to a local FSE. (3) %When extending our solution to constrained nonlinear games, we observe that the construction of the KKT conditions requires knowing high-order policy gradients, which are intractable to compute in practice. 
% We propose an efficient first-order quasi-policy approximation to the ground truth feedback Stackelberg policy, and we analyze the induced error. 
Finally, we propose an efficient PDIP method for approximately computing a local FSE for more general nonlinear games under (nonconvex) coupled equality and inequality constraints. The computed feedback policy locally approximates the ground truth nonlinear policy. Theoretically, we characterize the approximation error of our method, and show the exponential convergence under certain conditions. Empirically, we validate our algorithm in a highway lane merging scenario, demonstrating its ability to tolerate infeasible initializations and efficiently converge to a local FSE in constrained nonlinear~games.

\section{Related Works}
% [1. Feedback games, Nash and Stackelberg, constraints]
% Feedback games represent a class of multi-player decision model, where each player's action is dependent on the state information at different stage of a game. In matrix games, it is well-known that Stackelberg equilibrium is more beneficial to the leader than Nash equilibrium. It is known that mixed strategies may lead to better costs in feedback Stackelberg games than pure strategies \cite{basar1999dynamic}. In our work, we focus on pure strategies.
Closely related to the feedback Stackelberg equilibrium (FSE), the feedback Nash equilibrium (FNE) has been extensively studied, for example, in \cite{basar1976uniqueness, basar1999dynamic,reddy2016feedback,laine2023computation}. 
% The most evident difference between feedback games and open-loop games is that the decision is usually sub-game perfect, while the latter may not necessarily have this property, meaning that a small perturbation may lead to the sub-optimality or even unsafe maneuver of the subsequent actions. Though feedback is more robust than open-loop decision model, it's usually more challenging to compute, due to the sub-game perfectness. 
Our work builds upon \cite{laine2023computation}, where the authors proposed KKT conditions for constrained FNE. %, and we extend it to feedback Stackelberg equilibrium. 
%The major difference arises from the decision hierarchy in FSE. 
However, the FNE KKT conditions in \cite{laine2023computation} fail to hold true for FSE due to the decision hierarchy in FSE. In our work, we introduce a set of new KKT conditions for FSE. Another key difference is that we adopt the primal-dual interior point method for solving LQ and nonlinear games, whereas \cite{laine2023computation} considers the active-set method. In general, the former has polynomial complexity, but the latter has exponential complexity \cite{goswami2012comparative}. Moreover, we are able to prove the exponential convergence of our algorithm under certain conditions. However, there is no such convergence proof in \cite{laine2023computation}.% as well as other iterative LQ games approximation approaches \cite{khan2023leadership,nakamura2023opinion}. %Moreover, we show that the structure of PDIP allows us to characterize the convergence of our algorithm under certain conditions, whereas there is no convergence proof in \cite{laine2023computation}. 

% [2. Riccati equation approach, Pontryagin method, and Forrest's KKT] 
% Computing feedback Stackelberg equilibrium is difficult \cite{xie1997time}.
As highlighted in the literature, e.g., \cite{basar1999dynamic,vamvoudakis2012online,mylvaganam2014approximate,li2019two,xu2020adaptive}, the dominant approach to computing unconstrained FSE is using (approximate) dynamic programming. LQ games can be solved efficiently via exact dynamic programming; however, in more general nonlinear cases the value function could be hard to compute and, in general, has no analytical solution \cite{mylvaganam2014approximate}.
% As such, policy iteration \cite{vamvoudakis2012online} and approximate dynamic programming \cite{mylvaganam2014approximate} have been proposed for computing FSE. 
Compared with those works, our approach could be considered as computing an efficient local approximation of the value function along the state trajectory under a local FSE policy instead of approximating the value function everywhere as in \cite{mylvaganam2014approximate}. 

%There are many works on applying dynamic programming to solve linear quadratic games for feedback Nash or Stackelberg equilibrium. However, this approach may require to know the exact parameterization or function class of the value function. When we don't have such information, approximated dynamic programming is proposed to compute approximation solution. However, how to handle equality and inequality constraint in (approximate) dynamic programming approach can be challenging. Pontryagin minimum principle, on the other hand, is usually used for computing open-loop strategies, where the first-order optimality condition, or known as KKT condition, is leveraged for computing equilibrium strategies. Though we may apply the same technique to nonlinear games, the equilibrium strategy computed may be a local policy and, as discussed before, not sub-game perfect. 

% An approximated feedback Stackelberg solution for a class of differential games is proposed in \cite{mylvaganam2014approximate}, which approixmately solve the Hamilton-Jacobi PBE defined by a two-player feedback Stackelberg games. A class of feedback Stackelberg LQ games where players have decoupled constraints is considered in \cite{mondal2019linear}.

% Recently, there is a trend of using open-loop strategy to approximate feedback strategy in games. 

Finally, to further motivate our work, we examine whether the Stackelberg equilibrium can be effectively approximated by the Nash equilibrium and whether the FSE can be accurately approximated by its open-loop counterpart. According to \cite{korzhyk2011stackelberg}, in repeated matrix games, the Stackelberg equilibrium may coincide with the Nash equilibrium. However, in Appendix \ref{subsec:appendix comparing FSE and FNE}, we present a counterexample demonstrating that, in games with quadratic costs—reminiscent of oligopoly models in economics \cite{varian2014intermediate}—the Stackelberg equilibrium can deviate arbitrarily from the Nash equilibrium. %However, we include a counter-example in Appendix \ref{subsec:appendix comparing FSE and FNE} that a Stackelberg equilibrium could be arbitrarily different from a Nash equilibrium in games with quadratic costs, inspired by the Oligopoly games in economics \cite{varian2014intermediate}. 
Moreover, there is a recent trend of approximating feedback policies via receding horizon open-loop policies \cite{le2022algames,zhao2023stackelberg}, where an open-loop policy is re-solved at each time for future steps. 
% It has been observed that feedback Stackelberg equlibrium can coincide with feedback Nash equilibrium in matrix games, 
% People may argue the possibility of approximating feedback Stackelberg equilibrium via feedback Nash equilibrium, because they could coincide in matrix games \cite{korzhyk2011stackelberg}. However, we show a counter example in the Appendix that feedback Stackelberg equilibrium in LQ games could be sufficiently different from feedback Nash equilibrium. 
% Recently, people may argue using OLSE to approximate feedback Stackelberg equilibrium \cite{zhao2023stackelberg}. 
However, we show in another counter-example in Appendix \ref{subsec: appendix RHOL fails to approximate FSE} that the trajectory under the feedback Stackelberg policy and the one under the receding horizon open-loop Stackelberg policy could be quite different, even if there is no state perturbation. Thus, it is essential to develop specific tools for computing the feedback Stackelberg equilibrium.
% Unlike OLSE, there is no general method for computing feedback Stackelberg equilibrium \cite{martin2021coincidence}. People may also ask why not using Nash to approximate Stackelebrg, however, we provide a counter example. 

% Recently, KKT condition for feedback Nash equilibrium is proposed in \cite{laine2023computation}. 

% [3. Computation methods, Newton, active-set, PDIP] Efficient computation of equilibrium strategies is a challenging task. Active set has exponential complexity in the wosrt case, but PDIP has polynomial complexity \cite{goswami2012comparative}.

% We will introduce the comparison between Stackelberg equilibrium and Nash equilibrium, and feedback vs. open-loop Stackelberg equilibrium here. The main purpose is to motivate the problem of feedback Stackelberg equilibrium.

% We will also show that receding-horizon open-loop Stackelberg equilibrium strategy cannot replace feedback Stackeleberg equilibrium. This suggests the necessity of having tools for computing feedback Stackelberg equilibrium strategies. 

\section{Constrained Feedback Stackelberg Games}
In this section, we introduce the formulation of constrained feedback Stackelberg games. We formulate the problem by extending the $N$-player feedback Stackelberg games \cite{gardner1978feedback} to its constrained setting. We denote by $\mathbb{N}$ and $\mathbb{R}$ the sets of natural numbers and real numbers, respectively. Given $j,k\in\mathbb{N}$, we denote by $\mathbf{I}_{j}^k = \{j,j+1,\dots,k\}$ if $j\le k$ and $\emptyset$ otherwise. Let $T\in\mathbb{N}$ be the time horizon over which the game is played. %At each time $t$, we denote by $x_t \in \mathbb{R}^n$ the state of the game. %Each player is associated with a state vector $x_t^i\in\mathbb{R}^{n_i}$ 
At each time $t$, we denote by $x_t$ and $u_t^i\in\mathbb{R}^{m_i}$ the state of the entire game and the control input of player $i$, respectively. 
We define $u_t:=[u_t^1,u_t^2,\dots,u_t^N]\in\mathbb{R}^m$, with $m := \sum_{i=1}^N m_i$, to be the joint control input at time $t$.
Moreover, at each time $t$, players make decisions in the order of their indices. %For each $t\in\mathbf{I}_0^T$, we define %$x_t := [x_t^1,x_t^2,\dots,x_t^N]\in\mathbb{R}^{n}$, with $n:=\sum_{i=1}^N n_i$, and 
We consider the time-varying dynamics
\begin{equation}\label{eq:dynamics}
    x_{t+1} = f_t(x_t,u_t),
\end{equation}
where $f_t(x_t,u_t): \mathbb{R}^{n}\times \mathbb{R}^m \to \mathbb{R}^n$ is assumed to be a twice differentiable function. Given a sequence of control inputs $\uvec:=[u_0,u_1,\dots,u_T]\in\mathbb{R}^{Tm}$, we denote by $\xvec:=[x_0,x_1,\dots,\allowbreak x_{T+1}]\in\mathbb{R}^{(T+1)n}$ a state trajectory under dynamics \eqref{eq:dynamics}.  

At each time $t\in\mathbf{I}_0^T$, we denote the stage-wise cost of player $i\in\mathbf{I}_1^N$ by $\ell_t^i(x_t,u_t):\mathbb{R}^{n}\times \mathbb{R}^m\to \mathbb{R}$, and associate with each player a terminal cost, $\ell_{T+1}^i(x_{T+1}): \mathbb{R}^n \to \mathbb{R}$.  Each player $i\in\mathbf{I}_1^N$ considers the following time-separable costs,\vspace{-0.5em}
\begin{equation}\label{eq:costs}
    J^i(\xvec,\uvec) = \sum_{t=0}^T \ell_t^i(x_t,u_t) + \ell_{T+1}^i(x_{T+1}).\vspace{-0.5em}
\end{equation}
Moreover, let $n_{h,t}^i$ and $n_{g,t}^i$ be the number of equality and inequality constraints held by player $i\in\mathbf{I}_1^N$ at time $t$, respectively. We denote the equality and inequality constraint functions of player $i$ by $h_t^i(x_t,u_t): \mathbb{R}^n\times \mathbb{R}^m \to \mathbb{R}^{n_{h,t}^i}$ and $g_t^i(x_t,u_t):\mathbb{R}^n\times \mathbb{R}^m \to \mathbb{R}^{n_{g,t}^i}$, respectively. We specify the stage-wise equality and inequality constraints of player $i\in\mathbf{I}_{1}^N$ as
\begin{equation}\label{eq:constraints at t}
    \begin{aligned}
        0=h_t^i(x_t,u_t),\ 0\le g_t^i(x_t,u_t).
    \end{aligned}
\end{equation}
At the terminal time $t=T+1$, we represent the equality and inequality constraint functions of player $i\in\mathbf{I}_1^N$ by $h_{T+1}^i(x_{T+1}):\mathbb{R}^{n} \to \mathbb{R}^{n_{h,T+1}^i}$ and $g_{T+1}^i(x_{T+1}):\mathbb{R}^n \to \mathbb{R}^{n_{g,{T+1}}^i}$, respectively. We consider the following equality and inequality constraints of player $i\in\mathbf{I}_1^N$ at the terminal time,
\begin{equation}\label{eq:terminal constraints}
    0=h_{T+1}^i(x_{T+1}), \ 0\le g_{T+1}^i(x_{T+1}).
\end{equation}
We remark that these definitions generate coupled dynamics and constraints among different players at each time $t\in \mathbf{I}_0^{T+1}$. 
We consider the following regularity assumption, following \cite{laine2023computation,chinchilla2023newton}.
\begin{assumption}\label{assumption:feasible}
    The feasible set $\mathcal{F}:=\{x\in\mathbb{R}^{(T+1)n},u\in\mathbb{R}^{Tm}:h_t^i(x_t,u_t)=0,g_t^i(x_t,u_t)\ge0, h_{T+1}^i(x_{T+1})=0,g_{T+1}^i(x_{T+1})\ge 0, x_{t+1}=f_t(x_t,u_t),\forall i\in\mathbf{I}_{1}^N,t\in\mathbf{I}_0^T \}   $
    % \begin{equation}\label{eq:normal form cost}
    % \begin{aligned}
    %      \mathcal{F}:=\{&x\in\mathbb{R}^{(T+1)n},u\in\mathbb{R}^{Tm}:h_t^i(x_t,u_t)=0,g_t^i(x_t,u_t)\ge0,\\ & h_{T+1}^i(x_{T+1})=0,g_{T+1}^i(x_{T+1})\ge 0, x_{t+1}=f_t(x_t,u_t),\forall i\in\mathbf{I}_{1}^N,t\in\mathbf{I}_0^T \}   
    % \end{aligned}
    % \end{equation}
    is compact. The costs, dynamics, equality and inequality constraints are twice differentiable and bounded, but could be nonconvex in general. 
\end{assumption}

% For completeness, we provide a sufficient condition that ensures the existence of a unique feedback Stackelberg equilibrium, which is adapted from Corollary 4.4 in \cite{basar1999dynamic}. 
% \begin{assumption}\label{assumption:strict convex}
%     For all players $i\in\mathbf{I}_1^N$ and time $t\in\mathbf{I}_0^T$, the cumulative costs of player $i$ at time $t$, 
%     \begin{equation}\label{eq:reduced_form_cost}
%     \begin{aligned}
%         \sum_{\tau=t}^T \   \ell_\tau^i(x_\tau,u_\tau) + \ell_{T+1}^i(x_{T+1})
%     \end{aligned}
%     \end{equation}
%     where $x_{\tau+1} = f_\tau(x_\tau,u_\tau), \tau \in\{ t,\dots,T\}$, is strongly convex \cite{boyd2004convex} with respect to player $i$'s control $ \{u_\tau^i, \tau\in\mathbf{I}_t^T\}$, when we fix other players' controls $\{u_\tau^j, \tau\in\mathbf{I}_t^T, j\in\mathbf{I}_1^N\setminus \{i\}\}$. 
% \end{assumption}

% Assumption~\ref{assumption:strict convex} is a sufficient but not necessary condition for the existence of a feedback Stackelberg equilibrium. 
% We provide a counter example in Appendix~\ref{subsec:appendix existence of policy}, which violates Assumption~\ref{assumption:strict convex} but has a feedback Stackelberg policy. In addition, suppose that the dynamics are convex in states and controls, and the costs of each player are convex and non-decreasing in states and are strictly convex in its own control inputs. Then, Assumption~\ref{assumption:strict convex} also holds true because the cost \eqref{eq:reduced_form_cost} of each player is strictly convex in its control \cite{boyd2004convex}. 
\subsection{Local Feedback Stackelberg Equilibria}

In this subsection, we formalize the decision process of feedback Stackelberg games. Before doing that, we introduce a few notations to compactly represent different players' control at different times. We define $u_{t:t'}^{i:i'}:=\{u_\tau^j,\tau\in\mathbf{I}_{t}^{t'}, j\in\mathbf{I}_{i}^{i'}\}$. In particular, we define $u_t^{1:i-1}:=\emptyset$ when $i=1$ and $u_t^{i+1:N}:=\emptyset$ when $i=N$. We also denote by $u_{t+1:T}^{1:i} := \emptyset$ when $t = T$. 

% At each stage of a feedback Stackelberg game, players make decisions following their index order. 
The policy of each player can be defined as follows. At the $t$-th stage, since player 1 makes a decision first, its policy function $\pi_t^1(x_t):\mathbb{R}^n \to \mathbb{R}^{m_1}$ depends only on the state $x_t$. For players $i\in\mathbf{I}_2^N$, the policies are modeled as $\pi_t^i(x_t,u_t^{1:i-1}):\mathbb{R}^n\times \mathbb{R}^{\sum_{j=1}^{i-1}m_j}\to \mathbb{R}^{m_i}$. %Let $\Pi_t^i$ denote the class of all permissible policies of player $i$ at stage $t$, i.e., all measurable mappings $\pi_t^i$ from $\mathbb{R}^n\times \mathbb{R}^{\sum_{j=1}^{i-1}m_j}$ into $\mathbb{R}^{m_i}$. Let $\Pi^i:=[\Pi_0^i,\Pi_1^i,\dots,\Pi_T^i]$ and $\pi^i:=[\pi_0^i,\pi_1^i,\dots,\pi_T^i]$. 
We will define the concept of \emph{local feedback Stackelberg equilibria} in the remainder of this subsection. 

At the terminal time $t=T+1$, we define the state-value functions for a player $i\in\mathbf{I}_1^N$ as
\begin{equation}\label{eq:V_{T+1}}
        V_{T+1}^{i}(x_{T+1}):= \left\{\begin{aligned} &\ell_{T+1}^i(x_{T+1})  &&\textrm{if }\left\{\substack{0=h_{T+1}^i(x_{T+1})\\ 0\le g_{T+1}^i(x_{T+1}) }\right. \\ & \infty &&\textrm{else.}  \end{aligned}  \right. 
\end{equation}

At time $t\le T$, we first construct the state-action-value function for player $N$:%the $N$-th player:
\begin{equation}\label{eq:Z_t^N}
    Z_{t}^{N}(x_t, u_t^{1:N-1}, u_t^N):= \left\{\begin{aligned} & \ell_t^N(x_t,u_t)+V_{t+1}^{N}(x_{t+1}) &&\textrm{if }\left\{\substack{0= x_{t+1}-f_t(x_t,u_t) \\ 0=h_t^N(x_t,u_t)\\ 0\le g_t^N(x_t,u_t) }\right. \\ &\infty &&\textrm{else.} \end{aligned} \right.
\end{equation}
Given $(x_t,u_t^{1:N-1})$, there could be multiple $u_t^N$ minimizing $Z_t^N(x_t,u_t^{1:N-1},u_t^N)$. We define player $N$'s local FSE policy $\pi_t^{N}$ by picking an arbitrary local minimizer $u_t^{N*}$, %We define the policy $\pi_t^N$ such that it picks an arbitrary minimizer of the state-action-value function,%we define the policy in terms of its least square solution, 
\begin{equation}\label{eq:N-th player's strategy}
    % \pi_t^{N*}(x_t,u_t^{1:N-1}) := \arg \min_{u_t^N}  \|u_t^N\|_2^2 \hspace{1cm} \textrm{s.t. } u_t^N \in \arg \min_{\tilde{u}_t^N}  Z_t^{N}(x_t,u_t^{1:N-1},\tilde{u}_t^N)
    \pi_t^{N}(x_t,u_t^{1:N-1}) := u_t^{N*} \in \arg \min_{\tilde{u}_t^N}  Z_t^{N}(x_t,u_t^{1:N-1},\tilde{u}_t^N).
\end{equation}
%{\color{red}We note that notations without a star marker represent normal values, which may or may not be locally optimal. }
%where $\pi_t^{N*}$ is a lower semi-continuous mapping by construction. 
We then construct the state-action-value function of player $i\in\mathbf{I}_2^{N-1}$,
\begin{equation}\label{eq:Z_t^i}
    Z_{t}^{i}(x_t, u_t^{1:i-1}, u_t^i):= \left\{\begin{aligned} & \ell_t^i(x_t,u_t)+V_{t+1}^{i}(x_{t+1}) &&\textrm{if }\left\{\substack{ 0=x_{t+1} - f_t(x_t,u_t) \\ 0=h_t^i(x_t,u_t)\\ 0 \le g_t^i(x_t,u_t)\\
    u_t^j=\pi_t^{j}(x_t,u_t^{1:j-1}),\ j\in\mathbf{I}_{i+1}^N}\right. \\ &\infty &&\textrm{else,} \end{aligned} \right.
\end{equation}
and its local FSE policy $\pi_t^{i}$ by picking an arbitrary local minimizer $u_t^{i^*}$,
\begin{equation}\label{eq:i-th player's strategy}
        \pi_t^{i}(x_t,u_t^{1:i-1}): = %\arg\min_{u_t^i}  \|u_t^i\|_2^2\hspace{1cm}\textrm{s.t. } 
        u_t^{i*}\in \arg\min_{\tilde{u}_t^i}  Z_t^i(x_t,u_t^{1:i-1}, \tilde{u}_t^i).
\end{equation}
We finally construct the state-action-value function of the first player:
\begin{equation}\label{eq:Z_t^1}
    Z_t^{1}(x_t,u_t^1): = \left\{ \begin{aligned}
        &\ell_t^1(x_t,u_t)+ V_{t+1}^{1}(x_{t+1}) &&\textrm{if }\left\{\substack{0=x_{t+1}-f_t(x_t,u_t) \\ 0=h_t^1(x_t,u_t)\\ 0\le g_t^1(x_t,u_t)\\  u_t^j=\pi_t^{j}(x_t,u_t^{1:j-1}),\ j\in\mathbf{I}_{2}^N}\right. \\
        &\infty &&\textrm{else,}
    \end{aligned} \right.
\end{equation}
and its local FSE policy
\begin{equation}\label{eq:leader's strategy}
        \pi_t^{1}(x_t) := %\arg \min_{u_t^1} \|u_t^1\|^2_2\hspace{1cm}  \textrm{s.t. }
        u_t^{1*}\in \arg \min_{\tilde{u}_t^1} Z_t^1(x_t,\tilde{u}_t^1).
\end{equation}
We define the state-value function of player $i\in\{1,2,\dots,N\}$ at time $t\le T$ as 
\begin{equation}\label{eq:value function}
    \begin{aligned}
        V_t^i(x_t)=&Z_t^{i}(x_t,u_t^{1*},\dots,u_t^{i*}),
    \end{aligned}
\end{equation}
where $u_t^{j*} = \pi_t^{j}(x_t, u_t^{1:(j-1)*}), \forall j \in \mathbf{I}_1^i$.
% \begin{equation}\label{eq:defining optimal controls in dynamic programming}
% \begin{aligned}
%     u_t^{j*} = \pi_t^{j*}(x_t, u_t^{1:j-1*}), \forall j \in \mathbf{I}_1^i
%     % &u_t^{1*} = \pi_t^{1*}(x_t)\\
%     % & u_t^{2*} = \pi_t^{2*}(x_t,u_t^{1*})\\
%     % & \ \ \ \ \vdots \\
%     % & u_t^{i*} = \pi_t^{i*}(x_t,u_t^{1*},\dots,u_t^{(i-1)*})
% \end{aligned}
% \end{equation}

We formally define the local feedback Stackelberg equilibria as follows.
\begin{definition}[Local Feedback Stackelberg Equilibria \cite{basar1999dynamic}]\label{def:local fse policy}
    Let $\{\pi_t^{i}\}_{t=0,i=1}^{T,N}$ be a set of policies defined in \eqref{eq:N-th player's strategy}, \eqref{eq:i-th player's strategy} and \eqref{eq:leader's strategy}, and define $(\xvec^*, \uvec^*)$ to be a state and control trajectory under the policies $\{\pi_t^{i}\}_{t=0,i=1}^{T,N}$, i.e.,
    \begin{equation}
        \begin{aligned}
            x_{t+1}^* = f_t(x_t^*,u_t^*),\  %&&\forall t\in\mathbf{I}_0^T\\
            u_t^{i*} =  \pi_t^{i}(x_t^*,u_t^{1:(i-1)*}),\  \forall t \in \mathbf{I}_0^T,\  i\in\mathbf{I}_1^N.
        \end{aligned}
    \end{equation}
    We say that $(\xvec^*, \uvec^*)$ is a \textbf{local feedback Stackelberg equilibrium trajectory} %. Moreover, the set of policies $\{\pi_t^{i}\}_{t=0,i=1}^{T,N}$ is called a set of \textbf{local feedback Stackelberg policies} 
    if there exists an $\epsilon>0$ such that, for all $t\in\mathbf{I}_0^T$,
    \begin{equation}
        \begin{aligned}
            Z_t^{1}(x_t^*, \tilde{u}_t^1) &\ge Z_t^{1}(x_t^*,u_t^{1*}),\\
            % Z_t^{2}(x_t^*,u_t^{1*},\tilde{u}_t^2)&\ge Z_t^{2}(x_t^*,u_t^{1*}, u_t^{2*}),\\
            &\ \vdots\\
            Z_t^{N}(x_t^*,u_t^{1*},\dots,u_t^{(N-1)*},\tilde{u}_t^N)&\ge Z_t^N(x_t^*,u_t^{1*},\dots,u_t^{(N-1)*},u_t^{N*}) 
        \end{aligned}
    \end{equation}
    for all $ \tilde{u}_t^1 \in \{u:\|u - u_t^{1*} \|_2\le \epsilon\}$, $\dots$, and $\tilde{u}_t^N \in \{u:\|u-u_t^{N*}\|_2\le \epsilon\} $. %with $u_t^{1}=\pi_t^{1}(x_t)$, $\dots$, and $u_t^{N} = \pi_t^{N}(x_t,u_t^{1},\dots,u_t^{N-1})$.
    % for all $x_t \in \{x:\| x - x_t^*\|\le \epsilon\}$, $ \tilde{u}_t^1 \in \{u:\|u - u_t^{1} \|_2\le \epsilon\}$, $\dots$, and $\tilde{u}_t^N \in \{u:\|u-u_t^{N}\|_2\le \epsilon\} $, with $u_t^{1}=\pi_t^{1}(x_t)$, $\dots$, and $u_t^{N} = \pi_t^{N}(x_t,u_t^{1},\dots,u_t^{N-1})$. 
\end{definition}
% \begin{remark}
%     \color{red} missing a $u^*$?
% \end{remark}

The above definition encapsulates the traditional approach to computing feedback Stackelberg equilbiria. This involves optimizing over state-action-value functions, which are obtained by integrating other players' policies into each player's problem and then recording the overall costs. 
\begin{remark}[Existence of Local Feedback Stackelberg Equilibria]
In general, it is difficult to establish a sufficient condition for the existence of a feedback Stackelberg equilibrium \cite{bensoussan2019feedback}. The main difficulty is that the decision problem of each player is nested within that of other players. It must be solved hierarchically. For example, the existence of feedback Stackelberg policies \cite{lucchetti1987existence} of a player $i\in\mathbf{I}_1^{N-1}$ is related to the topological properties of the set of policies of players $j\in\mathbf{I}_{i+1}^N$. Even if all the players' costs are convex, the feedback Stackelberg policy of player $N$ at the terminal time could be lower semi-continuous. Subsequently, the cost of player $(N-1)$ could become upper semi-continuous when substituting in the $N$-th player's policy into the $(N-1)$-th player's cost. Since there may not exist a solution when minimizing an upper semi-continuous function, there may not exist a feedback Stackelberg policy for player $(N-1)$. However, if we can show that the policy of each player is always continuous in the state and prior players' controls, and the continuous costs are defined on a compact domain, then there exist feedback Stackelberg equilibria \cite{basar1999dynamic}. 
\end{remark}

% Having established the aforementioned definitions, 
%, subject to constraints. 
We will now proceed to characterize the feedback Stackelberg equilibria in greater detail in the subsequent section.

\section{Necessary and Sufficient Conditions for Local Feedback Stackelberg Equilibria}

We show in the following theorem that the dynamic programming problem, as described in Definition \ref{def:local fse policy}, can be reformulated as a sequence of nested constrained optimization problems. In this reformulation, the policies for other players are integrated as constraints within the problem of each player $i$, instead of being directly substituted into the costs for computing state-action-value functions, as is typical in traditional optimal control literature. %In this reformulation, the policies for other players are integrated as constraints within each player $i$’s problem, rather than being substituted directly to compute state-action-value functions in the traditional optimal control literature. 
% This differs from the traditional method in optimal control literature, where these policies are directly substituted when computing state-action-value functions.
% other players’ policies are represented as constraints within each player’s problem, rather than being directly substituted into each player's problem to compute state-action-value functions. 
This approach enables us to establish KKT conditions for feedback Stackelberg games in the latter part of this subsection.%which encode the decision hierarchy of different players and also the impact of an action of a player on future stages.
\begin{theorem}\label{thm:nested constrained optimization problem}
    % Consider an $N$-player feedback Stackelberg game. 
    Under Assumption~\ref{assumption:feasible}, for each $t\in \mathbf{I}_0^T$ and each $i\in\mathbf{I}_1^N$, a local feedback Stackelberg policy $\piti$ can be equivalently represented as an optimization problem, given the knowledge of current state $\knownxt$ and prior players' actions $\knownutoneiminusone$, %constitutes a set of local feedback Stackelberg policies if for all players $i\in \mathbf{I}_1^N$, at all times $t\in \mathbf{I}_0^T$, and with the knowledge of current state $\knownxt$ and prior players' action $\knownutoneiminusone$, its policy $\pi_t^i$ satisfies
    % \begin{equation}\label{eq:player's policy constrained problem}
    %     \begin{aligned}
    %         \pi_t^i(x_t,u_t^{1:i-1}) = \arg \min_{u_t^i \in \mathcal{U}_t^{i*}}  \|u_t^i\|_2^2
    %     \end{aligned}
    % \end{equation}
    % where we define the best response set of player $i$ at time $t$ as
    \begin{subequations}\label{eq:player's constrained problem}
    \begin{align}
        % \mathcal{U}_t^{i*}:= 
        \piti(\knownxt,\knownutoneiminusone) = \tilde{u}_t^i
        \in %\mathcal{U}_t^{i*}:=
        \underset{u_t^i}{\arg} \min_{\substack{
        u_{t}^{i:N} \\   u_{t+1:T}^{1:N}  \\ 
        x_{t+1:T+1}
        }} & \ell_t^i(\knownxt, \knownutoneiminusone, u_t^{i:N}) +\hspace{-0.2cm} \sum_{\tau = t+1}^T \ell_{\tau}^i(x_\tau, u_\tau)  +  \ell_{T+1}^i(x_{T+1})\\
        \textrm{s.t. }
        &0=u_{t}^j - \pitj(\knownxt, \knownutoneiminusone,u_t^{i:j-1}), &&\hspace{-1.2cm} j\in  \mathbf{I}_{i+1}^N \label{eq:subsequent player's strategy} \\
        % &0=u_{\tau}^j - \pi_\tau^{j}(x_\tau, u_\tau^{1:j-1}) && \tau \in \mathbf{I}_{t}^T,j\in \mathbf{I}_{i+1}^N \\
        &0=x_{t+1} - f_t(\knownxt, \knownutoneiminusone, u_t^{i:N}), \label{eq:dynamics at time t}\\
        &0 = h_t^i(\knownxt,\knownutoneiminusone, u_t^{i:N}), \
        0\le g_t^i(\knownxt,\knownutoneiminusone, u_t^{i:N}) \label{eq:constraints at time t}\\
        &0=u_{\tau}^j - \pitauj(x_\tau, u_\tau^{1:j-1}), &&\hspace{-3.2cm} \tau \in \mathbf{I}_{t+1}^T,j\in\mathbf{I}_1^N\setminus\{i\} \label{eq:future player's strategy} \\
        &0=x_{\tau+1} - f_\tau(x_\tau,u_\tau), &&\hspace{-1.2cm} \tau\in\mathbf{I}_{t+1}^T \label{eq:future player's dynamics}\\
        &0 = h_\tau^i(x_\tau,u_\tau), \
        0\le g_\tau^i(x_\tau,u_\tau), &&\hspace{-1.2cm} \tau\in\mathbf{I}_{t+1}^T \label{eq:future player's constraints} \\
        &0=h_{T+1}^i(x_{T+1}), \
        0\le g_{T+1}^i(x_{T+1})
    \end{align}
    \end{subequations}
    where we drop \eqref{eq:subsequent player's strategy} when $i=N$, and we drop \eqref{eq:future player's strategy}, \eqref{eq:future player's dynamics} and \eqref{eq:future player's constraints} when $t=T$. 
    The notation $\arg_{u}\min_{u,v}$ represents that we minimize over $(u,v)$ but only return $u$ as an output.
    % where $u_t^{1:i-1}=\emptyset$ when $i=1$ and $u_t^{i+1:N}=\emptyset$ when $i=N$. We also denote by $u_{t+1:T}^{1:i} = \emptyset$ when $t = T$.
\end{theorem}
\begin{proof}
    The proof can be found in the Appendix. 
\end{proof}
In what follows, we will characterize the KKT conditions of the constrained optimization problems in \eqref{eq:player's constrained problem}. Before doing that, we first introduce Lagrange multipliers, which facilitate the formulation of Lagrangian functions for all players.%, and therefore pave the way for KKT conditions. %Before doing that, we introduce a few assumptions to facilitate the discussion.

% {\color{red}Note that even though we have this monotone gradient assumption, how can we say that the log penalty will not destroy this assumption if we don't have convex feasible domain? }

% \begin{assumption}
%     The gradient vector $[\nabla_{x_t}\ell_t^1, \nabla_{u_t^1}\ell_t^1, \nabla_{x_t}\ell_t^2, \nabla_{u_t^2}\ell_t^2]$ is a strongly monotone operator.
% \end{assumption}

% {\color{red}Show the differentiability of the policy?}
% So, we know that if there exits a solution to the optimization problem at each stage, then from the definition of policy, we want to show that
% \begin{equation}
%     \lim_{\tilde{x}\to x}\frac{\partial u^*(\tilde{x})}{\partial \tilde{x}}
% \end{equation}
% exists. Essentially, that requires to have the existence of the solution.
% \begin{remark}\color{red}
%     Is it true that we will lose differentiability of the policy when having binding constraints? What about if we show that the policy resulted from PDIP will be differentiable? Yeah, that's true. So, a good solution can be that proving that the policies derived from the log-penalized objectives are differentiable. 
% \end{remark}

% \begin{enumerate}\color{red}
%     \item Even though we have the above assumption, the problem of the convergence proof is still, we don't know what an optimal strategy looks like. The difficulty comes from that the higher order gradient term of the policy is hard to compute. So, that motivates the quasi-Newton approximation. 
%     \item What about we show that the algorithm converges to the Quasi-Newton optimal solution?
% \end{enumerate}

Let $t\in\mathbf{I}_0^{T}$ and $i\in\mathbf{I}_1^N$. We denote by $\lambda_t^i\in\mathbb{R}^{n}$ the Lagrange multiplier for the dynamics constraint $0=x_{t+1}-f_t(x_t,u_t)$. Let $\mathbb{R}_{\ge0}$ be the set of non-negative real numbers. We define $\mu_t^i\in\mathbb{R}^{n_{h,t}^i}$ and $\gamma_t^i\in\mathbb{R}_{\ge 0}^{n_{g,t}^i}$ to be the Lagrange multipliers for the constraints $0=h_t^i(x_t,u_t)$ and $0\le g_t^i(x_t,u_t)$, respectively. When $t\le T$, the constrained problem \eqref{eq:player's constrained problem} of player $i<N$ considers the feedback interaction constraint $0=u_t^j - \pitj(x_t,u_t^{1:j-1})$, $j\in\mathbf{I}_{i+1}^N$. Thus, we associate those constraints with multipliers $\psi_t^i:=[\psi_t^{i,i+1},\psi_t^{i,i+2},\dots,\psi_t^{i,N}]$, where $\psi_t^{i,j}\in\mathbb{R}^{m_i}$. Moreover, when $t<T$, the constrained problem \eqref{eq:player's constrained problem} of a player $i\le N$ includes the feedback interaction constraints $0=u_{\tau+1}^j -\pi_{\tau+1}^j(x_{\tau+1},u_{\tau+1}^{1:j-1}) $, for $\tau\ge t$ and $j\in \mathbf{I}_1^N\setminus\{i\}$. Thus, we associate those constraints with multipliers $\eta_t^i:=[\eta_t^{i,1},\dots, \eta_t^{i,i-1}, \eta_t^{i,i+1},\allowbreak\dots,\eta_t^{i,N}]$, where $\eta_t^{i,j}\in\mathbb{R}^{m_j}$. Finally, we simplify the notation by defining $\lambda_t:=[\lambda_t^1,\lambda_t^2,\dots,\lambda_t^N]$, and define $\mu_t$, $\gamma_t$, $\eta_t$, and $\psi_t$ accordingly.  

Subsequently, we define the Lagrangian functions of all the players. We first consider player $i\in \mathbf{I}_{1}^{N}$,
% \begin{equation}
%     \begin{aligned}
%         L_t^1(x_t,&x_{t+1},u_t,u_{t+1} \lambda_t, \mu_t, \gamma_t, \eta_t, \psi_t): =  \ell_{t}^1(x_t,u_t)+\lambda_t^{1\top} (x_{t+1}-f_t(x_t,u_t)) \\
%         &- \mu_t^{1\top} h_t^1(x_t,u_t)-\gamma_t^{1\top} g_t^1(x_t,u_t) \\& + \sum_{j\in\mathbf{I}_2^N} \psi_t^{1,j^\top} (u_t^j - \pi_t^j(x_t,u_t^{1:j-1})) %+\psi_t^{1,2\top} (u_t^2 - \pi_t^2(x_t,u_t^{1}))  +\dots + \psi_t^{1,N\top} (u_t^N - \pi_t^N(x_t,u_t^{1:N-1})) 
%          + \sum_{j\in\mathbf{I}_2^N} \eta_t^{1,j\top} (u_{t+1}^j - \pi_{t+1}^j(x_{t+1},u_{t+1}^{1:j-1}))  %\eta_t^{1,2\top}(u_{t+1}^{2} - \pi_{t+1}^2(x_{t+1},u_{t+1}^{1}))
%         %+\dots+ \eta_t^{1,N\top} (u_{t+1}^N - \pi_{t+1}^N(x_{t+1},u_{t+1}^{1:N-1}))  
%     \end{aligned}
% \end{equation}
% where the terms represent the player $i$'s cost, dynamics constraint, equality and inequality constraint, and constraints encoding the feedback interaction among players at the current and next time steps. 
% Similarly, for the player $i\in\{2,\dots,N-1\}$, we have
\begin{equation}\label{eq:Lagrangian for other players, time t}
    \begin{aligned}
        L_t^i(x_{t:t+1},u_{t:t+1},\lambda_t,\mu_t, \gamma_t, \eta_t,\psi_t): &= \ell_{t}^i(x_t,u_t) - \lambda_t^{i\top} (x_{t+1} - f_t(x_t,u_t)) \\
        &- \mu_t^{i\top } h_t^i(x_t,u_t)-\gamma_t^{i\top }g_t^i(x_t,u_t) \\ & - \sum_{j\in\mathbf{I}_{i+1}^N} \psi_t^{i,j\top}(u_t^{j} - \pi_t^j(x_t,u_t^{1:j-1})) %\psi_{t}^{i,i+1\top}(u_{t}^{i+1} - \pi_{t}^{i+1}(x_t,u_t^{1:i+1}))
         %+ \dots + \psi_{t}^{i,N\top}(u_t^N - \pi_t^N(x_t,u_t^{1:N-1})) 
          \\&- \sum_{j\in \mathbf{I}_1^N\setminus\{i\}}  \eta_t^{i,j^\top} (u_{t+1}^j - \pi_{t+1}^j(x_{t+1}, u_{t+1}^{1:j-1})) %  + \eta_t^{i,1\top} (u_{t+1}^1 - \pi_{t+1}^1(x_t) )+\dots+\eta_{t}^{i,i-1\top}(u_{t+1}^{i-1} - \pi_{t+1}^{i-1}(x_{t+1},u_{t+1}^{1:i-1}))+\dots  \\ & + \eta_{t}^{i,i+1\top}(u_{t+1}^{i+1} - \pi_{t+1}^{i+1}(x_{t+1},u_{t+1}^{1:i+1})) + \dots  +\eta_{t}^{i,N\top}(u_{t+1}^N - \pi_{t+1}^N(x_{t+1},u_{t+1}^{1:N-1}))
    \end{aligned}
\end{equation}
where the right hand side terms represent player $i$'s cost, dynamics constraint, equality and inequality constraints, and constraints encoding the feedback interaction among players at the current and future time steps. 

Furthermore, at the terminal time $t=T$, for player $i\in\mathbf{I}_1^{N}$, we consider
% \begin{equation}
%     \begin{aligned}
%     L_{T}^1(x_T,&x_{T+1},u_T,\lambda_T, \mu_T, \gamma_T ,\psi_T,\gamma_{T+1},\mu_{T+1}):=\ell_T^1(x_T,u_T) + \ell_{T+1}^1(x_{T+1}) \\& + \lambda_T^{1\top}(x_{T+1} - f_T(x_T,u_T)) - \mu_T^{1\top}h_T^1(x_T,u_T) - \gamma_T^{1\top}g_T^1(x_T,u_T)\\ & -\gamma_{T+1}^{1\top} g_{T+1}^1(x_{T+1})  - \mu_{T+1}^{1\top} h_{T+1}^1(x_{T+1}) + \psi_{T}^{1,2\top}(u_T^2 - \pi_T^2(x_T,u_T^1))\\
%     & + \dots + \psi_T^{1,N^\top}(u_T^N - \pi_T^N(x_T,u_T^{1:N-1}))
%     \end{aligned}
% \end{equation}
% for the player $i\in \{2,\dots,N-1\}$,
\begin{equation}\label{eq:Lagrangian of player i, time T}
    \begin{aligned}
        L_{T}^i(x_{T:T+1},u_T,\lambda_T, \mu_{T:T+1}, \gamma_{T:T+1} ,\psi_T):&=\ell_T^i(x_T,u_T) + \ell_{T+1}^i(x_{T+1}) \\& - \lambda_T^{i\top} (x_{T+1} - f_T(x_T,u_T)) \\ & - \mu_T^{i\top}h_T^i(x_T,u_T) - \mu_{T+1}^{i\top} h_{T+1}^i(x_{T+1}) \\& - \gamma_T^{i\top}g_T^i(x_T,u_T)-\gamma_{T+1}^{i\top} g_{T+1}^i(x_{T+1})   \\ & - \sum_{j\in \mathbf{I}_{i+1}^N} \psi_T^{i,j\top} (u_T^j - \pi_T^j(x_T,u_T^{1:j-1})) % + \psi_{T}^{i,i+1\top}(u_T^{i+1} - \pi_T^{i+1}(x_T, u_T^{1:i}))\\
    % & + \dots + \psi_T^{i,N^\top}(u_T^N - \pi_T^N(x_T,u_T^{1:N-1}))
    \end{aligned}
\end{equation}
where the right hand side terms represent player $i$'s costs, dynamics constraint, equality and inequality constraints, and constraints encoding the feedback interaction among players at the terminal time $T$. Note that there is no more decision to be made at time $t=T+1$, and therefore, there is no term representing the feedback interactions among players for future time steps in \eqref{eq:Lagrangian of player i, time T}, which is different from \eqref{eq:Lagrangian for other players, time t}. % and \eqref{eq:Lagrangian for leader, time t}. 

% Finally, for the $N$-th player, we have
% \begin{equation}\label{eq:Lagrangian of leader, time T}
%     \begin{aligned}
%         L_{T}^N(x_{T:T+1},u_T,\lambda_T, \mu_{T:T+1}, \gamma_{T:T+1} ,\psi_T):=&\ell_T^N(x_T,u_T) + \ell_{T+1}^N(x_T) \\& + \lambda_T^{N\top} (x_{T+1} - f_T(x_T,u_T)) \\ & - \mu_T^{N\top}h_T^N(x_T,u_T)  - \mu_{T+1}^{N\top} h_{T+1}^N(x_{T+1}) \\&  - \gamma_T^{N\top}g_T^N(x_T,u_T)  -\gamma_{T+1}^{N\top} g_{T+1}^N(x_{T+1}) 
%     \end{aligned}
% \end{equation}
% where the right handside terms represent that player $N$'s costs, dynamics constraint, equality and inequality constraints. Since player $N$ is the last player that makes decision, we don't account for any feedback interaction among players, which is different from \eqref{eq:Lagrangian of player i, time T}.

%Combining \eqref{eq:Lagrangian for other players, time t}, \eqref{eq:Lagrangian for leader, time t}, \eqref{eq:Lagrangian of player i, time T} and \eqref{eq:Lagrangian of leader, time T}, 
For all time steps $t\in\mathbf{I}_0^T$ and players $i\in\mathbf{I}_1^N$, assuming the state $x_t$ is given and each player $j<i$ has taken action $u_t^j$, we formulate the Lagrangian of the problem \eqref{eq:player's constrained problem} of player $i$ at the $t$-th stage as
\begin{equation}
    \begin{aligned}    
        \mathcal{L}_t^i (x_{t:T+1}, u_{t:T},& \lambda_{t:T}, \mu_{t:T+1}, \gamma_{t:T+1}, \eta_{t:T-1},\psi_{t:T} ): = \sum_{\tau=t}^{T-1} L_\tau^i(x_{\tau:\tau+1},u_{\tau:\tau+1}, \\ &\lambda_\tau, \mu_\tau, \gamma_\tau, \eta_\tau, \psi_\tau)  + L_T^i(x_{T:T+1},u_T,\lambda_T, \mu_{T:T+1}, \gamma_{T:T+1}, \psi_T)
    \end{aligned}
\end{equation}
where for each $\tau\in\mathbf{I}_{t+1}^{T}$, the terms associated with $\psi_\tau$ in $L_\tau^i$ ensure constraints already addressed by the terms associated with $\eta_{\tau-1}$ in $L_{\tau-1}^i$ and can therefore be dropped when defining $\mathcal{L}_t^i$. %In what follows, we will propose KKT conditions for the nested constrained optimization problem \eqref{eq:player's constrained problem}. 
We can concatenate the KKT conditions of each player at each stage, and summarize the overall KKT conditions for \eqref{eq:player's constrained problem} in the following theorem.
\begin{theorem}[Necessary Condition]\label{thm:necessary condition}
    Under Assumption \ref{assumption:feasible}, let $(\xvec^*,\uvec^*)$ be a local feedback Stackelberg equilibrium trajectory. Suppose that the Linear Independence Constraint Qualification (LICQ) \cite{nocedal1999numerical} and strict complementarity condition \cite{boyd2004convex} are satisfied at $(\xvec^*,\uvec^*)$. Furthermore, suppose $\{\pi_t^{i}\}_{t=0,i=1}^{T,N}$ is a set of local feedback Stackelberg policies and $\pi_t^i$ is differentiable around $(x_t^*,u_t^{1:(i-1)*})$, $\forall t\in\mathbf{I}_0^T$, $i\in\mathbf{I}_1^N$. The KKT conditions of \eqref{eq:player's constrained problem} can be formulated as, for all $i\in\mathbf{I}_1^N$, $t\in\mathbf{I}_0^T$,
    \begin{equation}\label{eq:KKT}
        \begin{aligned}
            0&= \nabla_{u_{\tau}^i}\mathcal{L}_t^i (x^*_{t:T+1}, u^*_{t:T}, \lambda_{t:T}, \mu_{t:T+1}, \gamma_{t:T+1}, \eta_{t:T-1},\psi_{t:T} )&&\forall \tau\in \mathbf{I}_{t}^T\\
            0&= \nabla_{x_{\tau}} \mathcal{L}_t^i (x^*_{t:T+1} u^*_{t:T}, \lambda_{t:T}, \mu_{t:T+1}, \gamma_{t:T+1}, \eta_{t:T-1},\psi_{t:T} ) &&\forall \tau \in \mathbf{I}_{t+1}^{T+1}\\
            0&= \nabla_{u_{t}^{j}}\mathcal{L}_t^i (x^*_{t:T+1}, u^*_{t:T}, \lambda_{t:T}, \mu_{t:T+1}, \gamma_{t:T+1}, \eta_{t:T-1},\psi_{t:T} )&& \forall j\in \mathbf{I}_{i+1}^N\\
            % 0&=u_{t}^{j*} - \pi_t^{j*}(x_t^*,u_t^{1:(j-1)*}),&& \forall j\in \mathbf{I}_{i+1}^N\\
            0&= \nabla_{u_{\tau}^{j}}\mathcal{L}_t^i (x^*_{t:T+1}, u^*_{t:T}, \lambda_{t:T}, \mu_{t:T+1}, \gamma_{t:T+1}, \eta_{t:T-1},\psi_{t:T} )&& \forall j\in\mathbf{I}_{1}^N\setminus\{i\}, \forall \tau\in \mathbf{I}_{t+1}^T\\
            % 0&=u_{\tau}^{j*} - \pi_\tau^{j*}(x_\tau^*, u_\tau^{1:(j-1)*}), &&\forall j \in \mathbf{I}_1^N\setminus\{i\}, \forall \tau \in \mathbf{I}_{t+1}^T\\
            0& = x^*_{\tau+1} - f_\tau(x^*_\tau,u^*_\tau) && \forall \tau\in \mathbf{I}_{t}^T\\
            0&=h_\tau^i(x^*_\tau,u^*_\tau) && \forall \tau\in \mathbf{I}_{t}^T\\ 
            0&\le \gamma_\tau^{i} \perp g_\tau^i(x^*_\tau,u^*_\tau)\ge 0 && \forall \tau\in \mathbf{I}_{t}^T\\
            0&= h_{T+1}^i(x^*_{T+1})\\
            0&\le \gamma_{T+1}^{i} \perp g_{T+1}^i(x^*_{T+1})\ge 0
\end{aligned}
    \end{equation}
    where %$\partial_x \ell(x)$ represents the sub-differential of a function $\ell(x)$ with respect to $x$ and 
    $\perp$ represents the complementary slackness condition \cite{boyd2004convex}. Then, there exists Lagrange multipliers %$\{\lambdavec,\muvec,\gammavec,\etavec,\psivec\}$, where 
    $\lambdavec:=[\lambda_t]_{t=0}^T$, $\muvec:=[\mu_t]_{t=0}^{T+1}$, $\gammavec: = [\gamma_t]_{t=0}^{T+1}$, $\etavec: = [\eta_t]_{t=0}^{T-1}$, and $\psivec:=[\psi_t]_{t=0}^T$, such that \eqref{eq:KKT} holds true.
\end{theorem}
\begin{proof}
    The proof can be found in the Appendix.
\end{proof}

Constructing the KKT conditions in \eqref{eq:KKT} requires the computation of policy gradients, $\{\nabla \pi_t^i\}_{t=0,i=1}^{T,N}$, which appear in the first four rows of \eqref{eq:KKT}. However, knowing the policy itself is not required, as any solution satisfying the KKT conditions obeys the corresponding feedback Stackelberg policy, as shown in the proof of Theorem~\ref{thm:necessary condition}. %, but not necessarily the policies $\{\pi_t^i\}_{t=0,i=1}^{T,N}$. 
A key distinction between \eqref{eq:KKT} and the FNE KKT conditions in \cite{laine2023computation} lies in the accommodation of a decision hierarchy among the $N$ players at each stage. This is reflected in the terms $- \sum_{j\in \mathbf{I}_{i+1}^N} \psi_t^{i,j\top} (u_t^j - \pi_t^j(x_t,u_t^{1:j-1}))$ in the Lagrangian $\mathcal{L}_t^i$. Additionally, this decision hierarchy differentiates the construction of the FSE KKT conditions from those of FNE. We will outline a detailed procedure for constructing the FSE KKT conditions in Sections~\ref{sec:LQ games} and \ref{sec:nonlinear games}, with an example provided in Appendix~\ref{appendix:KKT}. 

Furthermore, we propose a sufficient condition for feedback Stackelberg equilibrium trajectories in the following theorem.
\begin{theorem}[Sufficient Condition]\label{thm:sufficient condition}
    Let $(\xvec^*,\uvec^*)$ be a trajectory and $\{\pi_t^i\}_{t=0,i=1}^{T,N}$ be the associated policies. Suppose there exist Lagrange multipliers $\{\lambdavec, \muvec, \gammavec, \etavec, \psivec\}$ satisfying \eqref{eq:KKT} and there exists an $\epsilon>0$ such that, %the following condition holds true, 
    for all $i\in\mathbf{I}_1^N$, $t\in \mathbf{I}_0^T$, and nonzero $\{\Delta x_{T+1}\}\bigcup  \{\Delta x_\tau, \allowbreak \Delta u_\tau\}_{\tau=t}^T$ satisfying
    \begin{equation}\label{eq:sufficient condition critical cone}
        \begin{aligned}
            &0=\Delta u_t^j - \nabla \pi_t^j(x_t^*,u_t^{1:(j-1)*})\begin{bmatrix}\Delta x_t \\ \Delta u_t^{1:j-1} \end{bmatrix},\forall j\in \mathbf{I}_i^N\\
            &0=\Delta u_\tau^j - \nabla \pi_\tau^j(x_\tau^*,u_\tau^{1:(j-1)*})\begin{bmatrix}\Delta x_\tau \\ \Delta u_\tau^{1:j-1} \end{bmatrix},\forall j\in \mathbf{I}_1^N,\forall \tau\in\mathbf{I}_{t+1}^T\\
            &0= \Delta x_{\tau+1} - \nabla f_\tau(x_\tau^*, u_\tau^*) \begin{bmatrix}
                \Delta x_{\tau} \\ \Delta u_\tau 
            \end{bmatrix},\forall \tau\in \mathbf{I}_{t}^T \\
            &0 = \nabla h_\tau^j(x_\tau^*,u_\tau^*) \begin{bmatrix}
                \Delta x_\tau \\ \Delta u_\tau
            \end{bmatrix},0 = \nabla h_{T+1}^j(x_{T+1}^*) \Delta x_{T+1},\forall \tau \in \mathbf{I}_0^T, \forall j \in \mathbf{I}_1^N \\ 
            % &0 = \nabla h_{T+1}^j(x_{T+1}) \Delta x_{T+1},\forall j \in \mathbf{I}_1^N
        \end{aligned}
    \end{equation}
    we have 
    \begin{equation}
    \sum_{\tau=t}^{T} \begin{bmatrix}
                \Delta x_\tau \\ \Delta u_\tau^i
            \end{bmatrix}^\top \nabla^2_{[x_\tau^*,u_\tau^{i*}]} L_\tau^{i}\begin{bmatrix}
                \Delta x_\tau \\ \Delta u_\tau^i
            \end{bmatrix} + \Delta x_{T+1}^\top \nabla^2_{[x_{T+1}^*]} L_T^{i} \Delta x_{T+1}
            %+ \Delta x_{T+1}^\top \nabla^2_{x_{T+1}^*} L_{T+1}^i\Delta x_{T+1}
            >0.
    \end{equation}
    %$\sum_{\tau=t}^T \Delta u_\tau^{i\top} \nabla^2_{u_\tau^i} L_\tau^i \Delta u_\tau^i > 0.$
    % \begin{equation}\label{eq:sufficient condition}
    %     \begin{aligned}
    %         \sum_{\tau=t}^T \Delta u_\tau^{i\top} \nabla^2_{u_\tau^i} L_\tau^i \Delta u_\tau^i > 0.
            % &  \sum_{\tau=t}^T \begin{bmatrix}
            %     \Delta x_\tau \\ \Delta u_\tau^i
            % \end{bmatrix}^\top \nabla^2_{[x_\tau,u_\tau^i]} L_\tau^{i}\begin{bmatrix}
            %     \Delta x_\tau \\ \Delta u_\tau^i
            % \end{bmatrix} + \Delta x_{T+1}^\top \nabla^2_{x_{T+1}} L_{T+1}^i\Delta x_{T+1}>0.
            % % & \forall \{\Delta x_{T+1}\}\cup\{\Delta x_\tau, \Delta u_\tau\}_{\tau=t}^T:  \\
            % &\hspace{1cm} 0=\Delta u_t^j - \nabla \pi_t^i(x_t,u_t^{1:j-1})\begin{bmatrix}\Delta x_t \\ \Delta u_t^{1:j-1} \end{bmatrix},\forall j\in \mathbf{I}_{i}^N\\
            % &\hspace{1cm} 0=\Delta u_\tau^j - \nabla \pi_\tau^i(x_\tau,u_\tau^{1:j-1})\begin{bmatrix}\Delta x_\tau \\ \Delta u_\tau^{1:j-1} \end{bmatrix},\forall j\in \mathbf{I}_{1}^N,\forall \tau\in\mathbf{I}_{t+1}^T\\
            % &\hspace{1cm} 0= \Delta x_{\tau+1} - \nabla f_\tau(x_\tau, u_\tau) \begin{bmatrix}
            %     \Delta x_{\tau} \\ \Delta u_\tau 
            % \end{bmatrix},\forall \tau\in \mathbf{I}_{t}^T \\
            % &\hspace{1cm} 0 = \nabla h_\tau^i(x_\tau,u_\tau) \begin{bmatrix}
            %     \Delta x_\tau \\ \Delta u_\tau
            % \end{bmatrix},\forall \tau \in \mathbf{I}_0^T, \forall i \in \mathbf{I}_1^N \\ 
            % &\hspace{1cm} 0 = \nabla h_{T+1}^i(x_{T+1}) \Delta x_{T+1},\forall i \in \mathbf{I}_1^N 
    %     \end{aligned}
    % \end{equation}
    Then, %$\{\pi_t^i\}_{t=0,i=1}^{T,N}$ is a set of local feedback Stackelberg policies and 
    $(\xvec^*,\uvec^*)$ constitutes a local feedback Stackelberg equilibrium trajectory.
\end{theorem}
\begin{proof}
    The proof can be found in the Appendix.
\end{proof}
\begin{remark}
    The gap between the necessity condition in Theorem~\ref{thm:necessary condition} and the sufficiency condition in Theorem~\ref{thm:sufficient condition} is due to the fact that a solution to \eqref{eq:KKT} may not necessarily be a feedback Stackelberg equilibrium, and that there exist feedback Stackelberg equilibria where the cost functions possess zero second-order gradients. 
    %Theorem~\ref{thm:sufficient condition} can be relaxed if we further restrict $\{\Delta x_{T+1}\}\bigcup\allowbreak \{\Delta x_\tau, \Delta u_\tau\}_{\tau=t}^T$ to lie at the intersection of the feasible set and the null space of the Jacobian of all the inequality constraints.  
\end{remark}

Theorems~\ref{thm:necessary condition} and \ref{thm:sufficient condition} establish conditions to certify whether a trajectory $ (\xvec ,\uvec)$ constitutes a feedback Stackelberg equilibrium with a set of feedback Stackelberg policies $\{\pi_t^i\}_{t=0,i=1}^{T,N}$. 
However, computing feedback Stackelberg equilibria %such policies $\{\pi_t^i\}_{t=0,i=1}^{T,N}$ 
can be challenging. %When we do not have those given policies $\{\pi_t^i\}_{t=0,i=1}^{T,N}$, the computation of them is not trivial. 
In the following sections, we will discuss how to approximately compute local feedback Stackelberg equilibria. We will first compute feedback Stackelberg equilibria for Linear Quadratic games and then extend the result to nonlinear games.

\section{Constrained Linear Quadratic Games}\label{sec:LQ games}

We consider the linear dynamics
\begin{equation}\label{eq:linear dynamics}
    x_{t+1} =f_t(x_t,u_t)= A_t x_t + B_t^1 u_t^1 + \dots + B_t^N u_t^N+c_t,\ t\in \mathbf{I}_0^T,
\end{equation}
where $A_t\in\mathbb{R}^{n\times n}$, $B_t^i\in\mathbb{R}^{n\times {m_i}}$ and $c_t\in\mathbb{R}^{n}$. We denote by $B_t: = [B_t^1,B_t^2,\dots,B_t^N]$. The cost of the $i$-th player is defined as
\begin{equation}\label{eq:define lq game costs}
    \begin{aligned}
        \ell_t^i(x_t,u_t)=&\frac{1}{2}\begin{bmatrix}
            x_t \\ u_t
        \end{bmatrix}^\top \begin{bmatrix}
            Q_t^{i} & S_t^{i\top} \\ S_t^i & R_t^i
        \end{bmatrix} \begin{bmatrix}x_t \\ u_t \end{bmatrix} + q_t^{i\top}x_t + r_t^{i\top}u_t,\ t\in\mathbf{I}_0^T,\\
        \ell_{T+1}^i(x_{T+1})=& \frac{1}{2} x_{T+1}^\top Q_{T+1}^i x_{T+1} + q_{T+1}^{i\top} x_{T+1},
    \end{aligned}
\end{equation}
where symmetric matrices $Q_t^i\in\mathbb{R}^{n\times n}$ and $R_t^i\in\mathbb{R}^{m\times m}$ are positive semidefinite and positive definite, respectively. The off-diagonal matrix is denoted as $S_t^i\in\mathbb{R}^{m\times n}$. In particular, we partition the structure of $R_t^i$, $S_t^i$ and $r_t^i$ as follows
\begin{equation}
    R_t^i = \begin{bmatrix}
        R_t^{i,1,1} & R_t^{i,1,2} & \cdots & R_t^{i,1,N} \\ 
        R_t^{i,2,1} & R_t^{i,2,2} & \cdots & R_t^{i,2,N} \\ 
        \vdots & \vdots & \ddots & \vdots\\ 
        R_t^{i,N,1} & R_t^{i,N,2} & \cdots & R_t^{i,N,N}
    \end{bmatrix}, S_t^{i}=\begin{bmatrix}
        S_t^{i,1} \\ S_t^{i,2} \\ \vdots \\  S_t^{i,N} 
    \end{bmatrix}, r_t^i = \begin{bmatrix}
        r_t^{i,1} \\ r_t^{i,2} \\ \vdots \\ r_t^{i,N}
    \end{bmatrix},
\end{equation}
where $R_t^{i,j,k}$, $S_t^{i,j}$ and $r_t^i$ represent the cost terms $u_t^{j\top} R_t^{i,j,k} u_t^{k}$, $u_t^{j\top} S_t^{i,j} x_t$ and $r_t^{i,j^\top} u_t^j$ in $\ell_t^i(x_t,u_t)$. The linear equality and inequality constraints are specified as,
\begin{equation}\label{eq:define lq game constraints}
\begin{aligned}
    &0 = h_t^i(x_t,u_t) = H_{x_t}^ix_t +\sum_{j\in\mathbf{I}_1^N} H_{u_t^j}^i u_t^j + \bar{h}_t^i, &&t \in \mathbf{I}_0^T\\ 
    &0\le g_t^i(x_t,u_t) = G_{x_t}^ix_t + \sum_{j\in\mathbf{I}_1^N} G_{u_t^j}^i u_t^j + \bar{g}_t^i, && t \in \mathbf{I}_0^T\\
    &0 = h_{T+1}^i(x_{T+1}) = H_{x_{T+1}}^ix_{T+1} + \bar{h}_{T+1}^i,\\ 
    &0 \le g_{T+1}^i(x_{T+1})=G_{x_{T+1}}^i x_{T+1} + \bar{g}_{T+1}^i.
\end{aligned}
\end{equation}
\subsection{Computing Feedback Stackelberg Equilibria and Constructing the KKT Conditions for LQ Games}
In this subsection, we introduce a process for deriving FSE and the KKT conditions for LQ games. % for players $i\in\mathbf{I}_1^N$ at stages $t\in\mathbf{I}_0^T$. %Similar to Linear Quadratic Regulator problem, the optimal control policy should also be linear and differentiable in unconstrained LQ games. However, 
When we have linear inequality constraints, the optimal policies of LQ games are generally piecewise linear functions of the state \cite{besselmann2012explicit,laine2023computation}. %, which could also hold true in LQ games \cite{laine2023computation}. 
However, this makes them non-differentiable at the facets. % of a piecewise linear policy. 
In our work, we propose to use the primal-dual interior point (PDIP) method \cite{nocedal1999numerical} to solve constrained LQ games. The benefits of using PDIP are its polynomial complexity and tolerance of infeasible initializations. Critically, under certain conditions, PDIP yields a local differentiable policy approximation to the ground truth piecewise linear policy, as shown in the rest of this section and an example in Appendix~\ref{subsec:appendix quasi policy gradient}. %and we can guarantee the differentiability in the following proposition. 

% \begin{proposition}[Differentiability of PDIP strategies]
%     Consider a 1-stage LQR problem
%     \begin{equation}
%     \begin{aligned}
%         \min_{u_0}\ &\frac{1}{2} x_0^\top Q_0x_0 + u_0^\top R_0 u_0 + x_{1}^\top Q_1 x_1 \\ 
%         \textrm{s.t. }& x_{1}  = A x_0 + Bu_0\\
%         & H_{x_0} x_0 + H_{u_0} u_0 - h_0 = 0\\
%         & G_{x_0} x_0 + G_{u_0} u_0 - g_0 \ge 0\\
%         & H_{x_1} x_1 - h_1 = 0 \\
%         & G_{x_1} x_1 - g_1 \ge 0
%     \end{aligned}
%     \end{equation}
%     and its KKT condition
%     Then, the optimal policy $\pi_0:\mathbb{R}^n\to \mathbb{R}^m$ is a continuously differentiable function. 
% \end{proposition}
% \begin{proof}
    
% \end{proof}

To this end, we introduce a set of non-negative slack variables $\{s_t^i\}_{t=0,i=1}^{T+1,N}$ such that we can rewrite the inequality constraints as equality constraints for $t\in \mathbf{I}_0^{T+1}$ and $i\in\mathbf{I}_1^N$,
\begin{equation}
    \begin{aligned}
        g_t^i(x_t,u_t)-s_t^i=0, \ g_{T+1}^i(x_{T+1})-s_{T+1}^i = 0.
    \end{aligned}
\end{equation}
% We adopt the primal-dual interior point (PDIP) method such that inequality constraints can be transformed to equality constraints. To be more specific, in PDIP, we introduce slack variables such that
% \begin{equation}
%     g_t^i(x_t,u_t) - s_t^i = 0,\ s_t^i\ge 0, \ \gamma_t^i\ge 0.
% \end{equation}
% To ensure that $s_t^i$ is non-negative, we introduce a parameter $\rho>0$ and add a log constraint penalization term $-\rho\log(s_t^i)$ to the cost function $\ell_t^i(x_t,u_t)$ of each player $i\in\mathbf{I}_1^N$ at each stage $t\in\mathbf{I}_0^T$. 
In this paper, we consider PDIP as a homotopy method as in \cite{nocedal1999numerical}. Instead of solving the mixed complementarity problem \eqref{eq:KKT} directly, we seek solutions to the homotopy approximation of the complementary slackness condition
\begin{equation}\label{eq:homotopy}
    \gamma_t^{i}\odot s_t^i = \rho \mathbf{1}, \ s_t^i\ge 0, \ \gamma_t^i\ge 0
\end{equation}
where $\odot$ denotes the elementwise product and $\rho>0$ is a hyper-parameter to be reduced to 0 gradually such that we recover the ground truth solution when $\rho\to 0$. In the following section, we will construct the KKT conditions where we replace the mixed complementarity condition with its approximation \eqref{eq:homotopy}. For each $\rho>0$, we denote its corresponding local feedback policy as $\{\pitrhoi\}_{t=0,i=1}^{T,N}$, if it exists.

As shown in Theorem~\ref{thm:necessary condition}, the construction of the KKT conditions for player $i$ at stage $t$ requires the policy gradients of subsequent players at the current stage and future stages. In what follows, we construct those KKT conditions in reverse player order and backward in time. %a player's policy gradient could used to construct prior players' KKT conditions. 

% In what follows, we will first construct player $N$'s KKT condition at time $T$, and then derive its policy gradient $\nabla\piTrhoN$. We then use $\nabla \piTrhoN$ to construct KKT condition of the $(N-1)$-th player at the $T$-th stage. The policy gradient $\nabla \piTrho^{N-1}$ is then used to define the KKT condition of the $(N-2)$-th player. We continue this process until we complete the construction of the KKT condition for the first player at time $T$. Subsequently, we will use $\nabla \piTrho^1$ to define the $N$-th player's KKT condition at the stage $t=T-1$. Following a similar construction process as in the $T$-th stage, we will construct all the players' KKT conditions at the stage $t=T-1$. We will continue this construction process until we construct the KKT condition of the first player at the $0$-th stage.

\subsubsection{Player \texorpdfstring{$N$}{N} at the \texorpdfstring{$T$}{T}-th stage}\label{sec:N,T,KKT}
Before constructing the KKT conditions, we first introduce the variables of player $N$ at the terminal time $T$, $\zvec_{T}^N: = [u_T^N,\lambda_T^N, \mu_{T:T+1}^N,\gamma_{T:T+1}^N,\allowbreak s_{T:T+1}^N,x_{T+1}]$. As shown in Theorem~\ref{thm:necessary condition}, the KKT conditions of player $N$ at time $T$ can be written as
\begin{equation}\label{eq:T,N,KKT}
    0=\KTNrho(\zvec_T^N): = \begin{bmatrix}
    \nabla_{u_T^N}L_{T}^N   \\
    \nabla_{x_{T+1}}L_{T}^N \\
    x_{T+1} - f_T(x_T,u_T) \\
    h_T^N(x_T,u_T)\\
    h_{T+1}^N(x_{T+1}) \\
    % \gamma_{T+1}^N\odot s_{T+1}^N - \rho \mathbf{1}  \\
    g_T^N(x_T,u_T) - s_{T}^N\\
    g_{T+1}^N(x_{T+1})-s_{T+1}^N\\
    \gamma_{T:T+1}^N\odot s_{T:T+1}^N - \rho \mathbf{1} 
\end{bmatrix},
\end{equation}
where the rows of $\KTNrho(\zvec_T^N)$ represent the stationarity conditions with respect to $u_T^N$ and $x_{T+1}$, dynamics constraint, equality constraints, inequality constraints, and relaxed complementarity conditions. To obtain a local policy and its policy gradient around a $\zvec_T^N$ satisfying \eqref{eq:T,N,KKT}, we build a first-order approximation to \eqref{eq:T,N,KKT},
\begin{equation}\label{eq:T,N-th player}
\nabla \KTNrho \cdot \Delta \zvec_T^N + \nabla_{[x_T,u_T^{1:N-1}]}\KTNrho  \cdot \begin{bmatrix}
    \Delta x_T\\ \Delta u_T^{1:N-1}
\end{bmatrix} + \KTNrho(\zvec_T^N) = 0.  
\end{equation}
If there is no solution $\Delta \zvec_T^N$ to \eqref{eq:T,N-th player}, then we claim there is no feedback Stackelberg policy. Suppose \eqref{eq:T,N-th player} has a solution $\Delta \zvec_T^N$, then we can define $\Delta \zvec_T^N$ as % to \eqref{eq:T,N-th player} as
\begin{equation}\label{eq:Delta z_T^N}
    \begin{aligned}
        \Delta \zvec_T^N=\underbrace{-\big(\nabla \KTNrho\big)^{+}  \cdot \Big( \nabla_{[x_T,u_T^{1:N-1}]} \KTNrho\cdot \begin{bmatrix}
    \Delta x_T \\ \Delta u_T^{1:N-1}
\end{bmatrix} +\KTNrho(\zvec_T^N)\Big)}_{F_T^N(\Delta x_T, \Delta u_T^{1:N-1})},
    \end{aligned}
\end{equation}
where $(\cdot)^+$ represents the pseudo-inverse and we denote $\Delta \zvec_T^N$ as a function $F_T^N$ of $(\Delta x_T, \Delta u_T^{1:N-1})$.
% From the above equation, we observe that $\Delta \zvec_T^N$ is a function of $(\Delta x_T, \Delta u_T^{1:N-1})$. We can represent \eqref{eq:Delta z_T^N} compactly by defining a function $F_T^N : (\Delta x_T, \Delta u_T^{1:N-1})\to \Delta \zvec_T^N$,
% \begin{equation}
%     \Delta \zvec_T^N = F_T^N(\Delta x_T, \Delta u_T^{1:N-1}).
% \end{equation}
Since $\Delta u_T^N$ represents the first $m_N$ entries of $\Delta \zvec_T^N$, we consider $\Delta u_T^N$ as a function of $(\Delta x_T,\Delta u_T^{1:N-1})$,% as well,
\begin{equation}\label{eq:T,N,u}
\begin{aligned}
\Delta u_T^N=& -\big[\big(\nabla \KTNrho\big)^{+}  \big]_{u_T^N} \cdot \Big( \nabla_{[x_T,u_T^{1:N-1}]} \KTNrho\cdot  \begin{bmatrix}
    \Delta x_T \\ \Delta u_T^{1:N-1}
\end{bmatrix} +\KTNrho(\zvec_T^N)\Big)   , \\
    % =& \pi_T^N(\Delta x_T, \Delta u_T^{1:N-1}),
\end{aligned}
\end{equation}
where $\big[\big(\nabla \KTNrho\big)^{+}  \big]_{u_T^N}$ represents the rows of the matrix $\big(\nabla \KTNrho\big)^{+}$ corresponding to the variable $u_T^N$, i.e., the first $m_N$ rows of the matrix $\big(\nabla \KTNrho\big)^{+}$.

Furthermore, for some $x\in\mathbb{R}^n$ and $u^{1:N-1} \in\mathbb{R}^{\sum_{i=1}^{N-1}m_i}$, let $\Delta x_T = x - x_T$, $\Delta u_T^{1:N-1} = u^{1:N-1} - u_T^{1:N-1}$ and $\Delta u_T^N = u^N - u_T^N$. Substituting them into \eqref{eq:T,N,u}, we obtain a local policy $\localpiTrhoN$ for player $N$ at time $T$,
\begin{equation}\label{eq:T,N,construct policy}
\begin{aligned}
    u^N =& \localpiTrhoN(x,u^{1:N-1}) \\
    :=& u_T^N -\big[\big(\nabla \KTNrho\big)^{+}  \big]_{u_T^N} \cdot \Big( \nabla_{[x_T,u_T^{1:N-1}]} \KTNrho \cdot \begin{bmatrix}
    x - x_T \\ u^{1:N-1} - u_T^{1:N-1}
\end{bmatrix} +\KTNrho(\zvec_T^N)\Big) .
\end{aligned}
\end{equation}
Suppose that $\nabla  \KTNrho(\zvec_T^N)$ has a constant row rank in an open set containing $\zvec_T^N$, then, by the constant rank theorem \cite{janin1984directional}, the policy $\localpiTrhoN$ of player $N$ at time $T$ is locally differentiable with respect to $(x,u^{1:N-1})$, and its gradient over $(x,u^{1:N-1})$ is
\begin{equation}\label{eq:T,N,policy gradient}
\begin{aligned}
\begin{aligned}
\nabla \localpiTrhoN  = & -\big[\big(\nabla \KTNrho\big)^{+}  \big]_{u_T^N} \cdot \nabla_{[x_T,u_T^{1:N-1}]} \KTNrho.
% & + \left[-\big[\big(\nabla_{\zvec_T^N}K_T^N\big)^{+}  \big]_{u_T^N}  \nabla_{[x_T,u_T^{1:N-1}]} K_T^N\right]_{x_T} \\
% \nabla_{u^i} \pi_T^N  = & \left[-\big[\big(\nabla_{\zvec_T^N}K_T^N\big)^{+}  \big]_{u_T^N}  \nabla_{[x_T,u_T^{1:N-1}]} K_T^N\right]_{u_T^i},\ i\in\mathbf{I}_1^{N-1}
\end{aligned}
\end{aligned}
\end{equation}
In the following subsection, we construct the KKT conditions of a player $i<N$ at stage $T$. %the policy gradient $\nabla \localpiTrhoN$ will be used to construct KKT condition of the $(N-1)$-th player at time $T$. 

%There is no error at the stage $t=T$. This means that $\pi_{T-1}^i$ is eqaul to the high-order policy. However, when we take gradient of $\pi_{T-1}^i$, we observe that $\nabla \pi_{T-1}^i = \nabla \Big[\big( \nabla^2 K_T^N \big)^{+}\Big]$

\subsubsection{Players \texorpdfstring{$i<N$}{i<N} at the \texorpdfstring{$T$}{T}-th stage}
For player $i < N$, assuming that $\zvec_T^{i+1}$ has been defined and $\nabla \tilde{\pi}_{T,\rho}^{i+1}$ has been computed, we first introduce variables 
\begin{equation}\label{eq:define z_T^i by y_T^i}
    \yvec_T^i: = [u_T^i, \psi_T^{i}, \lambda_T^i,  \mu_{T:T+1}^i,\allowbreak\gamma_{T:T+1}^i, s_{T:T+1}^i ] \textrm{ and } \zvec_T^i:=[\yvec_T^i,\zvec_T^{i+1}].
\end{equation}
The KKT conditions of player $i$ at time $T$ is
\begin{equation}\label{eq:T,i-th KKT}
    0=\KTirho(\zvec_T^{i}):=\begin{bmatrix}
    \hatKTirho(\yvec_T^i) \\
    \KTnextirho(\zvec_T^{i+1})
\end{bmatrix},\ \hatKTirho(\yvec_T^i):=\begin{bmatrix}
    \nabla_{u_T^i}L_{T}^i   \\
    \nabla_{x_{T+1}}L_{T}^i \\
    \nabla_{u_T^j} L_T^i, \forall j\in \mathbf{I}_{i+1}^N\\
    h_T^i(x_T,u_T)\\
    h_{T+1}^i(x_{T+1}) \\
    % \gamma_{T+1}^i\odot s_{T+1}^i - \rho \mathbf{1}  \\
    g_T^i(x_T,u_T) - s_{T}^i\\
    g_{T+1}^i(x_{T+1})-s_{T+1}^i\\
    \gamma_{T:T+1}^i\odot s_{T:T+1}^i - \rho \mathbf{1}
\end{bmatrix},
\end{equation}
% \begin{equation}
%     0=\KTirho(\zvec_T^{i}):=\begin{bmatrix}
%     \nabla_{u_T^i}L_{T}^i   \\
%     \nabla_{x_{T+1}}L_{T}^i \\
%     \nabla_{u_T^j} L_T^i, \forall j\in \mathbf{I}_{i+1}^N\\
%     h_T^i(x_T,u_T)\\
%     h_{T+1}^i(x_{T+1}) \\
%     % \gamma_{T+1}^i\odot s_{T+1}^i - \rho \mathbf{1}  \\
%     g_T^i(x_T,u_T) - s_{T}^i\\
%     g_{T+1}^i(x_{T+1})-s_{T+1}^i\\
%     \gamma_{T:T+1}^i\odot s_{T:T+1}^i - \rho \mathbf{1} \\
%     \KTnextirho(\zvec_T^{i+1})\end{bmatrix}
% \end{equation}
where the definition of $L_T^i$ involves the policy $\piTrhoinext$, as shown in \eqref{eq:Lagrangian of player i, time T}. Building a first-order approximation to $0=\KTirho(\zvec_T^i)$, we have
\begin{equation}\label{eq:T,i,Newton}
    \nabla \KTirho \cdot \Delta \zvec_T^i + \nabla_{[x_T, u_T^{1:i-1}]} \KTirho \cdot \begin{bmatrix}
        \Delta x_T \\ \Delta u_T^{1:i-1}
    \end{bmatrix} + \KTirho(\zvec_T^i) = 0.
\end{equation}

However, a drawback of PDIP is that the policy $\piTrhoinext$ is nonlinear in state $x_T$ and prior players' controls $u_T^{1:i}$, as shown in a simplified problem in Appendix~\ref{subsec:appendix quasi policy gradient}. The computation of $\nabla \KTirho$ involves the evaluation of $\nabla(\nabla_{u_T^i} L_T^i)$, which requires the computation of $\nabla (\psi_T^{i,i+1} \nabla \pi_{T,\rho}^{i+1}) = \nabla \psi_T^{i,i+1} \cdot \nabla \pi_{T,\rho}^{i+1} + \psi_T^{i,i+1} \cdot \nabla^2 \pi_{T,\rho}^{i+1}$. Furthermore, to evaluate $\nabla^2 \pi_{T,\rho}^{i+1}$, we need the computation of $\nabla^3 \pi_{T,\rho}^{i+2}$. In other words, the construction of $\nabla \KTirho$ needs the evaluation of $\nabla^2 \pi_{T,\rho}^{i+1}$, $\nabla^3 \pi_{T,\rho}^{i+2}$, ..., and $\nabla^{N-i+1} \pi_{T,\rho}^{N}$. The evaluation of high-order policy gradients is challenging in practice \cite{laine2023computation} because there is no closed-form solution to the KKT equation $0=\KTnextirho(\zvec_T^{i+1})$. 

We prove in Appendix~\ref{subsec:appendix quasi policy gradient} that the high-order policy gradients could decay to zero as $\rho \to0$, when the ground truth policy is piecewise linear and differentiable around $x_T$. Motivated by this observation, we propose to approximate the nonlinear policy $\pi_{T,\rho}^{i+1}$ by its first-order approximation $\tilde{\pi}_{T,\rho}^{i+1}$ in \eqref{eq:T,N,construct policy}. With this approximation, we have $\nabla (\psi_T^{i,i+1} \nabla \tilde{\pi}_{T,\rho}^{i+1}) = \nabla \psi_T^{i,i+1} \cdot \nabla \tilde{\pi}_{T,\rho}^{i+1}$. We refer to this policy $\tilde{\pi}_{T,\rho}^{i+1}$ as a \emph{\textbf{\quasi}}.% approximation. 

% \begin{definition}[Quasi Feedback Stackelberg Equilibrium]
%     We say a set of linear policies $\{\tilde{\pi}_{t}^i\}_{t=0,i=1}^{T,N}$ is a set of Quasi Stackelberg policies if there exist variables $\{x,u,\lambda,\mu,\gamma,\eta,\psi,\svec\}$ such that \eqref{eq:KKT} is satisfied. Moreover, $\{\xvec, \uvec\}$ constitutes a Quasi Feedback Stackelberg Equilibrium trajectory.
% \end{definition}

In the remainder of this section, we will always approximate the ground truth nonlinear policy by \quasi{} when we define the KKT conditions. 
% \begin{remark}
%     It is difficult to characterize the high-order policy gradients of $\pi_{T,\rho}^{N}$ %$\|\nabla \pi_t^i - \nabla \pi_{t,\rho}^i\|$ 
%     because there is no closed-form solution to the KKT equation $K_\rho(\zvec_\rho)=0$. To better understand this policy gradient error bound, we characterize a simplified case of constrained Linear Quadratic Regulator example in the Appendix, which suggests that the policy gradient decays to zero as $\rho\to 0$ at each $\zvec_\rho$ where $\pi_t^i$ is differentiable. We show in the Appendix that quasi policy provides a good first-order approximation to the ground truth nonlinear policy $\pi_{T,\rho}^N$ in a simplified example, because the high-order policy gradients of $\pi_{T,\rho}^N$ decays to zero as $\rho \to0$.
% \end{remark}

% we have
% \begin{equation}\label{eq:T,i-th player}
% \begin{aligned}
% \nabla_{\zvec_T^i}K_T^i \cdot \Delta \zvec_T^i + \nabla_{[x_T,u_T^{1:i-1}]}K_T^i \cdot \begin{bmatrix}
%     \Delta x_T\\ \Delta u_T^{1:i-1}
% \end{bmatrix} + K_T^i(\zvec_T^i) = 0
% \end{aligned}
% \end{equation}
% % Notice that when constructing $K_T^i(\zvec_T^i)$, $i<N$, the dynamics constraint $0=x_{T+1} - f_T(x_T,u_T)$ has already been considered in $K_T^N(\zvec_T^N)$. 
% Suppose that there exists a solution $\Delta \zvec_T^i$ to \eqref{eq:T,i-th player}, then we have
% \begin{equation}
%     \Delta \zvec_T^i=-\big[\big(\nabla_{\zvec_T^i}K_T^i\big)^{+}\big]\cdot \big(\nabla_{[x_T,u_T^{1:i-1}]}K_T^i \cdot \begin{bmatrix}
%         \Delta x_T \\ \Delta u_T^{1:i-1}
%     \end{bmatrix} + K_T^i(\zvec_T^i)\big)
% \end{equation}

Solving equation \eqref{eq:T,i,Newton}, we can obtain $\Delta \zvec_T^i$ and $\nabla \localpiTrhoi$ as in \eqref{eq:Delta z_T^N} and \eqref{eq:T,N,policy gradient}, respectively. However, by construction, the dimension of $\Delta \zvec_T^i$ is higher than $\Delta \zvec_T^{i+1}$. Therefore, it is more expensive to compute $(\nabla \KTirho)^+$ than $(\nabla \KTnextirho)^+$, and it is worthwhile to reduce the complexity of computing $\Delta \zvec_T^i$ by leveraging the computation that we have done for $\Delta \zvec_T^{i+1}$ and $\nabla \localpiTrho^{i+1}$. To this end, by exploiting the structure $\zvec_T^i=[\yvec_T^i, \zvec_T^{i+1}]$ in \eqref{eq:define z_T^i by y_T^i}, we can rewrite \eqref{eq:T,i,Newton} as, 
% In what follows, we will show how we can reduce the computational complexity by using the information we got when we solve $\{\nabla \pi_T^{i+1},\dots, \nabla \pi_T^N\}$. To be more specific, we can rewrite \eqref{eq:T,i-th KKT} by exploiting the structure $\zvec_T^i=[\yvec_T^i, \zvec_T^{i+1}]$, 
\begin{equation}\label{eq:T,i,decompose}
    \left\{\begin{aligned}
        &\nabla \hatKTirho \cdot \Delta \yvec_T^i + \nabla_{[x_T,u_T^{1:i-1}]} \hatKTirho \cdot \begin{bmatrix}
            \Delta x_T \\ \Delta u_T^{1:i-1}
        \end{bmatrix} + \hatKTirho(\yvec_T^i) = 0\\
        &\nabla \KTnextirho  \cdot \Delta \zvec_T^{i+1} + \nabla_{[x_T,u_T^{1:i}]} \KTnextirho \cdot \begin{bmatrix}
            \Delta x_T \\ \Delta u_T^{1:i}
        \end{bmatrix} + \KTnextirho(\zvec_T^{i+1}) =0
    \end{aligned}\right.
\end{equation}
Observe that we have solved the second equation of \eqref{eq:T,i,decompose} in Section \ref{sec:N,T,KKT}. What remains to be solved is the first equation in \eqref{eq:T,i,decompose}. We solve it as follows, %To solve $\Delta \zvec_T^i$ in \eqref{eq:T,i,Newton}, instead of computing $(\nabla K_T^i)^+$, we only need to compute $(\nabla \hat{K}_T^i)^+$, which has a smaller dimension than the former one. We have
\begin{equation}\label{eq:Delta y T i}
\begin{aligned}
    \Delta \yvec_T^i =& \underbrace{- \big( \nabla \hatKTirho  \big)^+ \cdot \Big(\nabla_{[x_T, u_T^{1:i-1}]} \hatKTirho  \cdot \begin{bmatrix}
        \Delta x_T \\ \Delta u_T^{1:i-1}
    \end{bmatrix} + \hatKTirho (\yvec_T^i)\Big)}_{\hat{F}_T^i (\Delta x_T, \Delta u_T^{1:i-1})},\\
    \Delta u_T^i =& - \big[\big( \nabla \hatKTirho \big)^+\big]_{u_T^i} \cdot \Big(\nabla_{[x_T, u_T^{1:i-1}]} \hatKTirho \cdot \begin{bmatrix}
        \Delta x_T \\ \Delta u_T^{1:i-1}
    \end{bmatrix} + \hatKTirho (\yvec_T^i)\Big), \\
    \nabla \localpiTrhoi=& -\big[\big(\nabla \hatKTirho\big)^{+}  \big]_{u_T^i} \cdot \nabla_{[x_T,u_T^{1:i-1}]} \hatKTirho.
\end{aligned}
    \end{equation}
% and 
% \begin{equation}\label{eq:T,i,u}
%     \Delta u_T^i = - \big[\big( \nabla \hat{K}_T^i \big)^+\big]_{u_T^i} \cdot \Big(\nabla_{[x_T, u_T^{1:i-1}]} \hat{K}_T^i \cdot \begin{bmatrix}
%         \Delta x_T \\ \Delta u_T^{1:i-1}
%     \end{bmatrix} + \hat{K}_T^i(\yvec_T^i)\Big)
% \end{equation}
Combining \eqref{eq:Delta y T i} and \eqref{eq:Delta z_T^N}, we have
\begin{equation}\label{eq:Delta Z_T^i expression}
    \begin{aligned}
        \Delta \zvec_T^i =\begin{bmatrix}
            \Delta \yvec_T^i \\ \Delta \zvec_T^{i+1} 
        \end{bmatrix}= \begin{bmatrix}
            \hat{F}_T^i(\Delta x_T, \Delta u_T^{1:i-1})\\
            F_T^{i+1}(\Delta x_T, \Delta u_T^{1:i})
        \end{bmatrix}.
    \end{aligned}
\end{equation}
Since $\Delta u_T^i$ is also a function of $(\Delta x_T,\Delta u_T^{1:i-1})$, as shown in \eqref{eq:Delta y T i}, we can represent \eqref{eq:Delta Z_T^i expression} compactly as $\Delta \zvec_T^i = F_T^i(\Delta x_T, \Delta u_T^{1:i-1})$.
% where the challenging part is only the computation of $(\nabla \hat{K}_T^i)^+$ when solving $\Delta \yvec_T^i$, but it is easier to compute than $(\nabla K_T^i)^+$ because of its lower dimension.
% can represent the relationship between $\Delta \zvec_T^i$ and $(\Delta x_T,\Delta u_T^{1:i-1})$ by a function $F_T^i$, i.e., 
% \begin{equation}\label{eq:Delta z T i}
%     \begin{aligned}
%         \Delta \zvec_T^i = \begin{bmatrix}
%             \Delta \yvec_T^i \\ \Delta \zvec_T^{i+1}
%         \end{bmatrix} = F_T^i(\Delta x_T,\Delta u_T^{1:i-1})
%     \end{aligned}
% \end{equation}

As such, given that the KKT conditions of player $(i+1)$ at time $T$ have been constructed, we have finished the construction of the KKT conditions for player $i$ %$i\in\mathbf{I}_1^{N-1}$ 
at time $T$, and we introduced a computationally efficient way to compute $\nabla \localpiTrhoi$. 
We can derive the KKT conditions and \quasigradient{} of player $i<N$ at time $T$, sequentially, from $i = N-1$ to $i=1$. %Once constructing the KKT condition $0 = K_T^i(\zvec_T^i)$, we can compute the policy gradient $\nabla \pi_T^i$ as in \eqref{eq:Delta y T i}. 
%We then construct the $(i-1)$-th player's KKT condition until completing the construction of the first player's KKT condition $0 = K_T^1(\zvec_T^1)$.

\subsubsection{Player \texorpdfstring{$N$}{N} at a stage \texorpdfstring{$t<T$}{t<T}}
% The policy gradient of players at the $T$-th stage can be then used to construct the KKT conditions for players at time $t<T$. Specifically, 
At a stage $t<T$, assuming that we have constructed the KKT conditions $0 = \Knexttonerho(\zvec_{t+1}^1)$, %assuming that the policy gradients for all the players at the $(t+1)$-th stage have been computed, 
we are ready to derive the KKT conditions for player $N$ at time $t$. We first introduce the variable $\zvec_t^N:=[\yvec_t^N,\zvec_{t+1}^1]$, with $\yvec_t^N := [u_t^N,\eta_{t}^N, \lambda_t^N,\mu_t^N, \gamma_t^N,\allowbreak s_t^N,x_{t+1}]$. We construct the KKT conditions of player $N$ at time $t$ as follows,
\begin{equation}
    0=\KtNrho (\zvec_t^N): = \begin{bmatrix}
        \nabla_{u_t^N} L_t^N \\
        \nabla_{x_{t+1}}L_{t}^N \\
        \nabla_{u_{t+1}^j} L_t^N,\forall j\in \mathbf{I}_1^{N-1}  \\
        x_{t+1} - f_t(x_t,u_t)\\
        h_t^N(x_t,u_t)\\
        g_t^N(x_t,u_t) - s_t^N\\
        \gamma_{t}^N\odot s_{t}^N - \rho \mathbf{1}\\
        \Knexttonerho (\zvec_{t+1}^1)
    \end{bmatrix}.
\end{equation}
Building a first-order approximation to the above equation, 
% we have
% \begin{equation}\label{eq:t,N-th player}
%     \begin{aligned}
%         \nabla_{\zvec_t^N} K_t^N \cdot \Delta \zvec_t^N + \nabla_{[x_t,u_t^{1:N-1}]} K_t^N\cdot \begin{bmatrix}
%             \Delta x_t \\ \Delta u_t^{1:N-1}
%         \end{bmatrix} + K_t^N(\zvec_t^N) = 0
%     \end{aligned}
% \end{equation}
we can obtain \quasigradient{} $\nabla \localpitrhoN$ as in \eqref{eq:Delta y T i} when it exists. % when \eqref{eq:t,N-th player} is feasible. 
% Furthermore, we can construct the policy of the player $i<N$ at stage $t$. %, under the condition that the policies $\pi_t^j$ for the players $j\in\{i+1,\dots,N\}$ have been computed. 
\subsubsection{Players \texorpdfstring{$i<N$}{i<N} at a stage \texorpdfstring{$t<T$}{t<T}}
Suppose that we have constructed the KKT conditions for the $(i+1)$-th player at the $t$-th stage, we are then ready to construct the KKT conditions for player $i$ at the $t$-th stage. We introduce the variable $\zvec_t^i: = [\yvec_t^i,\zvec_t^{i+1}]$ with $\yvec_t^i:=[u_t^i,\psi_t^i, \eta_t^i,\lambda_t^i,\mu_t^i,\gamma_t^i,s_t^i]$. The KKT conditions of player $i$ at time $t$ is
\begin{equation}
    0= \Ktirho (\zvec_t^i):= \begin{bmatrix}
        \nabla_{u_t^i} L_t^i\\
        \nabla_{x_{t+1}} L_{t}^i \\
        \nabla_{u_t^j} L_t^i ,\forall j \in \mathbf{I}_{i+1}^N\\
        \nabla_{u_{t+1}^j} L_t^i ,\forall j \in \mathbf{I}_{1}^N\setminus \{i\}\\
        % x_{t+1} - f_t(x_t,u_t)\\
        h_t^i(x_t,u_t)\\
        g_t^i(x_t,u_t) - s_t^i\\
        \gamma_t^i\odot s_t^i - \rho \mathbf{1}\\
        \Ktnextirho (\zvec_t^{i+1})
    \end{bmatrix}.
\end{equation}
Building a first approximation to the above equation,
% we have
% \begin{equation}\label{eq:t,i-th player}
%     \begin{aligned}
%         \nabla_{\zvec_t^i} K_T^i \cdot \Delta \zvec_t^i + \nabla_{[x_t,u_t^{1:i-1}]}K_t^i\cdot \begin{bmatrix}
%             \Delta x_t \\ \Delta u_t^{1:i-1}
%         \end{bmatrix} + K_t^i(\zvec_t^i) = 0
%     \end{aligned}
% \end{equation}
we can obtain the \quasigradient{} $\nabla \localpitrhoi$ as in \eqref{eq:Delta y T i}, when it exists. % when \eqref{eq:t,i-th player} is feasible.
% We also remark here that when we construct the first player's policy, the second term of \eqref{eq:t,i-th player} is simplified as $\nabla_{x_t}K_t^1 \Delta x_t$, and the policy $\pi_t^i$ can be computed as in \eqref{eq:inverting KKT Jacobian to get strategy}, when there exists a solution to \eqref{eq:t,i-th player}. 

We observe that, by construction, the KKT conditions in \eqref{eq:KKT} is equivalent to $0=\Kzeroonerho(\zvec_0^1)$. To simplify notation, we define 
\begin{equation}
    \zvec: = \zvec_0^1, \ \ \Krho(\zvec): = \Kzeroonerho(\zvec).    
\end{equation}
The KKT conditions \eqref{eq:KKT} can be represented compactly as $0=\Krho(\zvec)$. To more effectively illustrate the construction process of KKT conditions described above, we have included detailed examples of the KKT conditions for two-player LQ games in Appendix~\ref{appendix:KKT} as a reference. 

\begin{algorithm}[t]
% \caption{PDIP-LQ Games}
\caption{Local Feedback Stackelberg Equilibrium via PDIP}
\begin{algorithmic}[1]\label{alg:pdip LQ}
\REQUIRE %$\{A_t,B_t,c_t\}_{t=0}^T$
$\{f_t\}_{t=0}^T$, $\{\ell_t^i,h_t^i,g_t^i\}_{t=0,i=1}^{T+1,N}$, initial homotopy parameter $\rho$, contraction rate $\sigma\in(0,1)$, parameters $\beta\in (0,1)$ and $\kappa\in(0,1)$, tolerance $\epsilon$, initial solution $\zvecrho^{(0)}:=[\xvec_\rho^{(0)}, \uvec_\rho^{(0)}, \lambdavec_\rho^{(0)}, \muvec_\rho^{(0)},  \gammavec_\rho^{(0)}, \etavec_\rho^{(0)}, \psivec_\rho^{(0)},\svec_\rho^{(0)}]$ with $\svec_\rho^{(0)}>0$ and $\gammavec_\rho^{(0)}>0$
\ENSURE policies $\{\localpitrhoi\}_{t=0,i=1}^{T,N}$, converged solution $\zvecrho$
\FOR{$k^{\mathrm{out}} = 1,2,\dots,k_{\mathrm{max}}^{\mathrm{out}}$}
% \FOR{$k^{\mathrm{in}} = 1,2,\dots,k_{\mathrm{max}}^{\mathrm{in}}$}
% \IF{failed to construct the KKT conditions $0 = \Krho(\zvecrho^{(k)})$}
%     \RETURN ``Local feedback Stackelberg strategy does not exist."    
% \ENDIF
\WHILE{the merit function $\|\Krho(\zvecrho^{(k)})\|_2>\epsilon$}
\STATE construct the first-order approximation of the KKT conditions \\$0 =  \nabla \Krho \cdot \Delta \zvecrho+\Krho(\zvecrho^{(k)}) $\label{alg_step:local LQ}
\STATE $\Delta \zvecrho \gets -\big(\nabla \Krho\big)^+\cdot \Krho(\zvecrho^{(k)})$
\STATE initialize the step size for line search, $\alpha \gets 1$
\WHILE{$\|\Krho(\zvecrho^{(k)} + \alpha \Delta \zvecrho)\|_2 > \kappa \|\Krho(\zvecrho^{(k)})\|_2  $ or $\hat{\zvec}_\rho:=(\zvecrho^{(k)} + \alpha \Delta \zvecrho)$ has a non-positive element in its sub-vector $[\hat{\svec}_\rho,\hat{\gammavec}_\rho]$}\label{alg_step:evaluate Krho}
\STATE $\alpha \gets \beta \cdot \alpha$
\ENDWHILE
\IF{$\alpha==0$}
\STATE claim \textbf{failure} to find a feedback Stackelberg equilibrium
\ENDIF
\STATE $\zvecrho^{(k+1)}\gets \zvecrho^{(k)}+\alpha \Delta \zvecrho$
% \IF{$\|\Krho(\zvecrho^{(k+1)})\|<\epsilon$}
% \STATE $\zvecrho\gets \zvecrho^{(k+1)}$ and terminate inner for loop.
% \ENDIF
\ENDWHILE
% \ENDFOR
\STATE  $\rho \gets \sigma \cdot \rho$
\ENDFOR
\STATE construct $\{\localpitrhoi\}_{t=0,i=1}^{T,N}$ as in \eqref{eq:T,N,construct policy} and record $\zvecrho \gets \zvecrho^{(k)}$. 
\RETURN $\{\localpitrhoi\}_{t=0,i=1}^{T,N}$, $\zvecrho$
\end{algorithmic}
\end{algorithm}

% \begin{remark}
%     \color{red}We want to show that $u_t$ is differentiable with respect to $x_t$. $u = K(x)^+ x$
% \end{remark}

% We denote the contraction parameter $\beta\in(0,1]$. We will do the line search over $\alpha\in[0,1]$ such that $\|K(\zvec + \alpha \Delta \zvec)\|<\beta \|K(\zvec)\|$. 

% The convergence of PDIP requires to assume that the $M(X,\Lambda)$ is an L-Lipschitz continuous function of $(X,\Lambda)$. However, the difficulty lies in how to bound the Lipschitz continuity of $\nabla \pi_t^i$. Can we show that $\|\nabla \pi_t^1(x) - \nabla \pi_t^1(y)\|\le L \|x,y\|$? Yes, this is a linear function and therefore its Jacobian is constant. 
\subsection{Primal-Dual Interior Point Algorithm and Convergence Analysis in Constrained LQ Games}
In this subsection, we propose the application of Newton's method to compute $\zvec^*=[\xvec^*, \uvec^*,\lambdavec^*, \muvec^*, \gammavec^*, \etavec^*,\allowbreak \psivec^*, \svec^*]$, ensuring $0 = \Krho(\zvec^*)$. This approach guarantees that the associated \quasipolicies{} form a set of local FSE policies, provided that we anneal the parameter $\rho$ to zero and the sufficient condition in Theorem~\ref{thm:sufficient condition} is satisfied. We formalize our method in Algorithm~\ref{alg:pdip LQ}.
% In what follows, we apply Newton's method to compute $\zvec^*=\{\xvec^*, \uvec^*,\lambdavec, \muvec, \gammavec, \etavec, \psivec,\allowbreak \svec\}$ such that KKT condition $0=K(\zvec)$ is met. As shown in \eqref{eq:T,N,construct policy}, a byproduct of this solution is a set of local policies $\{\pi_t^i\}_{t=0,i=1}^{T,N}$, which is defined around the trajectory $\{\xvec^*,\uvec^*\}$. We formalize the PDIP method in Algorithm~\ref{alg:pdip LQ}. 

In Algorithm~\ref{alg:pdip LQ}, we gradually decay the homotopy parameter $\rho$ to zero such that $\lim_{\rho\to 0}\zvecrho$ recovers an FSE solution. % to the problem \eqref{eq:problem formulation}. 
For each $\rho$, at the $k$-th iteration, we first construct the KKT conditions $0=\Krho(\zvec)$ along the trajectory $\zvecrho^{(k)}$. We compute the Newton update direction $\Delta \zvec :=- (\nabla \Krho)^+\cdot \Krho(\zvecrho^{(k)})$. Since we aim at finding a solution $\zvec^*$ to $0=\Krho(\zvec^*)$, a natural choice of merit function is $\|\Krho(\zvec)\|_2$. Given this choice of the merit function, we perform a line search to determine a step size $\alpha$ and update $\zvecrho^{(k+1)}=\zvecrho^{(k)} + \alpha \Delta \zvec$ until convergence. The converged solution is denoted as $\zvecrho^*$. Subsequently, we steadily decay $\rho$ and repeat these Newton update steps. We characterize how the magnitude of the KKT residual value $\|\Krho(\zvec)\|_2$ influences the convergence rate of Algorithm~\ref{alg:pdip LQ} when solving LQ games in the following result.

\begin{theorem}\label{thm:PDIP convergence}
    Under Assumption~\ref{assumption:feasible}, let $\F_z: = \{\zvec=[\xvec,\uvec,\lambdavec,\muvec,\gammavec, \etavec,\psivec,\svec ]:\gammavec \ge 0, \svec\ge 0\}$ be the solution set. We denote by $ \nabla\Krho(\zvec)$ and $\nablastar\Krho(\zvec)$ the Jacobians of the KKT conditions with and without considering \quasigradients, respectively. Suppose that $\nabla \Krho(\zvec)$ is invertible and there exist constants $D$ and $C$ such that
    \begin{subequations}
        \begin{align}
            & \|(\nabla \Krho(\zvec))^{-1}\|_2 \le D, &&\forall i\in\mathbf{I}_1^N,\forall \zvec \in \F_z, \label{eq:convergence pseudo inverse of K} \\
            & \|\nablastar \Krho(\zvec) - \nablastar \Krho(\tilde{\zvec})\|_2 \le C \|\zvec - \tilde{\zvec}\|_2, &&\forall i\in\mathbf{I}_1^N, \forall \zvec, \tilde{\zvec} \in \F_z. \label{eq:convergence Lipschitz} \vspace{-0.4em}
        \end{align}
    \end{subequations}
    %Suppose that there exists a sufficiently small $\rho>0$ such that $\sum_{t=0}^T\sum_{i=1}^N\|\nabla\pi_{t}^i -\tilde{\nabla} \pi_{t,\rho}^i \| \le \delta < \frac{1}{D}$, for all $\zvec\in\mathcal{F}_z$, where $\{\nabla \pi_t^i(x_t,u_t^{1:i-1}) |_{\zvec}\}_{t=0,i=1}^{T,N}$ exist. 
    % Denote by $\Delta \zvec:= -(\nabla_\zvec \Krho (\zvec))^{-1}(\nabla_{x_0}\Krho(\zvec)\cdot \Delta x_0 + \Krho(\zvec))$. 
    Let $\alphaupper\in[0,1]$ be the maximum feasible stepsize for all $\zvec \in \mathcal{F}_\zvec$, i.e., $\alphaupper:= \max\{\alpha\in[0,1]:\zvec,\zvec + \alpha\Delta\zvec\in\F_z\}$. Moreover, suppose $\|\nablastar\Krho(\zvec) - \nabla \Krho(\zvec)\|_2\le \delta$ for all $\zvec\in\mathcal{F}_\zvec$ and $D \cdot \delta < 1$. Then, for all $\zvec \in \mathcal{F}_z$, %where $\{\nabla \pi_t^i(x_t,u_t^{1:i-1})\}_{t=0,i=1}^{T,N}$ exist, 
    there exists $\alpha \in [0,\alphaupper]$ such that
    % \begin{equation}
    %     \|\Krho(\zvec + \alpha \Delta \zvec )\|_2 \le \Big(1-\frac{1}{2}\alpha_\zvec (1-D\delta_\zvec)\Big)\|\Krho(\zvec)\|_2
    % \end{equation}
    \begin{enumerate}
        \item if $\|\Krho(\zvec)\|_2> \frac{1- D \delta
        }{ D^2C \alphaupper}$, then $\|\Krho(\zvec + \alpha  \Delta \zvec )\|_2 \le \|\Krho(\zvec)\|_2 - \frac{(1-D\delta)^2}{2D^2C}$;
    %     % \begin{equation*}
    %     %     \|\Krho(\zvec + \alpha  \Delta \zvec )\|_2 \le \|\Krho(\zvec)\|_2 - \frac{1}{2D^2C}
    %     % \end{equation*}
    %     % \begin{equation*}
    %     %     \|\Krho(\zvec + \alpha  \Delta \zvec )\| \le \left\{\begin{aligned}& \| \Krho(\zvec) \| - \frac{1}{2D^2C} && \mathrm{if}\ \alpha_\zvec > \frac{1}{D^2C \|\Krho(\zvec)\|}, \\ & (1-\frac{1}{2}\alpha_\zvec) \|\Krho(\zvec)\|  && \mathrm{otherwise.}\end{aligned} \right.   
    %     % \end{equation*}
    %     % \begin{equation*}
    %     %     \|\Krho(\zvec + \alpha  \Delta \zvec )\| \le \| \Krho(\zvec) \| - \frac{\alpha_\zvec}{2D^2C} %(1-{ \delta D})^2
    %     % \end{equation*}
        \item if $\|\Krho(\zvec)\|_2 \le \frac{1- D \delta
        }{D^2C \alphaupper}$, then $            \|\Krho(\zvec + \alpha \Delta \zvec )\|_2 \le \big(1-\frac{1}{2}\alphaupper (1-D\delta)\big)%(1+{\delta D}) 
        \cdot \|\Krho(\zvec)\|_2$, and we have exponential convergence.
    \end{enumerate}
\end{theorem}\vspace{-0.5em}

\begin{proof}
    The proof can be found in the Appendix. 
\end{proof}\vspace{-0.1em}
Theorem~\ref{thm:PDIP convergence} suggests that, under certain conditions, the merit function $\|K_\rho(\zvec)\|_2$ decays to zero exponentially fast, and Algorithm~\ref{alg:pdip LQ} converges to a solution satisfying the KKT conditions considering the \quasigradients. The above analysis can be considered as an extension of the classical PDIP convergence proof in \cite{boyd2004convex} to constrained feedback Stackelberg games where we consider feedback interaction constraints $0=u_t^i- \tilde{\pi}_{t,\rho}^i(x_t,u_t^{1:i-1})$ and the \quasigradients. The condition \eqref{eq:convergence pseudo inverse of K} equates to establishing a lower bound for the smallest nonzero singular value of $\nabla \Krho(\zvec)$. Practically, this can be achieved by adding a minor cost regularization term to the KKT conditions \cite{chinchilla2023newton}. Moreover, the constant $C$ in \eqref{eq:convergence Lipschitz} depends on the maximum singular values of the Hessians of costs, the Jacobian of constraints, and linear dynamics, which are all constant matricies in LQ games and can therefore be upper bounded. %all the Hessians of the costs, dynamics, and constraints can be upper bounded. 

Given a $\rho>0$, a converged solution $\zvec_\rho^*$ renders $\Krho(\zvec_\rho^*) = 0$. Note that the KKT conditions $0=\Krho(\zvec_\rho^*)$ reduce to the one in Theorem \ref{thm:necessary condition} when $\rho$ decays to zero. %Substituting $\zvec_\rho$ and $\pi_\rho$ into the KKT condition in Theorem \ref{thm:necessary condition}, we have $ \|K^*(\zvec_\rho)\|\le\rho \sum_{t=0,i=1}^{T,N} n_{g,t}^i$. 
As $\rho$ approaches zero, the solution $\zvec_\rho^*$, when converged, recovers a solution to the KKT conditions in Theorem~\ref{thm:necessary condition}. When the sufficient conditions in Theorem~\ref{thm:sufficient condition} are also satisfied, the computed solution converges to a local feedback Stackelberg equilibrium. %the \quasipolicies{} $\{\localpitrhoi\}_{t=0,i=1}^{T,N}$ constitute a set of local FSE policies. 

\section{From LQ Games to Nonlinear Games}\label{sec:nonlinear games}
In this section, we extend our solution for LQ games to feedback Stackelberg games with nonlinear dynamics. Without loss of generality, each player could have non-quadratic costs. Coupled nonlinear equality and inequality constraints could also exist among players.

\subsection{Iteratively Approximating Nonlinear Games via LQ Games by Aligning Their KKT Conditions}In this subsection, we introduce a procedure which iteratively approximates the constrained nonlinear games using constrained LQ games, and computes approximate local feedback Stackelberg equilibria for the nonlinear games. These LQ game approximations are designed to ensure that the first-order approximations of their KKT conditions, expressed as $0=\nabla \Krho(\zvec)\cdot \Delta \zvec + \Krho(\zvec) $, align with those of the original nonlinear games, specifically considering the inclusion of \quasipolicies. Our approach differs from the existing iterative LQ game approximation techniques \cite{khan2023leadership,hu2023emergent} for FSE policies, which linearize the dynamics and quadraticize only the costs. In contrast, our method linearizes the dynamics but quadraticizes the Lagrangian. This enables us to utilize the convergence results for LQ games, as discussed in the previous section, to analyze the convergence properties of our method in nonlinear games. Consequently, our work provides the first iterative LQ game approximation approach that has provable convergence guarantees for constrained nonlinear feedback Stackelberg games.
% \end{remark}

%how we apply Newton's method to solve the KKT condition in \eqref{eq:KKT} for general nonlinear dynamic games. 
% When applying Newton's method to $0=\Krho(\zvec)$, we have, 
% \begin{equation}
% \begin{aligned}
% &    \nabla \Krho(\zvec) \Delta \zvec + \nabla_{x_0} \Krho(\zvec) \Delta x_0 + \Krho(\zvec) = 0, \\ &\Delta \zvec = -(\nabla \Krho(\zvec))^{+} ( \nabla_{x_0} \Krho(\zvec)\Delta x_0 + \Krho(\zvec)).
% \end{aligned}
% \end{equation}
% Observe that it is equivalent to solve an LQ game, where the costs, constraints and dynamics are defined by grouping corresponding terms in the KKT Jacobian $\nabla K(\zvec)$ and $K(\zvec)$. To be more specific, from the terms corresponding to the dynamics and constraints in $\nabla K(\zvec)$ and $K(\zvec)$, 
In what follows, we introduce local LQ game approximations of the original nonlinear game. Let $\zvec$ be a solution in the set $\mathcal{F}_z$. We first define the following linear approximation of the dynamics and constraints around $\zvec$, for all $ t\in\mathbf{I}_0^T,i\in\mathbf{I}_1^N$,
\begin{equation}\label{eq:linearization}
    \begin{aligned}
        &A_{t}:=\nabla_{x_t}f_t(x_{t}, u_{t}),    &&B_{t}^i:=\nabla_{u_t^i}f_t(x_{t},u_{t}), &&\hspace{-0.2cm}c_{t}:=f_t(x_{t}, u_{t}) - x_{t+1},\\
        &H_{x_t}^i: = \nabla_{x_t} h_t^i,\  H^i_{u_t^j}:=\nabla_{u_t^j}h_t^i,  &&G_{x_t}^i:=\nabla_{x_t} g_t^i,\ G^i_{u_t^j}:=\nabla_{u_t^j}g_t^i, && \forall j\in\mathbf{I}_{1}^N,\\
        &\bar{h}_{t}^i:=h_t^i(x_{t},  u_{t}), && \bar{g}_{t}^i: = g_t^i(x_{t},u_{t}),\\
        & H_{x_{T+1}}^i: = \nabla_{x_{T+1}} h_{T+1}^i, && G_{x_{T+1}}^i:=\nabla_{x_{T+1}} g_{T+1}^i,\\
        & \bar{h}_{T+1}^i:= h_{T+1}^i(x_{T+1}), && \bar{g}_{T+1}^i := g_{T+1}^i(x_{T+1}).
    \end{aligned}
\end{equation}
For each $i\in\mathbf{I}_1^N$ and $t\in\mathbf{I}_0^T$, we represent the second order terms and cost-related terms in the Lagrangian $\mathcal{L}_t^i$ as quadratic costs \eqref{eq:define lq game costs}, with parameters defined as follows, %The following quadratic cost approximation of the nonlinear game can also be derived from the terms corresponding to the players' costs in $\nabla K(\zvec)$, for all $t\in\mathbf{I}_0^T,i\in\mathbf{I}_1^N$,
\begin{equation}\label{eq:quadraticization}
\begin{aligned}
    Q_{t}^i:=&\nabla_{xx}^2\ell_t^i+(\nabla^2_{xx}f_t)^\top \lambda_{t}^i-(\nabla_{xx}^2h_t^i)^\top \mu_{t}^i - (\nabla_{xx}^2g_t^i)^\top \gamma_{t}^i,\\
    S_{t}^i:=&\nabla_{ux}^2\ell_t^i + (\nabla_{ux}^2f_t)^\top\lambda_{t}^i- (\nabla_{ux}^2h_t^i)^\top \mu_{t}^i - (\nabla_{ux}^2g_t^i)^\top \gamma_{t}^i,\\
    R_{t}^i:=& \nabla_{uu}^2\ell_t^i + (\nabla_{uu}^2 f_t)^\top \lambda_{t}^i - (\nabla_{uu}^2 h_t^i)^\top \mu_{t}^i- (\nabla_{uu}^2g_t^i)^\top \gamma_{t}^i,\\
    % q_{t,k}^i:=& \nabla_x \ell_t^i\\
    % r_{t,k}^i:=& \nabla_u \ell_t^i\\
    Q_{T+1}^i:=& \nabla_{xx}^2\ell_{T+1}^i - (\nabla_{xx}^2 h_{T+1}^i)^\top \mu_{T+1}^i - (\nabla_{xx}^2g_{T+1}^i)^\top \gamma_{T+1}^i,\\ 
    q_{t}^i:=& \nabla_x \ell_t^i, \hspace{1cm}
    r_{t}^i:= \nabla_u \ell_t^i  
,\hspace{1cm} q_{T+1}^i:= \nabla_x\ell_{T+1}^i.
\end{aligned}
\end{equation}
% and at the terminal time
% \begin{equation}\label{eq:terminal quadraticization}
%     \begin{aligned}
%         Q_{T+1,k}^i:=& \nabla_{xx}^2\ell_{T+1}^i - (\nabla_{xx}^2 h_{T+1}^i)^\top \mu_{T+1,k}^i - (\nabla_{xx}^2g_t^i)^\top \gamma_{T+1,k}^i\\
%         q_{T+1,k}^i:=& \nabla_x\ell_{T+1}^i
%     \end{aligned}
% \end{equation}
We can modify Algorithm~\ref{alg:pdip LQ} to address nonlinear games by applying an LQ game approximation around the solution $\zvecrho^{(k)}$ in step \ref{alg_step:local LQ} of Algorithm~\ref{alg:pdip LQ} and formulate the resulting approximate KKT conditions $0=\Krho(\zvecrho^{(k)})$ defined with terms in \eqref{eq:linearization} and \eqref{eq:quadraticization}. Furthermore, this LQ game approximation is reiterated around $\zvecrho^{(k)}+\alpha \Delta \zvecrho$ in step \ref{alg_step:evaluate Krho}, when we evaluate the merit function $\|\Krho(\zvecrho^{(k)}+\alpha \Delta \zvecrho)\|_2$ during line search.%{\color{red}Moreover, we also do this LQ game approximation around $\zvecrho^{(k)}$ when we evaluate $\|\Krho(\zvecrho^{(k)})\|_2$ in step \ref{alg_step:evaluate Krho} of Algorithm~\ref{alg:pdip LQ}. }%Then, we compute a Newton update direction $\Delta \zvec$. We employ a line search to find an appropriate stepsize $\alpha \in[0,1]$ and update $\zvec' = \zvec + \alpha \Delta \zvec$. We reconstruct a local LQ approximation for this new solution $\zvec'$ and repeat the above line search steps until convergence under the given parameter $\rho>0$. We gradually decay $\rho>0$ to zero and repeat the above steps until convergence.
% We first obtain an linear quadratic game approximation of the nonlinear game, and then compute the Newton update direction. We update the primal variable and dual variable via line search, same as in Algorithm~\ref{alg:pdip LQ}. 
 % until convergence. % We summarize this iterative approximate solution to nonlinear games in Algorithm~\ref{alg:nonlinear}.

\subsection{Quasi-Policies Approximation Error and Exponential Convergence Analysis in Nonlinear Games}In the above solution procedure, we approximate the ground truth nonlinear policies of nonlinear games by \quasipolicies. However, different from LQ games, the ground truth feedback Stackelberg policies for nonlinear games could have nonzero high-order policy gradients. Thus, it is worthwhile to characterize the error caused by the \quasigradients. Essentially, there are two error sources. The first type of error is due to the fact that we have neglected high-order policy gradients when evaluating the KKT Jacobian $\nabla \Ktirho(\zvec)$, and the second form of error is how these changes propagate into the expression of KKT conditions $0=\Ktirho(\zvec)$ for earlier players and stages. Suppose those two error sources could be upper bounded; then, we can characterize their impact on the policy gradients error in the following proposition.
\begin{proposition}\label{prop:Quasi policy}
    Under Assumption \ref{assumption:feasible}, let $\zvec$ and $\zvectilde$ be two elements in the solution set $\mathcal{F}_z$. We denote by $\{\pitrhoi\}_{t=0,i=1}^{T,N}$ a set of policies around $\zvec$ and $\{\localpitrhoi\}_{t=0,i=1}^{T,N}$ a set of \quasipolicies{} around $\zvectilde$, respectively. We denote by $\{\Ktirho(\zvec)\}_{t=0,i=1}^{T,N}$ and $\{\Ktirhostar(\zvec)\}_{t=0,i=1}^{T,N}$ the KKT conditions with and without \quasipolicies, respectively. %and $\{\nabla \Ktirho  (\zvec_\rho)\}_{t=0,i=1}^{T,N}$ the KKT conditions and the KKT Jacobians defined around $\zvecrho$, respectively, considering the quasi policy approximation. Similarly, let $\{\Ktirhostar(\zvec)\}_{t=0,i=1}^{T,N}$ and $\{\nabla \Ktirhostar(\zvec)\}_{t=0,i=1}^{T,N}$ be the KKT conditions and the KKT Jacobians defined around $\zvec$, respectively, without considering the quasi policy approximation. 
    Let $i\le N$ and $t\le T$.
    Suppose that the Jacobian matrices $\nabla K_{t,\rho}^i(\zvectilde)$, $\nabla \Ktirhostar(\zvectilde)$ and $\nabla \Ktirhostar(\zvec)$ are invertible. Let $\epsilon_{\zvec,\zvectilde}>0$ be an upper error bound such that 
    \begin{equation}\label{eq:epsilon definition policy gradient}
        \begin{aligned}
            \max \Big\{&\|\nabla \Ktirho(\zvectilde) - \nabla \Ktirhostar(\zvectilde)\|_2, &&\|\nabla \Ktirhostar(\zvectilde) - \nabla \Ktirhostar(\zvec)\|_2,\\
            & \|\Ktirho(\zvectilde) - \Ktirho(\zvec)\|_2  ,  &&\|\Ktirho(\zvec) - \Ktirhostar(\zvec)\|_2 \Big\}\le \epsilon_{\zvec,\zvectilde}.
        \end{aligned}
    \end{equation}
    Then, the error between the \quasigradient{} and the policy gradient can be bounded as follows,
    \begin{equation}\label{eq:linear policy approximation error}
        \begin{aligned}
            \|\nabla   \tilde{\pi}_{t,\rho}^i&(\zvectilde) - \nabla \pi_{t,\rho}^i (\zvec) \|_2  \le \epsilon_{\zvec,\zvectilde}\cdot \Big(2\|\nabla \Ktirhostar (\zvec)^{-1}\|_2  + \\&\big(\|\nabla \Ktirho(\zvectilde)^{-1}\|_2+ \|\nabla \Ktirhostar(\zvec)^{-1}\|_2 \big) \cdot\|\nabla \Ktirhostar(\zvectilde)^{-1}\|_2\|\Ktirho(\zvectilde)\|_2\Big).
        \end{aligned}
    \end{equation}
    % where $C$ is a constant independent to $\zvec$ and $\tilde{\zvec}$.
\end{proposition}
\begin{proof}
    The proof can be found in Appendix. 
\end{proof}

Proposition~\ref{prop:Quasi policy} suggests that the error introduced by the \quasigradients{} is proportional to $\epsilon_{\zvec,\zvectilde}$, as described in \eqref{eq:linear policy approximation error}. However, it is challenging to obtain an analytical bound $\epsilon_{\zvec,\zvectilde}$ because the evaluation of $\Ktirhostar(\zvec)$ and $\nabla \Ktirhostar(\zvec)$ requires computing the high-order policy gradients. 
%This difficulty is further compounded by the need to consider the cumulative error from other players at later stages $t\in\mathbf{I}_t^T$, as shown in the definitions of $\epsilon_{\zvec, \zvectilde}$ and the KKT conditions of player $i$ at time $t$. % and $K_{t,\rho}^i(\zvec)$. %The difference in policy gradients would lead to different trajectories. 
The above analysis only provides a partial analysis for the policy gradient error introduced by the \quasigradients. In principle, it is possible that the \quasigradients{} could lead to a different feedback Stackelberg policy from the ground truth feedback Stackelberg policy. %{\color{red}The quasi-policy is a local linear approximation of the ground truth nonlinear feedback Stackelberg policy. It is only approximately optimal for future stages if a state perturbation occurs at an intermediate stage when the ground truth feedback Stackelberg policy is nonlinear.} 
However, it is intractable to compute high-order policy gradients when we have a long horizon game. In general, the quasi-policy is a local linear approximation to the ground truth nonlinear feedback Stackelberg policy, and when a state perturbation occurs at time $t$, such policies are only approximately optimal for the resulting sub-game. We believe that the local feedback Stackelberg \quasi{} is the closest computationally tractable approximation possible when we consider the first-order policy approximation techniques for long-horizon feedback Stackelberg games.

% Assuming the approximation error of the \quasigradients{} could be upper bounded, 
Furthermore, we can leverage the sufficient condition of the local FSE and the convergence analysis in Theorem~\ref{thm:PDIP convergence} to show that we will converge to a local FSE of nonlinear games under certain conditions on the iterative LQ game approximations. %\vspace{-1em}
\begin{theorem}[Exponential Convergence in Nonlinear Games]\label{thm:convergence of nonlinear games}
    Suppose that there exist constants $(D,C,\delta,\alphaupper)$, as defined in Theorem~\ref{thm:PDIP convergence}, such that at each iteration $k$ of Algorithm~\ref{alg:pdip LQ}, the approximate LQ game defined in \eqref{eq:linearization} and \eqref{eq:quadraticization} satisfies the conditions of Theorem~\ref{thm:PDIP convergence}. Then, for each $\rho>0$ and a sufficiently large $k$, $\zvecrho^{(k)}$ converges exponentially fast to a solution $\zvecrho^*$, which renders $\|\Krho(\zvecrho^*)\|_2=0$. Moreover, if the limit $\zvec^*:=\lim_{\rho\to 0} \zvecrho^*$ exists and Theorem~\ref{thm:sufficient condition}, which provides a sufficient condition for local FSE trajectories, holds true at $\zvecrho^*$ for all $\rho>0$, then the converged solution $\zvec^*$ recovers a local FSE trajectory. %\quasipolicies{} $\lim_{\rho\to 0}\{\localpitrhoi\}_{t=0,i=1}^{T,N}$ are local FSE policies.
\end{theorem}
\begin{proof}
    The proof can be found in the Appendix.
\end{proof}

\section{Experiments}
In this section, we consider a two-player feedback Stackelberg game modeling highway driving\footnote{The code is available at \url{https://github.com/jamesjingqili/FeedbackStackelbergGames.jl.git}}, where two highway lanes merge into one and the planning horizon $T=20$. We associate with each player a 4-dimensional state vector $x_t^i=[p_{x,t}^i, p_{y,t}^i ,v_t^i, \theta_t^i]$, where $(p_{x,t}^i,p_{y,t}^i)$ represents the $(x,y)$ coordinate, $v_t^i$ denotes the velocity, and $\theta_t^i$ encodes the heading angle of player $i$ at time $t$. The joint state vector of the two players is denoted as $x_t = [x_t^1,x_t^2]$. Both players have nonlinear unicycle dynamics, $\forall t\in\mathbf{I}_0^T$, $\forall i\in\{1,2\}$,
\begin{equation}\label{eq:unicycle dynamics}
    \begin{aligned}
        &p_{x,t+1}^i  = p_{x,t}^i + \Delta t \cdot v_{t}^i\sin (\theta_t^i),
        &&p_{y,t+1}^i  = p_{y,t}^i + \Delta t \cdot v_{t}^i\cos(\theta_t^i), \\ 
        &v_{t+1}^i  = v_{t}^i + \Delta t \cdot a_t^i,
        &&\theta_{t+1}^i  = \theta_t^i + \Delta t \cdot \omega_t^i.
    \end{aligned}
\end{equation}
We consider the following cost functions, for all $t\in \mathbf{I}_0^T$,
\begin{equation}
        \ell_t^1(x_t,u_t) = 10(p_{x,t}^1-0.4)^2 + 6(v_t^1 - v_t^2)^2 + 2\|u_t^1\|_2^2, \ell_t^2(x_t,u_t) =  \|\theta_t^2\|_2^4 + 2\|u_t^2\|_2^2,
\end{equation}
and the terminal costs $\ell_{T+1}^1(x_{T+1}) = 10(p_{x,T+1}^1-0.4)^2+ 6(v_t^1 - v_t^2)^2$ and $ \ell_{T+1}^2(x_{T+1}) = \|\theta^2_{T+1}\|_2^4$. 
% \begin{equation}
%     \begin{aligned}
%         &\ell_t^1(x_t,u_t) = 10(p_{x,t}^1-\frac{1}{2})^2 + 2\|u_t^1\|_2^2, && \ell_t^2(x_t,u_t) =  4(v_t^1 - v_t^2)^2+\|\theta_t^2\|_2^4 + 2\|u_t^2\|_2^2, && \forall t\in\mathbf{I}_0^T,\\
%         % &\ell_t^2(x_t,u_t) =  \|\theta_t^2\|_2^4 + 2\|u_t^2\|_2^2,  && \forall t\in\mathbf{I}_0^T,\\
%         &\ell_{T+1}^1(x_{T+1}) = 6(p_{x,T+1}^1-\frac{1}{2})^2, && 
%         \ell_{T+1}^2(x_{T+1}) = %2(p_{x,T+1}^2 - \frac{1}{2})^2 + 
%         \|\theta_{T+1}\|_2^4.
%     \end{aligned}
% \end{equation}
% with terminal costs
% \begin{equation}
%     \begin{aligned}
%         \ell_{T+1}^1(x_{T+1}) &= 6(p_{x,T+1}^1-\frac{1}{2})^2 \\
%         \ell_{T+1}^2(x_{T+1}) & = 2(p_{x,T+1}^2 - \frac{1}{2})^2 + \|\theta_{T+1}\|_2^4
%     \end{aligned}
% \end{equation}
Note that we include a fourth-order cost term in player 2's cost at each stage to model its preference of small heading angle. We consider the following (nonconvex) constraints encoding collision avoidance, driving on the road, and control limits,
\begin{equation}
    \begin{aligned}
        &\sqrt{\|p_{x,t}^1 - p_{x,t}^2\|_2^2 + \|p_{y,t}^1 - p_{y,t}^2\|^2_2} - d_{\mathrm{safe}}\ge 0, && t\in\mathbf{I}_0^{T+1},\\
        &p_{x,t}^i - p_{l} \ge 0, \hspace{0.8cm} p_{r}(p_{y,t}^i, p_{x,t}^i) \ge 0,\hspace{0.8cm} \|u_t\|_{\infty} \le u_{\max}, &&  t\in \mathbf{I}_0^{T+1},i \in\{1,2\},
        % &p_{r}(p_{y,t}^i, p_{x,t}^i) \ge 0, && t\in\mathbf{I}_0^{T+1}, i\in\{1,2\},\\
        % & \|u_t\|_{\infty} \le u_{\max}, &&  t\in\mathbf{I}_0^T,
    \end{aligned}
\end{equation}
where we define $p_l\in\mathbb{R}$ to be the left road boundary and denote by $p_r(p_{x,t}^i,p_{y,t}^i)$ the distance between player $i$ and the right road boundary curve. 
We also consider the following equality constraints at the terminal time
\begin{equation}
    \begin{aligned}
        v_{T+1}^1 - v_{T+1}^2 = 0,\ \theta_{T+1}^1 = 0,
    \end{aligned}
\end{equation}
where the two players aim to reach a consensus on their speeds, with player 1 maintaining its heading angle pointing forwards. % Therefore, we must employ Algorithm~\ref{alg:nonlinear} to solve this constrained feedback Stackelberg game. }

\begin{figure}[t!]
\centering
    \begin{minipage}[t]{.24\textwidth}
        \centering
        \includegraphics[width=\textwidth,trim=10mm 10mm 1mm 10mm]{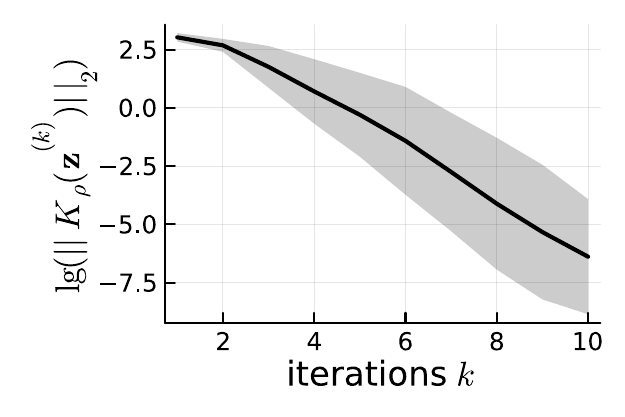}
        \subcaption{$\rho = 1$.}\vspace{-1.8em}\label{fig:loss 1}
    \end{minipage}
    \hfill
    \begin{minipage}[t]{.24\textwidth}
        \centering
        \includegraphics[width=\textwidth,trim=8mm 10mm 1mm 10mm]{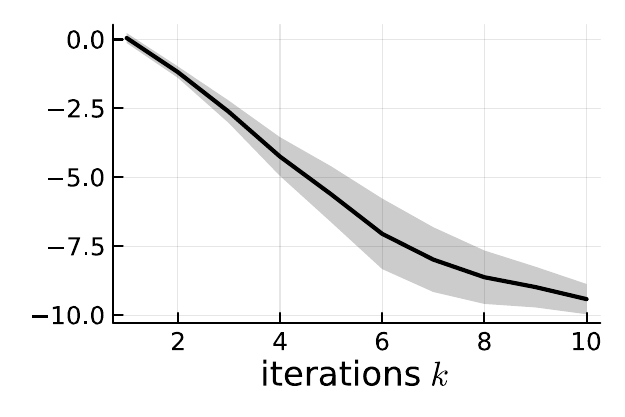}
        \subcaption{$\rho = 2^{-1}$.}\vspace{-1.8em}\label{fig:loss 2}
    \end{minipage}  
    \hfill
    \begin{minipage}[t]{.24\textwidth}
        \centering
        \includegraphics[width=\textwidth,trim=8mm 10mm 1mm 10mm]{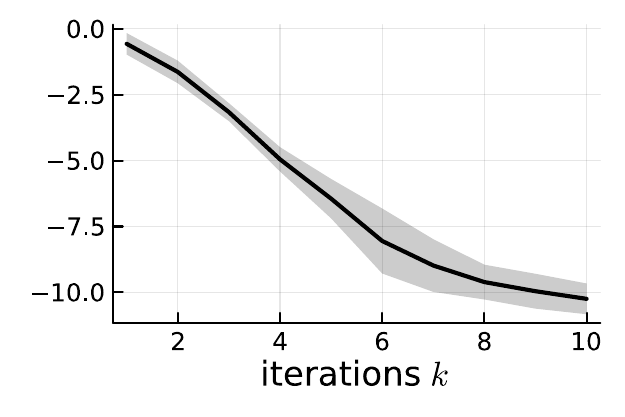}
        \subcaption{$\rho = 2^{-5}$.}\vspace{-1.8em}\label{fig:loss 3}
    \end{minipage}  
    \hfill
    \begin{minipage}[t]{.24\textwidth}
        \centering
        \includegraphics[width=\textwidth,trim=8mm 10mm 1mm 10mm]{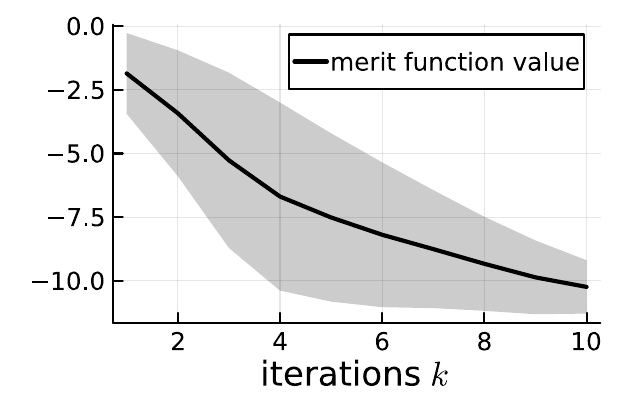}
        \subcaption{$\rho = 2^{-10}$.}\vspace{-1.8em}\label{fig:loss 4}
    \end{minipage}  
    \caption{Convergence of Algorithm~\ref{alg:pdip LQ} with iterative LQ game approximations under different values of the homotopy parameter $\rho$ from 10 sampled initial states. The solid curve and the shaded area denote the mean and the standard deviation of the logarithm of the merit function values, respectively. By gradually annealing $\rho$ to zero, the solution converges to a local FSE trajectory. Moreover, under each $\rho$, the plots above empirically support the linear convergence described in Theorem~\ref{thm:convergence of nonlinear games}. %The average computation time for each iteration is 0.83 seconds.
    }\label{fig:convergence} \vspace{-1em}
\end{figure}
\begin{figure}[t!]
    \begin{subfigure}[t]{.24\textwidth}
        \centering
        \includegraphics[width=\textwidth, trim = 2.5mm 10mm 0mm 0mm]{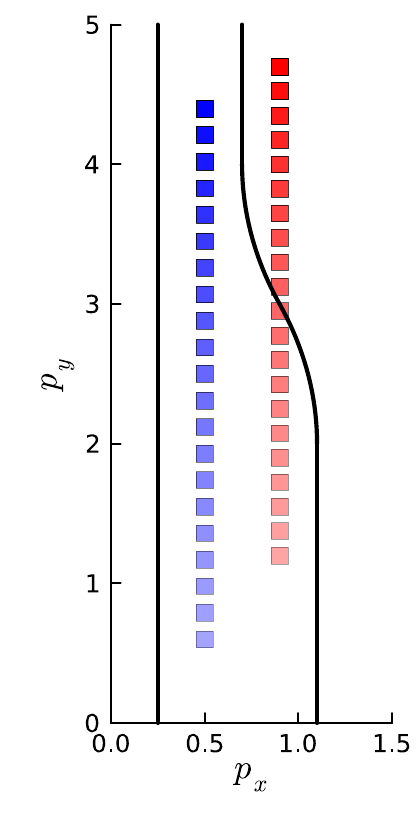}
        \caption{$\rho = 1$, $k=0$.}\label{subfig:1}\vspace{-1.8em}
    \end{subfigure}
    \hfill
    \begin{subfigure}[t]{.24\textwidth}
        \centering
        \includegraphics[width=0.963\textwidth, trim = 5mm 10mm 0mm 0mm]{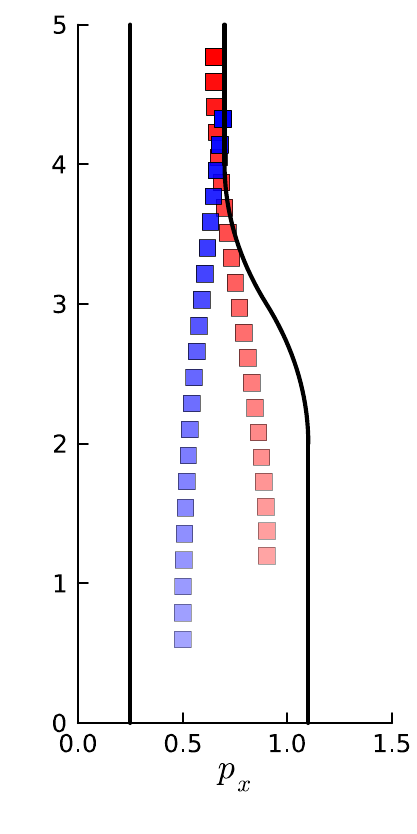}
        \caption{$\rho = 1$, $k=3$.}\label{subfig:2}\vspace{-1.8em}
    \end{subfigure}  
    \hfill
    \begin{subfigure}[t]{.24\textwidth}
        \centering
        \includegraphics[width=0.963\textwidth, trim = 5mm 10mm 0mm 0mm]{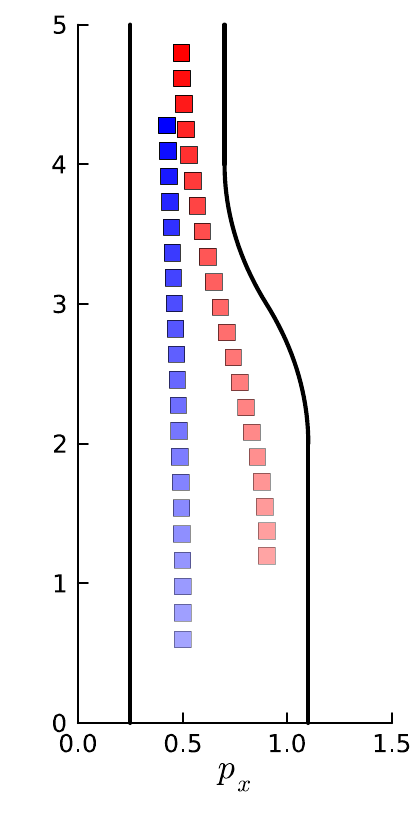}
        \caption{$\rho = 1$, $k=6$.}\label{subfig:3}\vspace{-1.8em}
    \end{subfigure}  
        \hfill
    \begin{subfigure}[t]{.24\textwidth}
        \centering
        \includegraphics[width=0.963\textwidth, , trim = 5mm 10mm 0mm 0mm]{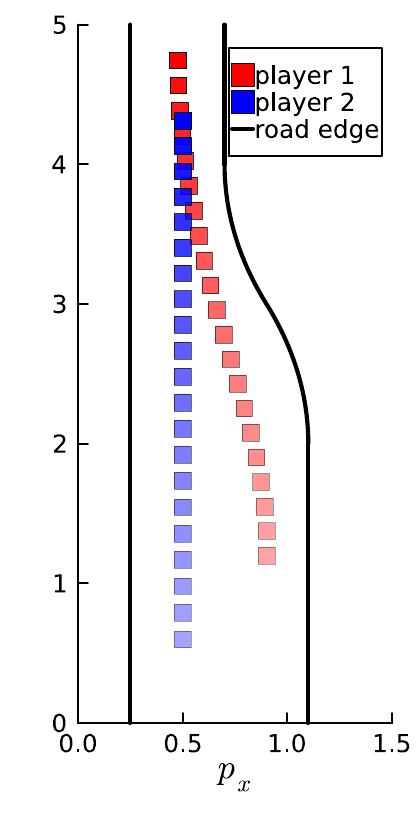}
        \caption{$\rho = 2^{-10}$, $k=10$.}\vspace{-1.8em}\label{subfig:convergence}
    \end{subfigure}
    \caption{Tolerance of an infeasible trajectory initialization and the converged trajectories of two players. In Figure \ref{subfig:1}, we plot the initial state trajectories of two players, where player 1's trajectory is infeasible because it violates the road boundary constraint. When $\rho=1$, we plot the state trajectories in the third and the sixth iterations in Figure~\ref{subfig:2} and Figure~\ref{subfig:3}, respectively. They become feasible at the sixth iteration. In Figure \ref{subfig:convergence}, we plot the converged solution, with $\rho = 2^{-10}$.  }\label{fig:tolerance of infeasible initialization}\vspace{-1em}
\end{figure}

The nominal initial states of two players are specified as $x_0^1= [0.9,1.2, 3.5, 0.0]$ and $x_0^2=[0.5,0.6,3.8,0.0]$, respectively. We randomly sample 10 initial states around $x_0=[x_0^1,x_0^2]$ under a uniform distribution within the range of $-0.1$ to $0.1$. From each sampled $\hat{x}_0$, we obtain an initial state trajectory $\xvec^{(0)}$ by simulating the nonlinear dynamics \eqref{eq:unicycle dynamics} with the initial controls $\uvec^{(0)} = \mathbf{0}$. Set the initial slack variables for the inequality constraints as $\svec^{(0)} = \mathbf{1}$, along with the corresponding Lagrange multipliers $\gammavec^{(0)} = \mathbf{1}$. We set all other Lagrange multipliers $\{\lambdavec^{(0)}, \muvec^{(0)}, \etavec^{(0)}, \psivec^{(0)}\}$ to zeros. Consequently, we have constructed an initial solution $\zvec^{(0)} = [\xvec^{(0)},\uvec^{(0)},\lambdavec^{(0)},\gammavec^{(0)} \muvec^{(0)}, \etavec^{(0)}, \psivec^{(0)}, \svec^{(0)}]$. We repeat this initialization trajectory defining process for different sampled~$\hat{x}_0.$ 

For each sampled initial state $\hat{x}_0$, we employ Algorithm~\ref{alg:pdip LQ} with iterative LQ game approximations to compute a local FSE trajectory. The convergence of our method under different sampled $\hat{x}_0$ is depicted in Figure \ref{fig:convergence}. For each $\rho$, the merit function value decreases as the iterations continue. Furthermore, since the cost functions are strongly convex with respect to each player's controls, Theorem~\ref{thm:sufficient condition} ensures that our converged solution constitutes a local FSE trajectory. Moreover, we show our method can tolerate infeasible initialization in Figure \ref{fig:tolerance of infeasible initialization}, where the right road boundary constraint is initially violated by initialization $\zvec^{(0)}$, and as the algorithm progresses, subsequent iterates $\zvec^{(k)}$ become feasible. %This suggests that as $\rho\to 0$, the resulted state trajectory converges to an approximate local feedback Stackelberg equilibrium.

% \begin{figure}[t]
%     \centering
%     \includegraphics[width = 0.2\textwidth]{Figures/highway_st_finalized.png}
%     \caption{Caption}
%     \label{fig:finalized highway}
% \end{figure}

\section{Conclusions}
In this paper, we considered general-sum feedback Stackelberg dynamic games with coupled constraints among $N$ players. We proposed a primal-dual interior point method to compute an approximate feedback Stackelberg equilibrium and the associated policies for all players. To the best of the authors' knowledge,  this represents the first attempt to compute approximate local feedback Stackelberg equilibria in both LQ games and nonlinear games under general coupled equality and inequality constraints, within continuous state and action spaces. We theoretically characterized the approximation error and the exponential convergence of our algorithm. Numerical experiments suggest that the proposed algorithm can tolerate infeasible initializations and efficiently converge to a feasible equilibrium solution. Future research should investigate the potential benefits of higher-order policy gradient approximations. Additionally, extending our approach to solve other types of equilibria in dynamic games is also a promising direction for future research. %stochastic feedback dynamic games %, which frequently arise in autonomous driving and human-robot interaction problems, 
%is also a promising future direction.

%We believe the KKT conditions proposed in this work can be further extended to Stackelberg-Nash equilibrium problems. 

% Future work should investigate 

\section*{Acknowledgments}
We would like to thank Professor Lillian Ratliff, Professor Forrest Laine, and Eli Brock for valuable discussions and comments. We used ChatGPT-4 \cite{openai2023chatgpt} to check the grammar in Sections 1 and 2.

\appendix
% \newpage
\section{Supplementary results} 
\begin{proof}[Proof of Theorem~\ref{thm:nested constrained optimization problem}]
% We approach the proof by showing that the optimal controls $\pi_t^{i*}$ defined in \eqref{eq:N-th player's strategy}, \eqref{eq:i-th player's strategy} and \eqref{eq:leader's strategy} are the same as 
% We approach the proof by showing that the set $\mathcal{U}_t^{i*}$ defined in \eqref{eq:N-th player's strategy}, \eqref{eq:i-th player's strategy} and \eqref{eq:leader's strategy}, for $i\in\mathbf{I}_1^N$, is equivalent to the set $\hat{\mathcal{U}}_t^{i*}$ in \eqref{eq:player's constrained problem}. 
% At the terminal time $t=T$, the $N$-th player's decision problem \eqref{eq:Z_t^N}
% Let $\{\pi_t^{i*}\}_{t=0,i=1}^{T,N}$ be a set of local feedback Stackelberg policies, as defined in Definition~\ref{def:local fse policy}. 
At the terminal time $t=T$, for ease of notation, we define $x_T = \bar{x}_T$, and $u_T^{1:i-1}=\bar{u}_T^{1:i-1}$. We observe that, for each player $i\in\mathbf{I}_1^N$, the equation \eqref{eq:player's constrained problem} can be rewritten as%the set $\mathcal{U}_T^{i*}$ defined in \eqref{eq:player's constrained problem} can be rewritten as\vspace{-0.5em}
\begin{align*}
    \tilde{u}_T^i \in\arg_{u_T^i} \min_{u_T^i} \Big\{\min_{\substack{u_{T}^{i+1:N} \\ x_{T+1}}}&  \ell_{T}^i(x_T,u_T) + V_{T+1}^i(x_{T+1})\Big\}\\
    \textrm{s.t. }
    &0=u_{T}^j - \pi_T^{j}(x_T, u_T^{1:j-1}) ,\  0=x_{T+1} - f_T(x_T,u_T)&& j\in  \mathbf{I}_{i+1}^N \\
    % & 0=x_{T+1} - f_T(x_T,u_T) \\
    % &0 = u_T^j - \pi_T^j(x_T, u_T^{1:j-1}) && j\in\mathbf{I}_1^N\setminus \{i\} \\
    &0 = h_T^i(x_T,u_T),\ 0\le g_T^i(x_T, u_T) \\
    & 0=h_{T+1}^i(x_{T+1}),\ 0\le g_{T+1}^i(x_{T+1})
\end{align*}
which implies $\tilde{u}_T^i\in \arg_{u_T^i} \min_{u_T^i} Z_T^i(\bar{x}_T,\bar{u}_T^{1:i-1},u_T^i)$. 
% \begin{align*}
%     \mathcal{U}_T^{i*}= \arg_{u_T^i} \min_{u_T^i} Z_T^i(x_T,u_T^{1:i-1},u_T^i)
%     % \mathcal{U}_T^{i*}= \arg_{u_T^i} \min_{\substack{
%     %         u_{T}^{i:N}  \\ 
%     %         x_{T+1}
%     %         }} & \ell_{T}^i(x_T, u^{1:N}_T)  +  \ell_{T+1}^i(x_{T+1})\\
%     %         \textrm{s.t. }
%     %         &0=u_{T}^j - \pi_T^{j}(x_T, u_T^{1:j-1}) && j\in  \mathbf{I}_{i+1}^N \\
%     %         &0=x_{T+1} - f_T(x_T,u_T^{1:N}) \\
%     %         &0 = h_T^i(x_T,u_T^{1:N})  ,0\le g_T^i(x_T,u_T^{1:N}) \\
%     %         &0=h_{T+1}^i(x_{T+1}),0\le g_{T+1}^i(x_{T+1})
% \end{align*}
% and therefore $\pi_T^i = \pi_T^{i*}$. 
Moreover, for all $t\in \mathbf{I}_0^{T-1}$ and $i\in \mathbf{I}_1^N$, for the ease of notation, we assume $x_t = \bar{x}_t$, and $u_t^{1:i-1}=\bar{u}_t^{1:i-1}$. We observe%\vspace{-1em}
\begin{align*}
    \tilde{u}_t^i\in\arg_{u_t^i} \min_{u_t^i} \Big\{\min_{\substack{u_{t}^{i+1:N} \\ u_{t+1:T}^{1:N} \\ x_{t+1:T+1}}}& \sum_{\tau = t}^T \ell_{\tau}^i(x_\tau,u_\tau) + \ell_{T+1}^i(x_{T+1})\Big\}\\
    \textrm{s.t. }
    & 0= u_t^j - \pi_t^j(x_t,u_t^{1:j-1}) && j\in \mathbf{I}_{i+1}^N \label{eq:}\\
    & 0=x_{\tau+1} - f_\tau(x_\tau,u_\tau) &&\tau\in\mathbf{I}_{t}^T \\
    &0 = u_\tau^j - \pi_\tau^j(x_\tau, u_\tau^{1:j-1}) && \tau\in \mathbf{I}_{t+1}^T,j\in\mathbf{I}_1^N\setminus \{i\} \\
    &0 = h_\tau^i(x_\tau,u_\tau) ,\ 0\le g_\tau^i(x_\tau, u_\tau) && \tau \in \mathbf{I}_{t}^T\\
    & 0=h_{T+1}^i(x_{T+1}),\ 0\le g_{T+1}^i(x_{T+1})
\end{align*}
% where we drop the first row when $i=N$. 
The above can be further rewritten as\vspace{-0.5em}
\begin{align*}
    \tilde{u}_t^i\in\arg_{u_t^i} \min_{u_t^i} \Big\{\min_{\substack{u_{t}^{i+1:N} \\ x_{t+1}}}&  \ell_{t}^i(x_t,u_t) + V_{t+1}^i(x_{t+1})\Big\}\\
    \textrm{s.t. }
    & 0= u_t^j - \pi_t^j(x_t,u_t^{1:j-1}),\  0=x_{t+1} - f_t(x_t,u_t) && j\in \mathbf{I}_{i+1}^N\\
    % & 0=u_{t+1}^j - \pi_{t+1}^j(x_{t+1}, u_{t+1}^{1:j-1}) && j \in \mathbf{I}_{1}^N\setminus \{i\}\\
    % & 0=x_{t+1} - f_t(x_t,u_t) \\
    &0 = h_t^i(x_t,u_t), \ 0\le g_t^i(x_t, u_t)
\end{align*}
% which can be rewritten as
% \begin{align*}
%     V_t^i(x_t) = \min_{u_t^i} \min_{\substack{u_t^{i+1:N} \\ u_{t+1:T}^{1:N} \\ x_{t+1:T+1}}}
% \end{align*}
It follows that $\tilde{u}_t^i \in \arg_{u_t^i} \min_{u_t^i} Z_t^i(\knownxt,\knownutoneiminusone,u_t^i)$. 
Therefore, the set of strategies $\{\pi_t^i\}_{t=0,i=1}^{T,N}$ constitutes a set of local feedback Stackelberg policies.
\end{proof}
\begin{proof}[Proof of Theorem~\ref{thm:necessary condition}]
    For a time $t\in\mathbf{I}_0^T$ and player $i\in\mathbf{I}_1^N$, we set the gradient of $\mathcal{L}_t^i$ with respect to $\{u_\tau^i\}_{\tau=t}^T$ and $\{x_\tau\}_{\tau=t+1}^{T+1}$ to be zero. This constitutes the first two rows of $\eqref{eq:KKT}$. In addition, a player $i<N$ considers the feedback interaction constraints $0=u_t^{j*} - \pi_t^{j}(x_t^*,u_1^{1:j-1*})$, for $j\in\mathbf{I}_{i+1}^N$. This constraint is implicitly ensured when we include player $j$'s KKT conditions into player $i$'s KKT conditions. Thus, we only need to ensure the gradient $\nabla_{u_t^j}\mathcal{L}_t^i$ to be zero, when synthesizing player $i$'s KKT conditions. This corresponds to the third row of \eqref{eq:KKT}. Moreover, at a time $t<T$, each player $i\in \mathbf{I}_1^N$ needs to account for the feedback reaction from other players in future steps. Again this constraint is implicitly ensured when we define player $j$'s KKT conditions. We only need to additionally set the gradient of $\mathcal{L}_t^i$ with respect to $u_{\tau}^j$ to be zero, where $\tau\in\mathbf{I}_{t+1}^T$ and $j\in\mathbf{I}_1^N\setminus\{i\}$. These correspond to the fourth row of \eqref{eq:KKT}. Finally, we include the dynamics constraints, equality and inequality constraints, and complementary slackness conditions in the last five rows of \eqref{eq:KKT}.
\end{proof}
\begin{proof}[Proof of Theorem~\ref{thm:sufficient condition}]
    We can check that the feasible set for the equality constraints of \eqref{eq:sufficient condition critical cone} is a superset of the critical cone of the problem \eqref{eq:KKT}. By Theorem 12.6 in \cite{nocedal1999numerical}, the solution$(\mathbf{x}^*,\mathbf{u}^*)$ constitutes a local feedback Stackelberg equilibrium trajectory.%, and $\{\pi_t^i\}_{t=0,i=1}^{T,N}$ is a set of local feedback Stackelberg policies. 
\end{proof}
\begin{proof}[Proof of Theorem~\ref{thm:PDIP convergence}]
    %Without considering \quasigradients{}, we can rewrite the first-order approximation of KKT conditions in a compact format as $    0=\nablastar \Krho(\zvec)\cdot \Delta \zvec + \Krho(\zvec) $. %Since there is no perturbation to the initial condition, we set $\Delta x_0=0$. 
    % By definition, $\Delta \zvec= -(\nabla_\zvec \Krho (\zvec))^{-1}(\nabla_{x_0}\Krho(\zvec)\cdot \Delta x_0 + \Krho(\zvec))$. Let $\Delta x_0=0$. 
    By fundamental theorem of calculus, we have $\Krho(\zvec+\alpha \Delta \zvec) = \Krho(\zvec) + \int_0^1 \nablastar \Krho(\zvec \allowbreak + \tau \alpha \Delta \zvec) \alpha \Delta \zvec d\tau$, and we have
    \begin{equation}\label{eq:main inequality}
        \begin{aligned}
            &\|\Krho(\zvec+\alpha \Delta \zvec )\|_2   = \left\|\Krho(\zvec) + \int_{0}^1 \nablastar \Krho(\zvec + \tau\alpha \Delta \zvec) \alpha \Delta \zvec d \tau \right\|_2 \\ 
            & \le \|\Krho(\zvec) + \alpha \nablastar \Krho(\zvec)
                \Delta \zvec  \|_2 + \left\|\int_0^1 (\nablastar \Krho(\zvec+\tau\alpha\Delta \zvec) -\nablastar \Krho(\zvec))\alpha \Delta \zvec
             d\tau \right\|_2
            % & =  \|\Krho(\zvec) - \alpha (\nabla_\zvec \Krho(\zvec)-\nabla_\zvec \Krho(\zvec) + \nabla_\zvec \Krhostar(\zvec))(\nabla_{\zvec}\Krhostar(\zvec))^{-1} \Delta \zvec\|_2 \\
            % & \le \frac{\alpha^2C}{2} \left\| (\nabla_\zvec \Krho(\zvec))^{-1} \Krho(\zvec) \right\|_2^2   
            % \le \frac{\alpha^2CD^2}{2} \|\Krho(\zvec)\|_2^2
        \end{aligned}
    \end{equation}
    % which yields
    % \begin{equation}\label{eq:bound on the integral}
    %     \|\Krho(\zvec + \alpha \Delta \zvec)\|_2 \le (1-\alpha)\|\Krho(\zvec)\|_2 + \frac{1}{2} \alpha^2 D^2 C \|\Krho(\zvec)\|_2^2
    % \end{equation}
Substituting $\Delta \zvec$ into $\|\Krho(\zvec) + \alpha \nablastar \Krho(\zvec)
                \Delta \zvec  \|_2$, we have 
\begin{equation}\label{eq:bound first K}
    \begin{aligned}
        \|\Krho(\zvec)& + \alpha \nablastar \Krho(\zvec)\Delta \zvec\|_2 =  \|\Krho(\zvec) - \alpha\nablastar \Krho(\zvec) (\nabla \Krho(\zvec))^{-1}\Krho(\zvec)\|_2\\
        % = & \|(I-\alpha \nabla_\zvec^* \Krho(\zvec) (\nabla_\zvec \Krho(\zvec))^{-1})\Krho(\zvec)\|_2\\
        \le & (1-\alpha) \|\Krho(\zvec)\|_2 + \alpha \|\nablastar \Krho(\zvec) - \nabla \Krho(\zvec)\|_2 \|(\nabla \Krho(\zvec))^{-1}\|_2 \|\Krho(\zvec)\|_2\\
        \le &(1-\alpha) \|\Krho(\zvec)\|_2 + \alpha \delta D \|\Krho(\zvec)\|_2=(1-\alpha(1-\delta D))\|\Krho(\zvec)\|_2
    \end{aligned}
\end{equation}
Combining \eqref{eq:bound first K} and \eqref{eq:main inequality}, we have
\begin{equation*}
\begin{aligned}
    \|&\Krho(\zvec + \alpha \Delta \zvec)\|_2\\
     % = (1-\alpha) \|\Krho(\zvec)\|_2 + \left\|\int_0^1 (\nabla_\zvec^* \Krho(\zvec+\tau\alpha\Delta \zvec) -\nabla_\zvec^* \Krho(\zvec))\alpha 
     %            \Delta \zvec  d\tau\right\|_2 \\
            & \le (1-\alpha(1-\delta D))\|\Krho(\zvec)\|_2 + \left\|\alpha 
                \Delta \zvec\right\|_2 \left\| \int_0^1 \|\nabla^* \Krho(\zvec+\tau\alpha \Delta \zvec) - \nabla^* \Krho(\zvec)\|d\tau \right\|_2\\
            % & \le (1-\alpha) \Krho(\zvec) + \| \alpha \Delta \zvec\|_2 \int_0^1 \tau C \|\alpha \Delta \zvec\|_2 d\tau \\
            & \le (1-\alpha(1-\delta D)) \|\Krho(\zvec)\|_2 + \frac{1}{2}\alpha^2  D^2  C \|\Krho(\zvec)\|_2^2
\end{aligned}
\end{equation*}
where the right hand side is minimized when $\alpha^* = \frac{1-D\delta}{D^2C\|\Krho (\zvec)\|_2}$. 
Suppose $\|\Krho(\zvec)\|_2>\frac{1-D\delta}{D^2C\alphaupper}$, then $\alphaupper > \frac{1-D\delta}{D^2 C \|\Krho(\zvec)\|_2}$ and we have $\|\Krho(\zvec + \alpha^*  \Delta \zvec )\|_2  \le  \|\Krho(\zvec)\|_2 -\frac{(1-D\delta)^2}{2D^2 C} $.

% \begin{equation}
%     \|\Krho(\zvec + \alpha^*  \Delta \zvec )\|_2  \le  \|\Krho(\zvec)\|_2 -\frac{1}{D^2 C} 
% \end{equation}
For the case $\|\Krho(\zvec)\|_2\le \frac{1- D \delta}{D^2C\alphaupper}$, let $\alpha := \alphaupper$. By $\alphaupper D^2C\|\Krho(\zvec)\|_2\le 1-D\delta$, we have $        \|\Krho(\zvec + \alphaupper \Delta \zvec)\|_2 \le (1-\frac{1}{2}\alphaupper(1-D\delta))\|\Krho(\zvec)\|_2
$.
% \begin{equation}
%     \begin{aligned}
%         \|\Krho(\zvec + \alpha_\zvec \Delta \zvec)\|_2 \le (1-\alpha_\zvec+\frac{1}{2}\alpha_\zvec)\|\Krho(\zvec)\|_2 = (1-\frac{1}{2}\alpha_\zvec)\|\Krho(\zvec)\|_2
%     \end{aligned}
% \end{equation}
\end{proof}

\begin{proof}[Proof of Proposition \ref{prop:Quasi policy}]
By definition, we have
\begin{equation*}
\begin{aligned}
    \|&\nabla   \tilde{\pi}_{t,\rho}^i(\zvectilde) - \nabla \pi_{t,\rho}^i (\zvec) \|_2 = \| \nabla \Ktirho (\zvectilde)^{-1} \Ktirho(\zvectilde) -\nabla \Ktirhostar(\zvec)^{-1} \Ktirhostar(\zvec)  \|_2\\ 
    & = \|  \nabla \Ktirho(\zvectilde)^{-1} \Ktirho(\zvectilde) - \nabla \Ktirhostar(\zvec)^{-1} \Ktirhostar(\zvec)\\ 
    & \hspace{1cm} + \nabla\Ktirhostar(\zvectilde)^{-1}\Ktirho(\zvectilde) - \nabla\Ktirhostar(\zvectilde)^{-1}\Ktirho(\zvectilde) \\
    & \hspace{1cm} + \nabla\Ktirhostar(\zvec)^{-1} \Ktirho(\zvectilde)- \nabla\Ktirhostar(\zvec)^{-1} \Ktirho(\zvectilde) \\ 
    & \hspace{1cm} + \nabla \Ktirhostar(\zvec)^{-1} \Ktirho(\zvec) -\nabla \Ktirhostar(\zvec)^{-1} \Ktirho(\zvec)\|_2\\
    & \le \Big(\|\nabla \Ktirho(\zvectilde)^{-1} - \nabla \Ktirhostar(\zvectilde)^{-1}\|_2 + \|\nabla \Ktirhostar(\zvectilde)^{-1} - \nabla \Ktirhostar(\zvec)^{-1}\|_2\Big) \|\Ktirho(\zvectilde)\|_2 \\
    & \hspace{1cm} + \Big(\|\Ktirho(\zvectilde) - \Ktirho(\zvec)\|_2  + \|\Ktirho(\zvec) - \Ktirhostar(\zvec)\|_2 \Big) \|\nabla \Ktirhostar (\zvec)^{-1}\|_2 \\
    & \le \epsilon_{\zvec,\zvectilde}\Big( \|\nabla \Ktirho(\zvectilde)^{-1}\|_2 + \|\nabla \Ktirhostar(\zvec)^{-1}\|_2 \Big)\|\nabla \Ktirhostar(\zvectilde)^{-1}\|_2\|\Ktirho(\zvectilde)\|_2\\
    & \hspace{1cm} + 2\epsilon_{\zvec,\zvectilde} \|\nabla \Ktirhostar (\zvec)^{-1}\|_2 
    % & \le \epsilon \big(\|\nabla \Ktirho(\zvec_\rho)^{-1}\|_2 + \|\nabla \Ktirhostar(\zvec)^{-1}\|_2  \big)\|\nabla \Ktirhostar(\zvec_\rho)^{-1}\|_2 \|\Ktirho(\zvec_\rho)\|_2 \\
    % &\hspace{1cm} +
\end{aligned}\vspace{-1em}
\end{equation*}
    %Observe that the KKT Jacobian matrices defined in \eqref{eq:simplified T,i-th player} and \eqref{eq:simplified t, i-th player} can be permuted into a symmetric matrix, where some off-diagonal matrix blocks represent the Jacobian of the policies. 
    where the last line follows by applying Lemma~\ref{lemma:symmetric matrix difference bound}.
\end{proof}

\begin{lemma}\label{lemma:symmetric matrix difference bound}
    Let $K$ and $\tilde{K}$ be two invertible matrices. Suppose $\|K - \tilde{K}\|_2\le \epsilon$, then we have $\|K^{-1} - \tilde{K}^{-1}\|_2 \le \epsilon\|K^{-1}\|_2 \cdot \|\tilde{K}^{-1}\|_2$.
    % \begin{equation}
    %     \|K^{-1} - \tilde{K}^{-1}\| \le 2\epsilon\|K^{-1}\| \cdot \|\tilde{K}^{-1}\|.
    % \end{equation}
\end{lemma}
\begin{proof}[Proof of Lemma~\ref{lemma:symmetric matrix difference bound}]
    Define $\bar{K} : = K - \tilde{K}$. Applying the Woodbury matrix equality, we have $\tilde{K}^{-1} = K^{-1} + K^{-1} \cdot \bar{K} \cdot \tilde{K}^{-1}$, and this implies $\|\tilde{K}^{-1} - K^{-1}\| \le  \epsilon \|K^{-1}\|\cdot \|\tilde{K}^{-1}\|$.
    % \begin{equation}
    %     \tilde{K}^{-1} = K^{-1} + K^{-1} \cdot D \cdot \tilde{K}^{-1}
    % \end{equation}
    % which implies that
    % \begin{equation}
    %     \|\tilde{K}^{-1} - K^{-1}\| \le 2 \epsilon \|K^{-1}\|\cdot \|\tilde{K}^{-1}\|
    % \end{equation}
\end{proof}

\begin{figure}[t!]
\centering
    \begin{minipage}[t]{.32\textwidth}
        \centering
        \includegraphics[width=\textwidth, trim = 20mm 20mm 10mm 30mm]{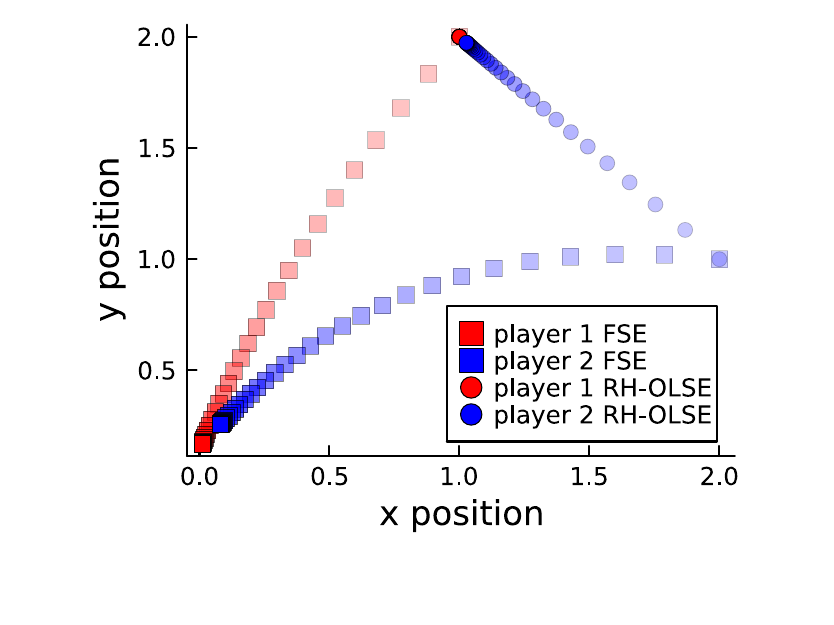}
        \subcaption{Trajectories}\label{fig:OL_FB_traj}\vspace{-2em}
    \end{minipage}
    \begin{minipage}[t]{.32\textwidth}
        \centering
        \includegraphics[width=\textwidth, trim = 10mm 12mm 10mm 0mm]{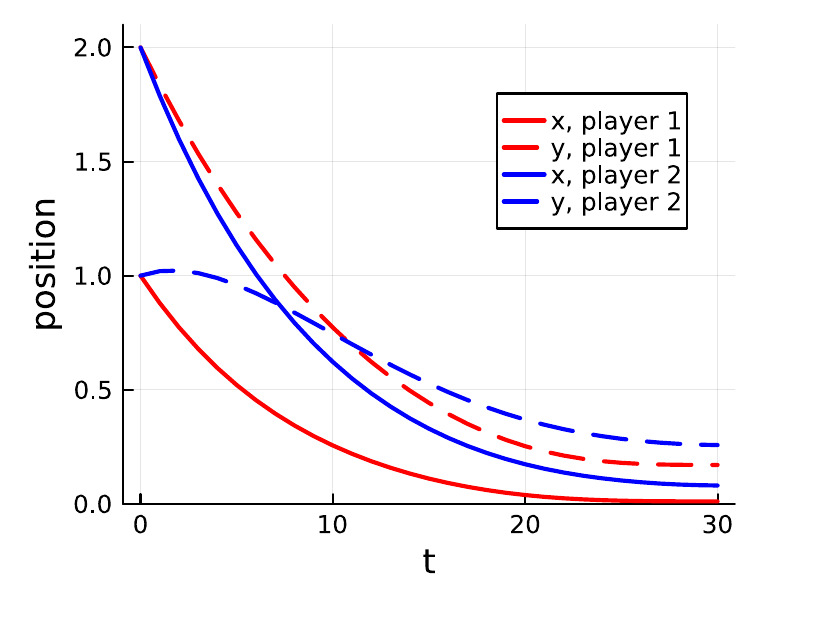}
        \subcaption{\centering Feedback Stackelberg policy}\label{fig:OL}\vspace{-2em}
    \end{minipage}
    \begin{minipage}[t]{.32\textwidth}
        \centering
        \includegraphics[width=\textwidth, trim = 10mm 12mm 10mm 0mm]{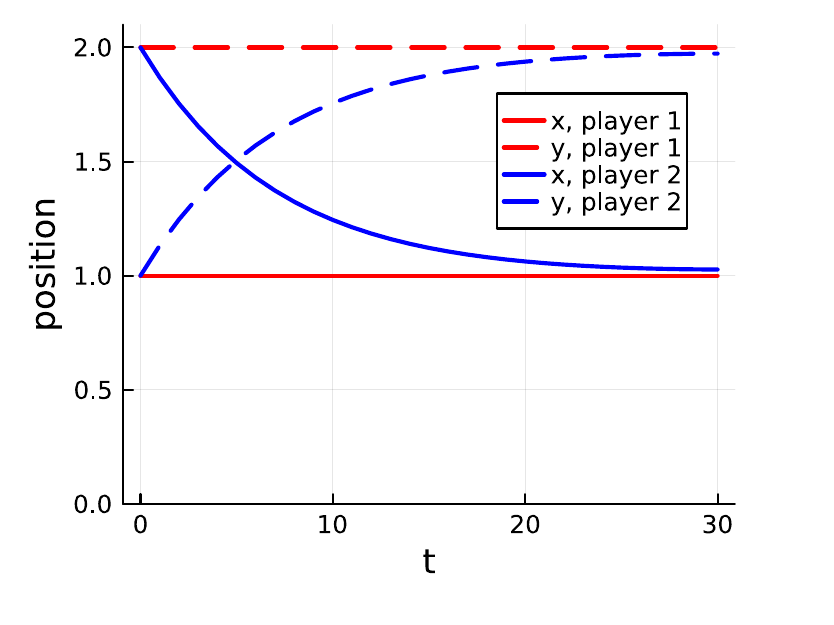}
        \subcaption{\centering Receding-horizon open-loop Stackelberg policy}\label{fig:FB}%\vspace{-2em}
    \end{minipage}  
    \caption{The trajectories under the receding horizon open-loop Stackelberg equilibrium (RH-OLSE) policy and those under the FSE policy are quite different, regardless of the initial conditions. %, except for the case where the initial position of player 1 is at the origin already. 
    % We can show that they are always different, except . 
    For example, in the above case, under the FSE policy, player 1 first moves towards the origin and then player 2 follows. However, under the RH-OLSE policy, player 1 always stays at its initial position, waiting for player 2 to~approach.}\label{fig:OL_FB}%\vspace{-2em}
\end{figure}
\begin{figure}[t!]
\centering
    \begin{minipage}[t]{.32\textwidth}
    % \begin{subfigure}{0.3\linewidth}
        \centering
        \includegraphics[width=\textwidth, trim = 5mm 75mm 5mm 70mm]{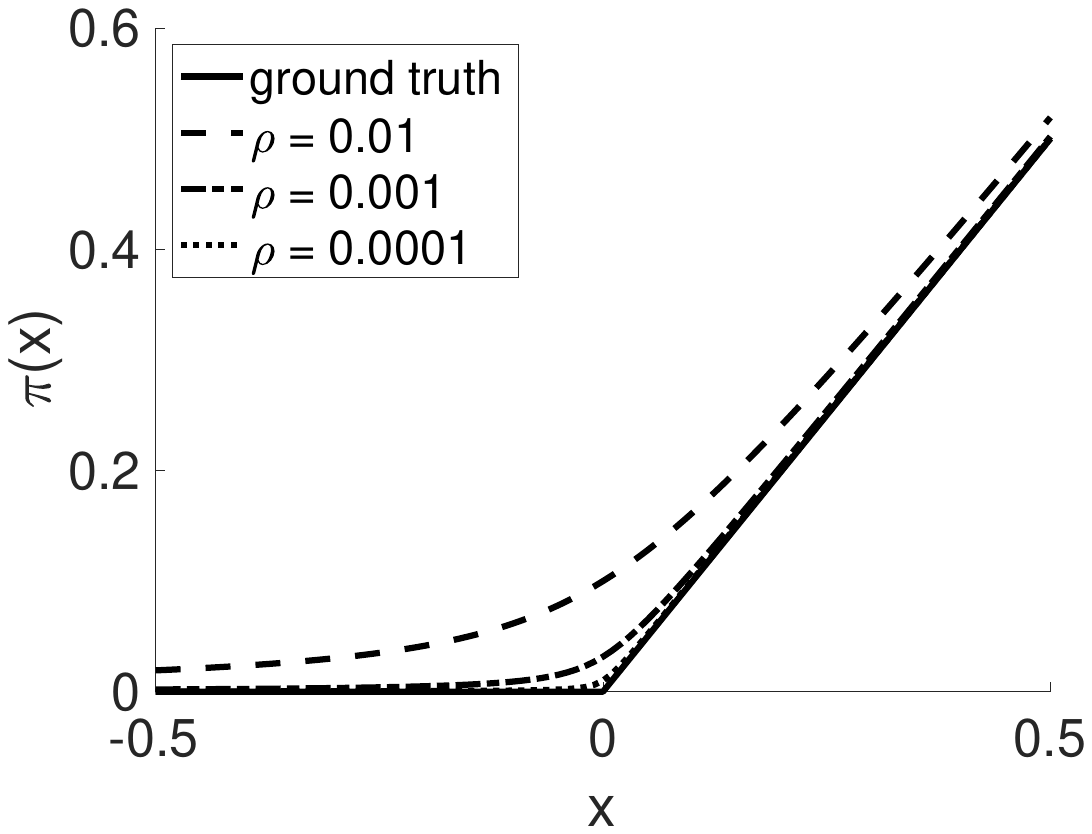}
        \subcaption{Zeroth order}\label{fig:policy gradient 0}%\vspace{-1.8em}
    % \end{subfigure}
    \end{minipage}
    % \hfill
    \begin{minipage}[t]{.32\textwidth}
    % \begin{subfigure}{0.3\linewidth}
        \centering
        \includegraphics[width=\textwidth, trim = 5mm 75mm 5mm 70mm]{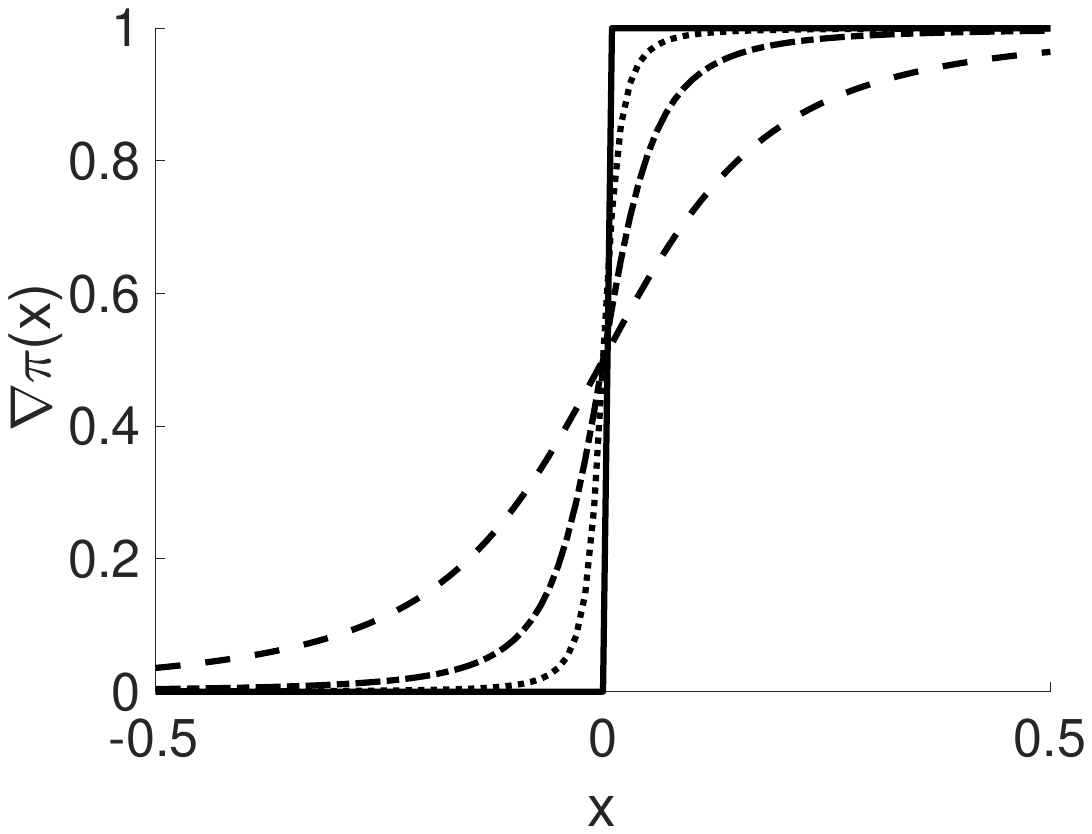}
        \subcaption{First order}\label{fig:policy gradient 1}\vspace{-1.8em}
    % \end{subfigure}
    \end{minipage}  
    % \hfill
    \begin{minipage}[t]{.32\textwidth}
        \centering
        \includegraphics[width=\textwidth, trim = 5mm 75mm 5mm 70mm]{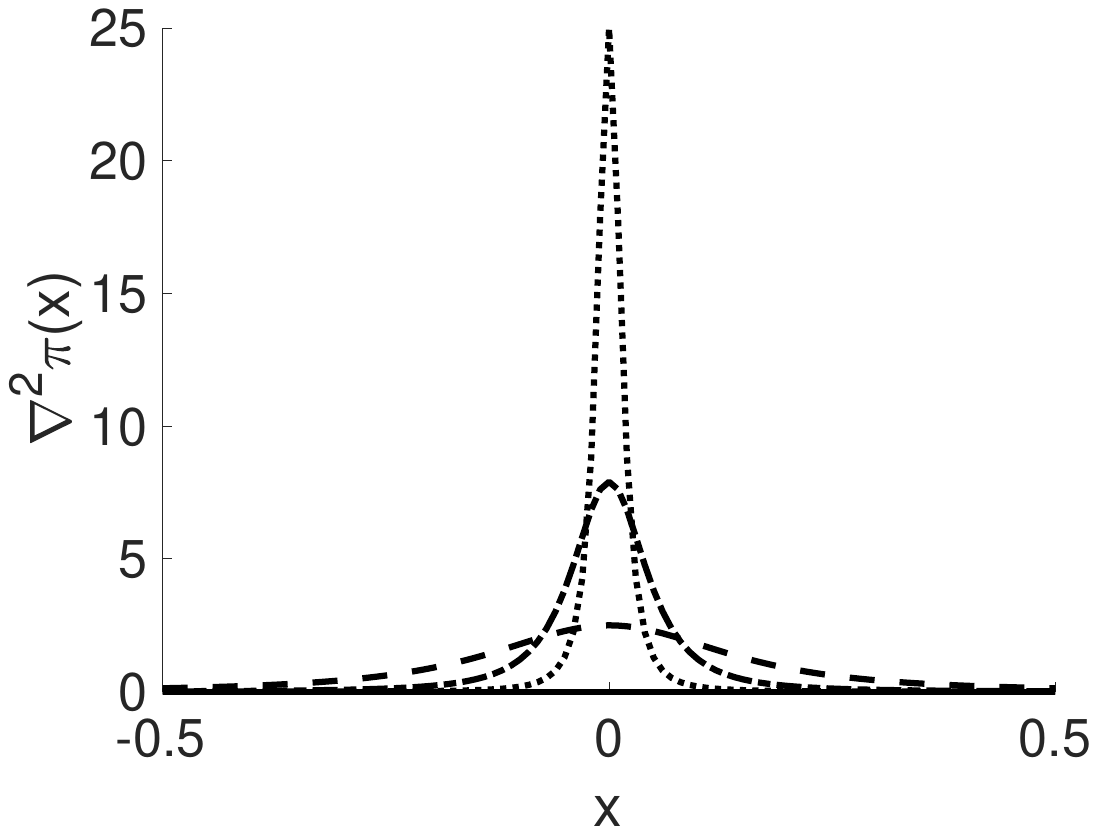}
        \subcaption{Second order}\label{fig:policy gradient 2}\vspace{-1.8em}
    \end{minipage}  
    % \hfill
    % \begin{minipage}[t]{.24\textwidth}
    %     \centering
    %     \includegraphics[width=\textwidth, trim = 5mm 40mm 5mm 40mm]{Figures/policy_gradient/third.pdf}
    %     \subcaption{the third order}\label{fig:policy gradient 3}
    % \end{minipage}  
    \caption{Visualization of the policy gradients of a constrained single-stage Linear Quadratic Regulator problem under different values of $\rho$. The cost is given by $(u_0-x_0)^2$. The dynamics is defined as $x_1 = x_0+u_0$. We consider a constraint $u_0\ge 0$. The ground truth piecewise linear policy is not differentiable at $x=0$. As $\rho\to 0$, the policy obtained from PDIP and its first-order gradient closely approximate the ground truth policy and its first-order gradient, for all nonzero $x$. As shown in Figure~\ref{fig:policy gradient 2}, the high-order gradient of the PDIP policy decays to zero as $\rho \to 0$, for all nonzero $x$. %{\color{red}This plot empirically supports the claim in Proposition~\ref{prop:quasi gradient zero}}
    }\label{fig:policy gradient}\vspace{-2em}
\end{figure}

\begin{proof}[Proof of Theorem~\ref{thm:convergence of nonlinear games}]
    Observe that the first order-approximation of the KKT conditions for the local LQ game approximations coincides with the one for nonlinear games. By Theorem~\ref{thm:PDIP convergence}, for each $\rho>0$, $\lim_{k\to \infty} \|\Krho(\zvecrho^{(k)})\|_2=0$, and we have exponential convergence when $k\ge \|\Krho(\zvecrho^{(0)})\|_2/ (\frac{1- D \delta
        }{ D^2C \alphaupper})$. Moreover, by Theorem~\ref{thm:sufficient condition}, the solution $\lim_{\rho\to 0}\zvecrho^*$ recovers a local FSE trajectory. %the local policies $\{\localpitrhoi\}_{t=0,i=1}^{T,N}$ are local FSE policies. It follows that $\lim_{\rho\to0}\{\localpitrhoi\}_{t=0,i=1}^{T,N}$ are local FSE policies as well. 
\end{proof}

% \begin{wrapfigure}{r}{0.31\textwidth}
%     \centering
%     \vspace{0em}
%     \includegraphics[width = 0.31\textwidth, trim=10mm 15mm 10mm 0mm]{Figures/FNE_FSE/fbst_fbne.pdf}
%     \vspace{-2em}
%     \caption{Players correctly exchange lanes under the FSE policy but %keep far away from each other 
%     fail to do so under the FNE policy due to safety concern.}
%     \vspace{-2.5em}
%     \label{fig:ST_NE}
% \end{wrapfigure}
% \end{figure}

% \subsection{Comparing the FNE with the FSE}\label{subsec:appendix comparing FSE and FNE}
% Consider a two-player lane exchanging problem\footnote{The code is available at \url{https://github.com/jamesjingqili/FeedbackStackelbergGames.jl.git}} with linear double integrator dynamics. %$x_{t+1} = A x_t + B^1 u_t^1 + B^2 u_t^2$. 
% Let $d_{\textrm{safe}}(x_t):=\frac{1}{2}((p_{x,t}^1- p_{x,t}^2)^2 + ((p_{y,t}^1 - p_{y,t}^2)^2)$ and $d_{\textrm{target}}(x_t):=(p_{x,t}^1- 1)^2 + (p_{y,t}^1-10)^2 + (p_{x,t}^2+1)^2 + (p_{y,t}^2 - 10)^2 $. Consider costs $\ell_t^1(x_t,u_t) = d_{\textrm{target}}(x_t)- d_{\textrm{safe}}(x_t) + (v_{x,t}^1)^2+ (v_{y,t}^1-1)^2   +4 \|u_t^1\|^2_2 $ and $\ell_t^2(x_t,u_t) = d_{\textrm{target}}(x_t) - d_{\textrm{safe}}(x_t) + (v_{x,t}^2)^2 + (v_{y,t}^2 - 1)^2   + 4 \|u_t^2\|_2^2$. Figure~\ref{fig:ST_NE} suggests that the FSE is a more appropriate equilibrium concept than the FNE when decision hierarchy exists. 
{
\subsection{The difference between Stackelberg equilibrium and Nash equilibrium in oligopoly games \cite{varian2014intermediate}}\label{subsec:appendix comparing FSE and FNE}
Consider a two-player Oligopoly game. We denote by the action $u^i$ the amount of production of player $i$. Consider three positive constants $C_1$, $C_2$, and $C_3$, where $C_1 > C_3$. The cost of each player $i\in\{1,2\}$ is modeled as 
\begin{equation}
     \ell^i (u^1, u^2) := -u^i\cdot (C_1 - C_2(u^1 + u^2) - C_3)
 \end{equation}
Both players aim to minimize their respective costs. We compute the Nash equilibrium by solving the KKT conditions
\begin{equation}\label{eq:KKT conditions of the oligopoly game}
    \begin{aligned}
        \frac{\partial}{\partial u^{1*}} \ell^1(u^{1*}, u^{2*}) & = -C_1 + 2 C_2u^{1*} + C_2 u^{2*} + C_3 \\
        \frac{\partial}{\partial u^{2*}} \ell^2(u^{1*}, u^{2*}) & = -C_1 + 2 C_2u^{2*} + C_2 u^{1*} + C_3 
    \end{aligned}
\end{equation}
The Nash equilibrium is $u_{Nash}^{1*} = u_{Nash}^{2*} = \frac{C_1-C_3}{3 C_2}$. In what follows, we compute the Stackelberg equilibrium. By solving player 2's KKT condition, as derived in the second line of \eqref{eq:KKT conditions of the oligopoly game}, the optimal reaction control of player 2 is given by
\begin{equation}
    u^{2*} = \frac{C_1-C_3}{2C_2} - \frac{1}{2} u^{1*}
\end{equation}
Substituting this into player $1$'s decision problem, we have 
\begin{equation}
\begin{aligned}
    \min_{u^1}\  & \ell^1(u^1,u^2) = -u^1\cdot (C_1 - C_2(u^1 + u^2) - C_3) \\
    \textrm{s.t. } & u^2 = \frac{C_1-C_3}{2C_2} - \frac{1}{2} u^{1}
\end{aligned}
\end{equation}
Solving the above problem, we have that the optimal action of player 1 is $u_{Stackelberg}^{1*}=\frac{C_1 - C_3}{2C_2}$, and the optimal action of player 2 is $u_{Stackelberg}^{2*} = \frac{C_1-C_3}{4C_2}$. The Stackelberg equilibrium can be arbitrarily different from the Nash equilibrium when the value $\frac{C_1 - C_3}{C_2}$ varies. 
}

\subsection{A counter example that the receding horizon open-loop Stackelberg equilibrium fails to approximate the FSE well}\label{subsec: appendix RHOL fails to approximate FSE}
% It is argued that feedback Stackelberg equilibrium could be same as feedback Nash equilibrium in dynamic games \cite{martin2021coincidence}. However, we provide counterexamples.
We consider Example 1 from \cite{li2023cost}. We show in Figure~\ref{fig:OL_FB_traj} that the receding-horizon open-loop Stackelberg policy could lead to a trajectory quite different from the one under FSE. Therefore, it is crucial to study the computation~of~FSE.

\subsection{The decay of high-order policy gradients when we apply PDIP to solve constrained LQ games%\quasigradients{}
}\label{subsec:appendix quasi policy gradient}
We validate the \quasi{} assumption in LQ games in Proposition~\ref{prop:quasi gradient zero}, and include a simplified example in Figure~\ref{fig:policy gradient}. \vspace{-.5em}

\begin{proposition}\label{prop:quasi gradient zero}
    Under the same assumptions of Theorem~\ref{thm:PDIP convergence}, let $\rho>0$ and denote by $\zvecrhostar$ a converged solution to an LQ game under Algorithm~\ref{alg:pdip LQ} with considering high-order policy gradients. Let $\{\pitrhoi\}_{t=0,i=1}^{T,N}$ be the converged policies. %defined around % \quasipolicies{} around 
    %$\zvecrhostar$, 
    %with considering high-order policy gradients. %, as defined in \eqref{eq:T,N,construct policy}. 
    Suppose that $\lim_{\rho\to0}\zvecrhostar$ exists and we denote it by $\zvec^*$. Moreover, suppose that the ground truth FSE policies $\{\pi_t^i\}_{t=0,i=1}^{T,N}$ are differentiable at $(\xvec^*, \uvec^*)$. Then, $\lim_{\rho\to 0} \| \nabla \pitrhoi - \nabla \pi_t^i\|_2 = 0$, and $\lim_{\rho\to 0}\|\nabla^j\pitrhoi\|_2=0,\forall i\in\mathbf{I}_1^N,t\in\mathbf{I}_0^T$, $j\ge2$.
\end{proposition}\vspace{-0.6em}
\begin{proof}
    At time $t=T$, there is no policy gradient term in the $N$-th player's KKT conditions. Recall that $\nabla \piTrhoN=-[(\nabla K_{T,\rho}^N)^{-1}]_{u_T^N} \nabla_{[x_T,u_T^{1:N-1}]}K_{T,\rho}^N$ and $\nabla \pi_T^N=-[(\nabla K_{T}^N)^{-1}]_{u_T^N} \nabla_{[x_T,u_T^{1:N-1}]}K_{T}^N$. Since $\lim_{\rho\to0}\|K_{T,\rho}^N(\zTrhoNstar) - K_T^N(\zTNstar)\|_2=0$, we have pointwise convergence $\lim_{\rho\to 0}\|\nabla \piTrhoN - \nabla \pi_T^N\|_2=0$ almost everywhere. We characterize those high-order \quasigradients{} of $\piTrhoN$ as follows. We denote the map from $\zTrhoNstar$ to the $j$-th order gradient of $\piTrhoN$ by an operator $\operator_T^{N,j}:\zTrhoNstar \to \nabla^j \piTrhoN$. Observe that $[(\nabla K_{T,\rho}^N)^{-1}]_{u_T^N}$ can be considered as the concatenation of a matrix inverse operator $\matrixinverseoperator: X\in\mathbb{R}^{n\times n}\to X^{-1}\in\mathbb{R}^{n\times n}$ and a linear operator $\linearoperator:\zTrhoNstar \to \nabla K_{T,\rho}^N$. Note that the matrix inverse is an infinitely differentiable operator when $X$ is invertible and $\nabla_{[x_T,u_T^{1:N-1}]}K_{T,\rho}^N$ is a constant matrix. Thus, by the chain rule \cite{rudin1976principles}, $\piTrhoN$ is infinitely differentiable, which also implies that $\nabla^j \piTrhoN$ is continuous, $\forall j\ge 1$. 
    
    Since $\nabla K_{T,\rho}^N(\zTrhoNstar)$ is invertible at $\zTrhoNstar$ and $\operator_T^{N,j}$, $\forall j\ge 1$, is a continuous operator, there exists a compact set $\compactdomainforgradientanalysis$ containing $\zTrhoNstar$ such that $\nabla K_{T,\rho}^N(\zvec_T^N)$ is invertible for all $\zvec_T^N\in \compactdomainforgradientanalysis$. %Since $\sup_{\zvec_T^N\in\compactdomainforgradientanalysis}\| \operator_T^{N,1}(\zvec_T^N) \|_\infty$ is bounded, we have $\sup_{\zvec_{T}^N\in\compactdomainforgradientanalysis} \|\operator_T^{N,2}(\zvec_{T}^N)\|_\infty$ is bounded. By induction, we can show that, for all $j\ge 1$, $\sup_{\zvec_{T}^N\in\compactdomainforgradientanalysis} \|\operator_T^{N,j}(\zvec_{T}^N)\|_\infty$ is bounded, 
    By the compactness of $\compactdomainforgradientanalysis$ and the continuity of $\operator_T^{N,j}$, we have that $\operator_T^{N,j}$ is a uniformly continuous operator on $\compactdomainforgradientanalysis$. By Theorem 2 in \cite{bartle1961preservation}, a uniformly continuous operator preserves the pointwise convergence. Thus, $\lim_{\rho\to 0}  \|\nabla^j\piTrhoN - \nabla^j \pi_T^N\|_2=0$. Since the ground truth policy $\pi_T^N$ is piecewise linear and the high-order gradients of $ \pi_T^N$ vanish% when $ \pi_T^N$ is differentiable at $(x_T,u_T^{1:N-1})$
    , we have $\lim_{\rho\to 0}\|\nabla^j \piTrhoN\|_2=0$, $\forall j>1$. 
    
    Subsequently, for player $i=N-1$, since $\lim_{\rho\to 0} \|\nabla \piTrhoN - \nabla \pi_T^N\|_2=0$, we have $\lim_{\rho\to 0} \|\nabla K_T^{i}(\zTrhoistar) - \nabla K_T^i(\zTi)\|_2=0$, which implies $\lim_{\rho\to 0}\|\nabla \piTrhoi - \nabla \pi_T^i\|_2 = 0$. A similar reasoning as above yields that $\lim_{\rho\to \infty} \|\nabla^j\piTrhoi - \nabla^j \pi_T^i\|_2 = 0$, $\forall j>1$. Moreover, we can show that for all players $i<N-1$, $\lim_{\rho\to 0}\|\nabla^j \piTrhoi - \nabla^j \pi_T^i\|_2=0$, $\forall j\ge 1$. We continue this backward induction proof of $\lim_{\rho\to0} \|\nabla^j \pitrhoi - \nabla^j \pi_t^i\|_2=0$, $\forall j\ge 1$, for prior stages backwards in players decision order until $t=0$ and $i=1$. 
\end{proof}\vspace{-1em}

\section{KKT conditions for two-player LQ games}\label{appendix:KKT}
The KKT conditions $0=K_{T,\rho}^2(\zvec_T^2)$  of player $2$ at time $T$ are \vspace{-0.7em}
\begin{equation*} %\label{eq: appendix, kkt, 2,T}%\small
    \left\{\begin{aligned}
        0=&\Sigma_{j=1}^2 R_T^{2,2,j} u_T^j + S_{T}^{2,2} x_T  + r_T^{2,2} + B_T^{2\top}\lambda_T^{2} - G_{u_T^2}^{2\top} \gamma_T^{2} -  H_{u_T^2}^{2\top} \mu_T^{2}   \\
        0=& Q_{T+1}^2 x_{T+1} + q_{T+1}^2 -\lambda_T^2  -  G_{x_{T+1}}^{2\top} \gamma_{T+1}^2 - H_{x_{T+1}}^{2\top} \mu_{T+1}^2  \\
        0=& x_{T+1} -A_T x_T -B_T^1 u_T^1 -B_T^2 u_T^2 - c_T \\
        0=& H_{u_T^2}^{2} u_T^2 + H_{x_{T}}^2 x_{T} + H_{u_T^1}^2 u_T^1 + \bar{h}_T^2\\
        0=& H_{x_{T+1}}^2 x_{T+1} + \bar{h}_{T+1}^2\\
        0=&\gamma_{T:T+1}^2 \odot s_{T:T+1}^2  - \rho \mathbf{1}\\
        % 0=& \gamma_{T+1}^2\odot s_{T+1}^2 - \rho\mathbf{1}\\
        0=& G_{u_T^2}^2 u_T^2  +G_{x_T}^2 x_T + G_{u_T^1}^2 u_T^1 + \bar{g}_{T}^2- s_T^2\\
        0=& G_{x_{T+1}}^2 x_{T+1} + \bar{g}_{T+1}^2 - s_{T+1}^2 
    \end{aligned}\right.
\end{equation*}
We construct the KKT conditions $0=K_{T,\rho}^1(\zvec_T^1)$ of player $1$ at time $T$: %\vspace{-0.5em}
% \begin{eqtion}
\begin{equation*} %\label{eq: appendix, kkt, 1,T}%\small
    \left\{\begin{aligned}
        0=& \Sigma_{j=1}^2 R_T^{1,1,j} u_T^j + S_T^{1,1} x_T + r_T^{1,1} + B_T^{1\top}\lambda_T^1 - G_{u_T^1}^{1\top} \gamma_T^1 - H_{u_T^1}^{1\top} \mu_T^1 + (\nabla_{u_T^1} \pi_{T,\rho}^2)^\top\psi_T^{1,2} \\
        0=& Q_{T+1}^1 x_{T+1} + q_{T+1}^1 - \lambda_T^1 - G_{x_{T+1}}^{1\top} \gamma_{T+1}^1 - H_{x_{T+1}}^{1\top} \mu_{T+1}^1 \\ 
        0=& \Sigma_{j=1}^2 R_T^{1,2,j}u_T^j  + S_T^{1,2}x_T  +r_T^{1,2}  + B_T^{2\top} \lambda_T^1  - G_{u_T^2}^{1\top} \gamma_{T}^1    - H_{u_T^2}^{1\top} \mu_{T}^1 - \psi_{T}^{1,2} \\
        0=& H_{u_T^1}^1 u_T^1 + H_{x_T}^1 x_T + H_{u_T^2}^1 u_T^2 + \bar{h}_{T}^1 \\ 
        0=& H_{x_{T+1}}^1 x_{T+1} + \bar{h}_{T+1}^1 \\ 
        0=& \gamma_{T:T+1}^1\odot s_{T:T+1}^1 - \rho \mathbf{1}\\
        % 0=& \gamma_{T+1}^1 \odot s_{T+1}^1 - \rho\mathbf{1}\\
        0=& G_{u_T^1}^{1} u_T^1 + G_{x_T}^1 x_T + G_{u_T^2}^1 u_T^2 + \bar{g}_{T}^1 - s_T^1\\
        0=& G_{x_{T+1}}^1 x_{T+1} + \bar{g}_{T+1}^1 - s_{T+1}^1 \\
        0=& K_{T,\rho}^2(\zvec_T^2)
    \end{aligned}\right.
\end{equation*}
% \begin{equation*}\small
%     \begin{aligned}
%         &\begin{bmatrix}
%             {R_T} &0 & -\hat{G}_{u_T}^\top& -\hat{H}_{u_T}^\top & \hat{B}_T^\top & {\Pi_T^{2\top}}&0&0&0&0 \\ 
%             0 & \hat{\gamma}_{T} & \hat{s}_{T} &0&0&0&0&0&0&0 \\
%             G_{u_T} & -I &0 &0&0 &0&0 &0 &0&0\\
%             H_{u_T}&0 & 0 & 0 & 0 & 0 & 0 &0&0&0 \\
%             -B_T&0&0 & 0 & 0 &  0 &  I_n  & 0 &0&0\\
%             0&0&0 & 0 & -I_{2n} & 0 & Q_{T+1} &0 & -\hat{G}_{x_{T+1}}^\top & -\hat{H}_{x_{T+1}}^\top   \\
%             0&0&0&0&0&0&0&\hat{\gamma}_{T+1} & \hat{s}_{T+1} &0\\
%             0&0&0 & 0 & 0 & 0& G_{x_{T+1}} & -I &0&0  \\
%             0&0 &0 &0 &0 &0  & H_{x_{T+1}} &0 &0&0 \\
%             {\tilde{R}_T^{12}}&0 &-\tilde{G}_{u_T}^{12\top}& -\tilde{H}_{u_T}^{12\top} & \tilde{B}_T^{2\top} & -I_{m_2} & 0 & 0 &0&0\end{bmatrix} 
%             \\ &
%             \cdot \begin{bmatrix}
%             \Delta u_T \\ \Delta s_T \\ \gamma_T \\ \mu_T \\ \lambda_T \\ \psi_T^1 \\ \Delta x_{T+1} \\ \Delta s_{T+1} \\ \gamma_{T+1} \\ \mu_{T+1} 
%         \end{bmatrix} + \begin{bmatrix}
%             {S_{x_T}} \\ 0 \\ G_{x_T}  \\ H_{x_T} \\ -A_T \\ 0 \\ 0 \\ 0 \\ 0 \\ S_{T}^{12}
%         \end{bmatrix}\Delta x_T + \begin{bmatrix}
%             r_T \\ %\hat{\gamma}_{T,k}s_{T,k}
%             -\rho\mathbf{1} \\ g_{T,k} - s_{T,k} \\ h_T \\ -c_T \\ q_{T+1} \\ %\hat{\gamma}_{T+1,k}s_{T+1,k}
%             -\rho\mathbf{1}\\ g_{T+1,k} - s_{T+1,k} \\ h_{T+1}  \\ r_{u_T^2}^1
%         \end{bmatrix} = 0
%     \end{aligned}
% \end{equation*}
% where $\Pi_T^{2} = [\nabla_{u_T^1} \pi_T^2,0]$
We construct the KKT conditions $0=K_{t,\rho}^2(\zvec_t^2)$ of player $2$ at time $t<T$: %\vspace{-0.5em}
\begin{equation*} %\label{eq: appendix, kkt, 2,t}%\small
    \left\{\begin{aligned}
        0= & \Sigma_{j=1}^2 R_t^{2,2,j} u_t^j + S_t^{2,2}x_t + r_t^{2,2} + B_t^{2\top} \lambda_t^2 - G_{u_t^2}^{2\top} \gamma_t^2 - H_{u_t^2}^{2\top} \mu_t^2 \\
        0=& Q_{t+1}^2 x_{t+1} + q_{t+1}^2 - \lambda_{t}^2 -G_{x_{t+1}}^{2\top} \gamma_{t+1}^2 - H_{x_{t+1}}^{2\top}\mu_{t+1}^2 \\ &- A_{t+1}^\top\lambda_{t+1}^{2} + \Sigma_{j=1}^2 S_{t+1}^{2,j}u_{t+1}^j + (\nabla_{x_{t+1}} \pi_{t+1,\rho}^1)^\top \eta_t^{2,1} \\
        0=&  \Sigma_{j=1}^2 R_{t+1}^{2,1,j} u_{t+1}^j + S_{t+1}^{2,1} x_{t+1}+ r_{t+1}^{2,1} + B_{t+1}^{1\top}\lambda_{t+1}^2 - G_{u_{t+1}^1}^{2\top} \gamma_{t+1}^2 - H_{u_{t+1}^1}^{2\top} \mu_{t+1}^2   -\eta_t^{2,1}\\
        0=& x_{t+1} - A_t x_t - B_t^{1} u_t^1- B_t^2 u_t^2 - c_t \\
        0=& H_{u_t^2}^2 u_t^2 + H_{x_t}^2 x_t+ H_{u_t^1}^2 u_t^1 +\bar{h}_t^2  \\
        0=& \gamma_t^2 \odot s_t^2 - \rho\mathbf{1} \\
        0=& G_{u_t^2}^2 u_t^2 + G_{x_t}^2 x_t + G_{u_t^1}^2 u_t^1 + \bar{g}_t^2 - s_t^2 \\
        0=& K_{t+1,\rho}^1(\zvec_{t+1}^1)
    \end{aligned}\right.
\end{equation*}
We construct the KKT conditions $0=K_{t,\rho}^1(\zvec_t^1)$ of player $1$ at time $t<T$: %\vspace{-0.5em}
\begin{equation*} %\label{eq: appendix, kkt, 1,t}%\small
    \left\{\begin{aligned}
        0=& \Sigma_{j=1}^2 R_t^{1,1,j} u_t^j  + S_t^{1,1}x_t + r_t^{1,1} + B_t^{1\top} \lambda_t^1 - G_{u_t^1}^{1\top} \gamma_t^1 - H_{u_t^1}^{1\top} \mu_t^1 + (\nabla_{u_t^1} \pi_{t,\rho}^2)^\top \psi_t^{1,2} \\
        0=& Q_{t+1}^1 x_{t+1} + q_{t+1}^1 - \lambda_t^1 - G_{x_{t+1}}^{1\top} \gamma_{t+1}^1 - H_{x_{t+1}}^{1\top} \mu_{t+1}^1 \\
        & - A_{t+1}^\top\lambda_{t+1}^1 + \Sigma_{j=1}^2 S_{t+1}^{1,j}u_{t+1}^j + (\nabla_{x_{t+1}} \pi_{t+1,\rho}^2)^\top \eta_t^{1,2} \\
        0=& \Sigma_{j=1}^2 R_t^{1,2,j}u_t^j + S_{t}^{1,2} x_t + r_t^{1,2} + B_t^{2\top} \lambda_t^1 - G_{u_t^2}^{1\top} \gamma_t^1 - H_{u_t^2}^{1\top} \mu_t^1 - \psi_t^{1,2} \\
        0=& \Sigma_{j=1}^2 R_{t+1}^{1,2,j}u_t^j + S_{t+1}^{1,2}x_{t+1} + r_{t+1}^{1,2} + B_{t+1}^{2\top}\lambda_{t+1}^1 -G_{u_{t+1}^2}^{1\top} \gamma_{t+1}^1 - H_{u_{t+1}^2}^{1\top} \mu_{t+1}^1 - \eta_t^{1,2} \\
        0=& H_{u_t^1}^1 u_t^1 + H_{x_t}^1 x_t + H_{u_t^2} u_t^2 + \bar{h}_t^1  \\ 
        0=& \gamma_t^1 \odot s_t^1 - \rho \mathbf{1} \\ 
        0=& G_{u_t^1}^{1} u_t^1 + G_{x_t}^1 x_t + G_{u_t^2}^{1} u_t^2 +\bar{g}_t^1 - s_t^1\\
        0=& K_{t,\rho}^2(\zvec_t^2)
    \end{aligned}\right.
\end{equation*}
We continue the above construction process until $i=1$ and $t = 0$. 

\bibliographystyle{siamplain}
\bibliography{references.bib}
\end{document}